\documentclass[12pt,twoside,a4paper]{article}
\usepackage[utf8]{inputenc}
\usepackage{amsmath, amssymb, mathtools, enumerate}
\usepackage{mathrsfs} 
\usepackage{amsfonts, amsthm, stmaryrd, fullpage}
\usepackage{authblk}
\usepackage{tikz-cd}
\usepackage[german, english]{babel}
\usepackage[nottoc,numbib]{tocbibind}

\usepackage{color}
\usepackage{hyperref}
\hypersetup{colorlinks=true,
linktoc=all,
citecolor=blue,
filecolor=black,
linkcolor=blue,
urlcolor=blue}

\newtheorem{thm}{Theorem}[section]
\newtheorem{lemma}[thm]{Lemma}
\newtheorem{cor}[thm]{Corollary}
\newtheorem{prop}[thm]{Proposition}

\theoremstyle{definition} 
\newtheorem{mydef}[thm]{Definition}
  
\newtheorem{example}[thm]{Example}

\theoremstyle{remark}
\newtheorem{rmk}[thm]{Remark}

\newcommand\Ban{{\rm Ban}}

\newcommand\Cond{{\rm Cond}}

\newcommand\eq{{\rm eq}}

\newcommand\Hom{{\rm Hom}}

\newcommand\im{{\rm im}}

\newcommand\Mod{{\rm Mod}}

\newcommand\Snm{{\rm Snm}}

\newcommand\spc{{\rm sp}}

\newcommand\Spa{{\rm Spa}}

\newcommand\Spec{{\rm Spec}} 

\newcommand\tf{{\rm tf}}

\newcommand\tor{{\rm tor}}

\begin{document}

\title{Banach modules, almost mathematics and condensed mathematics}
\author{Dimitri Dine}
\date{}
\maketitle

\begin{abstract}We study the relationship between almost mathematics, condensed mathematics and the categories of seminormed and Banach modules over a Banach ring $A$, with submetric (norm-decreasing) $A$-module homomorphisms for morphisms. If $A$ is a Banach ring with a norm-multiplicative topologically nilpotent unit $\varpi$ contained in the closed unit ball $A_{\leq1}$ such that $\varpi$ admits a compatible system of $p$-power roots $\varpi^{1/p^{n}}$ with \begin{equation*}\lVert\varpi^{1/p^{n}}\rVert=\lVert\varpi\rVert^{1/p^{n}}\end{equation*}for all $n$, we prove that the "almost closed unit ball" functor \begin{equation*}M\mapsto M_{\leq1}^{a}\end{equation*}is an equivalence between the category $\Ban_{A}^{\leq1}$ of Banach $A$-modules and submetric $A$-module maps and the category of $\varpi$-adically complete, $\varpi$-torsion-free almost $(A_{\leq1}, (\varpi^{1/p^{\infty}}))$-modules. We also obtain an analogous result for Banach algebras and almost algebras. The main novelty in our approach is that we show that the norm on the Banach module $M$ is completely determined by the corresponding almost $A_{\leq1}$-module $M_{\leq1}^{a}$, rather than being determined only up to equivalence.

We deduce from our results the existence of a natural fully faithful embedding \begin{equation*}\Ban_{A}^{\leq1}\hookrightarrow \Cond(A_{\leq1}, (\varpi^{1/p^{\infty}})_{A_{\leq1}})\end{equation*}of $\Ban_{A}^{\leq1}$ into the category of (static) condensed almost $(A_{\leq1}, (\varpi^{1/p^{\infty}}))$-modules in the sense of Mann, which factors through the full subcategory \begin{equation*}\Cond_{\blacksquare}(A_{\leq1}, (\varpi^{1/p^{\infty}})_{A_{\leq1}})\end{equation*}of solid almost $(A_{\leq1}, (\varpi^{1/p^{\infty}})_{A_{\leq1}})$-modules. If $A$ is a perfectoid Tate ring and the perfectoid space $\Spa(A, A^{\circ})$ is totally disconnected, we show that this embedding is also symmetric monoidal when $\Ban_{A}^{\leq1}$ is endowed with the complete tensor product and $\Cond_{\blacksquare}(A_{\leq1}, (\varpi^{1/p^{\infty}})_{A_{\leq1}})$ is endowed with (an almost analog of) the solid tensor product. \end{abstract}

\tableofcontents

\section{Introduction}

Almost mathematics first appeared in Faltings's work on $p$-adic Hodge theory and has subsequently become a powerful tool across $p$-adic arithmetic geometry as well as in mixed-characteristic commutative algebra. The purpose of this paper is to describe the relationship between almost mathematics and nonarchimedean analysis over a general Tate Banach ring, where the Tate Banach ring need not be a Banach algebra over any nonarchimedean field. We use our results to establish a new connection between classical functional analysis over a Tate Banach ring and the more modern language of condensed mathematics introduced by Clausen-Scholze \cite{Condensed}, \cite{ScholzeAnalytic}. To achieve these goals, we generalize to Tate Banach rings a classical concept from functional analysis over nonarchimedean fields $K$, namely, the concept of a gauge seminorm of a $K^{\circ}$-lattice in a $K$-vector space (see Schneider \cite{Schneider}, \S2). 

\subsection{The gauge seminorm on a module}

Our definition of the gauge seminorm on a module over a general Banach ring (or seminormed ring) is as follows. 
\begin{mydef}[Gauge seminorm]For a seminormed ring $(A, \lVert\cdot\rVert)$ with topologically nilpotent unit $\varpi$, an $A$-module $M$ and an $A_{\leq1}$-submodule $M_{0}$ of $M$ with $M_{0}[\varpi^{-1}]=M$ we call the function \begin{equation*}\lVert\cdot\rVert_{M,M_{0},\ast}: M\to\mathbb{R}_{\geq0}\end{equation*}given by \begin{equation*}\lVert x\rVert_{M,M_{0},\ast}=\inf\{\, r>0\mid x\in A_{\leq r}\cdot M_{0}\,\}, x\in M,\end{equation*}the gauge, or gauge seminorm, of $M_{0}$ in $M$.\end{mydef}In the above definition (and in the rest of this paper), the notation $A_{\leq r}$ (or, more precisely, $(A, \lVert\cdot\rVert)_{\leq r}$) refers to the closed unit ball of radius $r$ in a seminormed abelian group $(A, \lVert\cdot\rVert)$.

\subsection{Main results on Banach modules and almost modules}

For a seminormed ring $(A, \lVert\cdot\rVert)$, we write $(A, \lVert\cdot\rVert)^{\times,m}$ or, when the seminorm is understood from the context, $A^{\times,m}$ for the subgroup of $A^{\times}$ consisting of those units $\varpi$ in $A$ which are multiplicative with respect to the seminorm $\lVert\cdot\rVert$, i.e., which satisfy \begin{equation*}\lVert \varpi f\rVert=\lVert \varpi \rVert\lVert f\rVert\end{equation*}for all $f\in A$. We call a seminormed module $(M, \lVert\cdot\rVert_{M})$ over $(A, \lVert\cdot\rVert)$ submetric if its seminorm $\lVert\cdot\rVert_{M}$ satisfies $\lVert fx\rVert_{M}\leq \lVert f\rVert\lVert x\rVert$. Our main theorem can now be stated as follows.
\begin{thm}[Proposition \ref{Description of norms 1}, Theorem \ref{Banach modules and lattices}, Theorem \ref{Seminormed modules and almost lattices}]\label{Main theorem}Let $(A, \lVert\cdot\rVert)$ be a Banach ring with a norm-multiplicative topologically nilpotent unit $\varpi$ of norm $\leq1$ and suppose that $\varpi$ admits a compatible system $(\varpi^{1/p^{n}})_{n}$ of $p$-power roots in $A$ satisfying $\lVert\varpi^{1/p^{n}}\rVert=\lVert\varpi\rVert^{1/p^{n}}$ for all $n\geq1$. We perform almost mathematics relative to the setup $(A_{\leq1}, (\varpi^{1/p^{\infty}}))$. Then, for every submetric Banach $(A, \lVert\cdot\rVert)$-module $(M, \lVert\cdot\rVert_{M})$, the seminorm $\lVert\cdot\rVert_{M}$ on $M$ is explicitly given by \begin{equation*}\lVert x\rVert_{M}=\lVert\cdot\rVert_{M,M_{\leq1},\ast}.\end{equation*}Moreover, the functors \begin{equation*}(M, \lVert\cdot\rVert_{M})\mapsto M_{\leq1}^{a}\end{equation*}and \begin{equation*}M_{0}^{a}\mapsto (M_{0}[\varpi^{-1}], \lVert\cdot\rVert_{M,M_{0},\ast})\end{equation*}form a pair of quasi-inverse equivalences between the category $\Ban_{A}^{\leq1}$ of submetric Banach $(A, \lVert\cdot\rVert)$-modules with submetric $A$-module maps and the category of $\varpi$-adically complete, $\varpi$-torsion-free almost modules over the closed unit ball $A_{\leq1}$.\end{thm}
The difficulty of the above theorem lies in establishing that the seminorm of any submetric seminormed module $M$ is literally equal to the gauge seminorm $\lVert\cdot\rVert_{M,M_{\leq1},\ast}$ of its closed unit ball and not just equivalent to it. 

Note that the hypothesis on the Banach ring $(A, \lVert\cdot\rVert)$ in Theorem \ref{Main theorem} is satisfied for any perfectoid Tate ring $A$ equipped with a power-multiplicative norm defining its topology. It is also satisfied for any Banach algebra over a perfectoid Tate ring. In the body of the paper, we prove the first part of the theorem under a milder assumption on the Banach ring $A$: It suffices that to assume that the subgroup $\lVert A^{\times,m}\rVert$ of $\mathbb{R}_{\geq0}$ is dense (equivalently, non-discrete). We also have the following analog of Theorem \ref{Main theorem} for Banach algebras.
\begin{thm}[Theorem \ref{Seminormed algebras and almost lattices}]\label{Main theorem for algebras}Let $(A, \lVert\cdot\rVert)$ be a Banach ring and suppose that there exists a norm-multiplicative topologically nilpotent unit $\varpi$ of $A$ which admits a compatible system $(\varpi^{1/p^{n}})_{n}$ of $p$-power roots in $A$ satisfying $\lVert\varpi^{1/p^{n}}\rVert=\lVert\varpi\rVert^{1/p^{n}}$ for all $n\geq1$. We perform almost mathematics relative to the basic setup $(A_{\leq1}, (\varpi^{1/p^{\infty}})_{A_{\leq1}})$. For every $A_{\leq1}$-subalgebra $B_{0}$ of $B$ with $B_{0}[\varpi^{-1}]=B$ the gauge seminorm $\lVert\cdot\rVert_{B,B_{0},\ast}$ is a ring seminorm. The functor \begin{equation*}B\mapsto B_{\leq1}^{a}\end{equation*}from the category of submetric Banach $(A, \lVert\cdot\rVert)$-algebras to the category of $\varpi$-torsion-free and $\varpi$-adically complete $A_{\leq1}^{a}$-algebras is an equivalence of categories, with quasi-inverse given by \begin{equation*}C\mapsto (C_{\ast}[\varpi^{-1}], \lVert\cdot\rVert_{C_{\ast}[\varpi^{-1}],C_{\ast},\ast}).\end{equation*}The equivalence $B\mapsto B_{\leq1}^{a}$ restricts to an equivalence between the category of uniform Banach $(A, \lVert\cdot\rVert)$-algebras and the category of almost root closed and $\varpi$-adically complete $A_{\leq1}^{a}$-algebras.\end{thm}

\subsection{Relation to condensed mathematics} In recent years, it has become most common (and most natural) to view the category $\Ban_{A}$ of Banach modules over a Banach ring $A$, with bounded $A$-module homomorphisms as morphisms, as a full subcategory of the category of condensed modules (or of solid modules) over the (solid) condensed ring $\underline{A}$ associated with $A$, in the sense of the condensed mathematics of Clausen and Scholze \cite{Condensed}, \cite{ScholzeAnalytic}.

On the other hand, the category $\Ban_{A}^{\leq1}$ of (submetric) Banach $A$-modules with submetric $A$-module homomorphisms (i.e., the $A$-module homomorphisms which do not increase the norm of any element) has better categorical properties than $\Ban_{A}$; for example, it is complete and cocomplete. However, it is a priori unclear how the category $\Ban_{A}^{\leq1}$ is related to condensed mathematics. Using Theorem \ref{Main theorem}, we provide an answer to this question in the case when $A$ is a perfectoid Tate ring (or when $A$ is any Banach ring satisfying the hypothesis of Theorem \ref{Main theorem}).
\begin{thm}[Theorem \ref{Relation to condensed mathematics, precise version}]\label{Relation to condensed mathematics}Let $A$ be a Tate Banach ring satisfying the hypothesis of Theorem \ref{Main theorem} (for example, this is the case if $A$ is perfectoid). Then the category $\Ban_{A}^{\leq1}$ has a natural fully faithful embedding \begin{equation*}\Ban_{A}^{\leq1}\hookrightarrow \Cond(A_{\leq1}, (\varpi^{1/p^{\infty}})_{A_{\leq1}})\end{equation*}into the category $\Cond(A_{\leq1}, (\varpi^{1/p^{\infty}})_{A_{\leq1}})$ of (static) condensed almost $(A_{\leq1}, (\varpi^{1/p^{\infty}}))$-modules in the sense of Mann \cite{Mann22}, Definition 2.2.5. The image of this embedding is contained in the subcategory \begin{equation*}\Cond_{\blacksquare}(A_{\leq1}, (\varpi^{1/p^{\infty}})_{A_{\leq1}})\end{equation*}of solid almost $(A_{\leq1}, (\varpi^{1/p^{\infty}})_{A_{\leq1}})$-modules (see Definition \ref{Solid almost modules}). 

Moreover, if $A$ is a perfectoid Tate ring endowed with a power-multiplicative norm defining its topology (so that $A_{\leq1}=A^{\circ}$) and if the affinoid perfectoid space $\Spa(A, A^{\circ})$ is totally disconnected, then the above embedding is symmetric monoidal when $\Ban_{A}^{\leq1}$ is endowed with the complete tensor product $-\widehat{\otimes}_{A}-$ and $\Cond_{\blacksquare}(A^{\circ}, (\varpi^{1/p^{\infty}})_{A^{\circ}})$ is endowed with solid tensor product $-\otimes_{A^{\circ a}}^{\blacksquare}-$ of solid almost modules (see Definition \ref{Solid tensor product}).\end{thm}
We note that condensed almost modules were used by Mann \cite{Mann22} and Anschütz, Le Bras and Mann \cite{ALBM24} for the development of theories of quasi-coherent complexes on small v-stacks. The above theorem provides a link between categories of Banach modules with submetric module homomorphisms and these theories, which can facilitate translation between the older, more classical language of Banach modules and submetric module homomorphisms and the modern language of condensed mathematics. This could allow us to transport intuition from classical functional analysis to formulate and prove new results in the condensed setting or to solve problems originally formulated in the classical nonarchimedean-analytic language by using the power of condensed mathematics.   

\subsection{Relation to previously known results}

Theorem \ref{Main theorem} and Theorem \ref{Main theorem for algebras} initiate a dictionary between the analytic world of Banach algebras and Banach modules over the Banach ring $A$ and the algebraic world of $\varpi$-torsion-free, $\varpi$-adically complete almost algebras and almost modules over the closed ball unit $A_{\leq1}$ of $A$. In the case when $A=K$ is a nonarchimedean field, such dictionaries for Banach algebras over $K$ were provided in \cite{Andre18}, Sorite 2.3.1, and \cite{BhattNotes}, \S5.2, but for more general Tate Banach rings $A$ our results appear to be new. Another difference between the results of this paper and the previously known results is that in our setting we prove that the norm of any Banach module is determined by the closed unit ball "on the nose" rather than just up to equivalence, while in the previously known results one had to make additional assumptions on the Banach module (such as the Banach module being a uniform Banach algebra) for this to be the case. This last difference is crucial for the application to Theorem \ref{Relation to condensed mathematics}.   

\subsection{Outline of the paper} 

In Section \ref{sec:boundedness and continuity} we generalize the classical result from functional analysis over a nonarchimedean field $K$ saying that every continuous $K$-linear map between seminormed $K$-vector spaces is bounded. In Section \ref{sec:gauge seminorms}, we introduce and study gauge seminorms on modules over general seminormed rings $A$. We prove the first part of Theorem \ref{Main theorem} saying that, under a mild assumption on the value set of the seminorm of $A$, every submetric $A$-module seminorm on any $A$-module $M$ is equal to the gauge seminorm of its closed unit ball. In Section \ref{sec:almost mathematics}, we prove the rest of Theorem \ref{Main theorem} and Theorem \ref{Main theorem for algebras} by relating our results on gauge seminorms to almost mathematics. In Section \ref{sec:complete tensor products}, we study isometries and strict homomorphisms of Banach modules and give an explicit description of the the closed unit ball of the complete tensor product $M\widehat{\otimes}_{A}N$ of two Banach modules $M$, $N$ in terms of the closed unit balls of $M$ and $N$. Finally, Section \ref{sec:Proof of the theorem on condensed math} is devoted to the proof of Theorem \ref{Relation to condensed mathematics}.

\subsection{Notation and terminology}

In this paper we consider all rings to be commutative and with unit. Recall that a seminorm on an abelian group $G$ is a function $\lVert \cdot \rVert: G \to \mathbb{R}_{\geq 0}$ such that $\lVert 0 \rVert = 0$ and such that the triangle inequality $\lVert f + g \rVert \leq \lVert f \rVert + \lVert g \rVert$ holds for any two elements $f, g \in G$. A seminorm is said to be nonarchimedean if for any $f, g \in G$ the stronger inequality $\lVert f + g \rVert \leq \max(\lVert f \rVert, \lVert g \rVert)$ is satisfied. Unless explicitly stated otherwise we only consider nonarchimedean seminorms in this paper. A seminorm $\lVert \cdot \rVert$ is called a norm if $0$ is the only element of $G$ which is mapped to $0$ under $\lVert \cdot \rVert$. Two seminorms $\lVert \cdot \rVert_{1}$ and $\lVert \cdot \rVert_{2}$ are called bounded-equivalent if there exist constants $C_{1}, C_{2} > 0$ such that \begin{equation*} \lVert f \rVert_{2} C_{2} \leq \lVert f \rVert_{1} \leq C_{1} \lVert f \rVert_{2} \end{equation*}holds for all elements $f \in G$. The term 'bounded-equivalent' is taken from \cite{Johansson-Newton1}, \S2.1.

For a ring $A$, a seminorm on $A$ is a seminorm on the underlying additive abelian group of $A$ such that $\lVert 1\rVert=1$. A seminorm $\lVert \cdot \rVert$ on a ring $A$ is said to be submultiplicative (or a ring seminorm) if it is compatible with multiplication in the sense that \begin{equation*} \lVert fg \rVert \leq \lVert f \rVert \lVert g \rVert \end{equation*}for any two elements $f, g \in A$. By a (semi)normed ring (respectively, a Banach ring) $(A, \lVert \cdot \rVert)$ we will mean a ring $A$ equipped with a submultiplicative (semi)norm (resp., a complete submultiplicative norm) $\lVert \cdot \rVert$. In this paper we tacitly assume all our seminorms on rings to be submultiplicative and we often suppress the norm $\lVert\cdot\rVert$ from the notation for a (semi)normed ring (respectively, Banach ring). For a seminormed abelian group $A$ we denote by $A_{\leq r}$ the closed ball of radius $r>0$ in $A$, i.e.,\begin{equation*}A_{\leq r}=\{\, f\in A\mid \lVert f\rVert\leq r\,\}.\end{equation*}Since we restrict our attention to nonarchimedean seminorms, this is always a subgroup of $A$, open with respect to the topology defined by the given seminorm, and if $A$ is a seminormed ring and $c\leq1$, then $A_{\leq c}$ is an open subring. 

A seminorm on a ring $A$ is said to be power-multiplicative (resp., multiplicative) if $\lVert f^{n} \rVert = \lVert f \rVert^{n}$ for all $f$ and all integers $n \geq 0$ (resp., $\lVert fg \rVert = \lVert f \rVert \lVert g \rVert$ for all $f, g \in A$). A multiplicative norm on a ring $A$ is also called an absolute value. By a nonarchimedean field $K$ we mean a Banach ring $(K, \vert\cdot\vert)$ whose underlying ring is a field and whose norm is an absolute value.

A homomorphism $\varphi: G \to H$ between seminormed abelian groups is said to be bounded if there exists a constant $C>0$ with $\lVert \varphi(g) \rVert \leq C \lVert g \rVert$ for all $g \in G$. It is said to be submetric if the above constant $C$ can be taken to be equal to $1$. If $\varphi: A \to B$ is a bounded homomorphism between seminormed rings and the seminorm on the target $B$ is power-multiplicative, then $\varphi$ is automatically submetric (see \cite{BGR}, Proposition 1.3.1/2). A homomorphism $\varphi: G\to H$ of seminormed abelian groups is called strict if on $\im(\varphi)$ the subspace seminorm inherited from $H$ is bounded-equivalent to the quotient seminorm induced by $\varphi$. In particular, if $\varphi$ is injective and bounded, then it is strict if and only if there exists a constant $C'>0$ such that \begin{equation*}\lVert x\rVert_{G}\leq C'\lVert\varphi(x)\rVert_{H}\end{equation*}for all $x\in G$. This notion of strict homomorphisms is a special case of the categorical notion of a strict morphism in an additive category with kernels and cokernels (see \cite{Schneiders}, Def.~1.1.1): A morphism in such a category is called strict if it induces an isomorphism between the coimage (the cokernel of its kernel) and its image (the kernel of its cokernel). In the case of a bounded homomorphism $\varphi: G\to H$ of seminormed abelian groups, the coimage is given by $G/\ker(\varphi)$ with the quotient seminorm and the image is given by $\varphi(G)$ with the subspace seminorm.  

For a seminormed ring $(A, \lVert\cdot\rVert)$, a seminormed $A$-module is a seminormed abelian group $(M, \lVert\cdot\rVert_{M})$ such that $M$ is an $A$-module and there exists a constant $C>0$ such that \begin{equation*}\lVert fx\rVert_{M}\leq C\lVert f\rVert\lVert x\rVert_{M}\end{equation*}for all $f\in A$ and $x\in M$. The seminormed $A$-module $(M, \lVert\cdot\rVert_{M})$ is called a normed module (respectively, a Banach module) if the seminorm $\lVert\cdot\rVert_{M}$ is a norm (respectively, a complete norm). If for some seminormed module (respectively, Banach module) the above constant $C>0$ can be chosen to be equal to $1$, we call it a \textit{submetric seminormed module} (respectively, a \textit{submetric Banach module}). A seminormed module (respectively, a Banach module) over a nonarchimedean field $K$ is called a seminormed space (respectively, a Banach space) over $K$. Every seminormed space over a nonarchimedean field is submetric, by a simple argument involving the field's absolute value, but the analogous assertion over more general seminormed rings is false. For example, if $A\to B$ is a bounded homomorphism of seminormed rings, then $B$ is a seminormed $A$-module, but it is a submetric seminormed module only if the map $A\to B$ is submetric. On the other hand, whenever $(M, \lVert\cdot\rVert)$ is a seminormed $A$-module, there always exists a bounded-equivalent seminorm $\lVert\cdot\rVert_{M}'$ on $M$ satisfying \begin{equation*}\lVert fx\rVert_{M}'\leq \lVert f\rVert\lVert x\rVert_{M}'\end{equation*}for all $f\in A$ and $x\in M$ (see \cite{Bambozzi-Ben-Bassat16}, Remark 3.7 and Proposition 3.28). For a Banach ring $(A, \lVert\cdot\rVert)$ we use the following notations: \begin{itemize} \item $\Snm_{A}$ for the category of seminormed $A$-modules with bounded $A$-module homomorphisms as morphisms, called the category of seminormed $A$-modules; \item $\Ban_{A}$ for the category of Banach $A$-modules with bounded $A$-module homomorphisms as morphisms, called the category of Banach $A$-modules; \item $\Snm_{A}^{\leq1}$ for the category of submetric seminormed $A$-modules with submetric $A$-module homomorphisms as morphisms, called the category of submetric seminormed $A$-modules; \item $\Ban_{A}^{\leq1}$ for the category of submetric Banach $A$-modules with submetric $A$-module homomorphisms as morphisms, called the category of submetric Banach $A$-modules.\end{itemize}In the case when the seminormed ring $A$ in question is $\mathbb{Z}$ equipped with the trivial norm which sends every non-zero element to $1$, these specialize to the categories of (submetric) seminormed abelian groups and (submetric) Banach abelian groups. Finally, by a seminormed $A$-algebra (respectively, Banach $A$-algebra) we mean a seminormed ring (respectively, Banach ring) $B$ which is also a seminormed (respectively, Banach) $A$-module, i.e., a seminormed ring (respectively, a Banach ring) $B$ with a bounded ring homomorphism $A\to B$. Similarly, a submetric seminormed $A$-algebra (respectively, a submetric Banach $A$-algebra) is a seminormed ring (respectively, Banach ring) $A$ equipped with a submetric ring homomorphism $A\to B$.  

A Huber ring $A$ is a topological ring that contains an open adic subring $A_0$ (the so-called ring of definition) with a finitely generated ideal of definition. A pair $(A_{0}, I)$, where $A_{0}$ is a ring of definition of $A$ and $I$ is a finitely generated ideal of definition of $A_{0}$, is called a pair of definition of the Huber ring $A$. A Huber ring $A$ is called a Tate ring if it contains a topologically nilpotent unit $\varpi$. We sometimes refer to a topologically nilpotent unit in a Tate ring as a pseudo-uniformizer. Every topologically nilpotent unit $\varpi$ in a Tate ring $A$ is contained in a ring of definition $A_{0}$ and then $(A_{0}, (\varpi)_{A_{0}})$ is a pair of definition of the Huber ring $A$ (\cite{Huber0}, Proposition 2.2.6iii)). In the sequel we will use without further comment the following well-known relationship between Tate rings and seminormed rings (see, for example, \cite{Kedlaya-Liu}, Remark 2.4.4). If $A$ is a Tate ring (respectively, a Hausdorff Tate ring) with topologically nilpotent unit $\varpi$ and ring of definition $A_{0}$ containing $\varpi$, then \begin{equation*}\lVert f\rVert_{A_{0},\varpi}:= \inf \{\, 2^{-n} \mid n\in\mathbb{Z}, f\in\varpi^{n}A_{0} \,\},  \\\ f \in A \end{equation*}defines a seminorm (respectively, a norm) which induces the topology of $A$ (conversely, every seminormed ring $A$ with a topologically nilpotent unit $\varpi$ is a Tate ring, where we can take the closed unit ball $A_{\leq1}$ as a ring of definition). For a Tate ring (respectively, a Hausdorff Tate ring) $A$ with pair of definition $(A_{0}, \varpi)$, we call the seminorm $\lVert\cdot\rVert_{A_{0},\varpi}$ the canonical extension of the $\varpi$-adic seminorm (respectively, of the $\varpi$-adic norm) on the subring $A_{0}$. Note that for this seminorm, the topologically nilpotent unit $\varpi$ becomes seminorm-multiplicative in the sense of Definition \ref{Norm-multiplicative elements}. In this situation, we also have an analogously defined $A$-module seminorm $\lVert\cdot\rVert_{M_{0},\varpi}$ on any $A$-module $M$ and any $A_{0}$-submodule $M_{0}$ of $M$ satisfying $M=M_{0}[\varpi^{-1}]$, and we also call this seminorm on $M$ the canonical extension of the $\varpi$-adic seminorm on $M_{0}$. 

We call an element $f$ in a topological ring $A$ power-bounded if the set of all powers of $f$ is a bounded subset of $A$ and we call a Huber ring $A$ uniform if its subring $A^{\circ}$ consisting of power-bounded elements is bounded. If $A$ is a seminormed ring whose underlying topological ring is a Tate ring (i.e., $A$ has a topologically nilpotent unit), then $A$ is uniform as a Tate ring if and only if the topology of $A$ can be defined a bounded power-multiplicative seminorm (see, for example, \cite{Kedlaya17}, Exercise~1.5.13). Note that, by \cite{Kedlaya-Liu}, Remark 2.8.18 (or by Lemma \ref{Uniform Tate rings} below), such a power-multiplicative seminorm can always be chosen so as to make any given topologically nilpotent unit seminorm-multiplicative in the sense of Definition \ref{Norm-multiplicative elements}. 

Finally, in the context of almost mathematics, by a basic setup for almost mathematics, or an almost setup, we mean a pair $(V, \mathfrak{m})$ consisting of a ring $V$ and an ideal $\mathfrak{m}\subseteq V$ such that $\mathfrak{m}^{2}=\mathfrak{m}$ and such that the $V$-module $\widetilde{\mathfrak{m}}=\mathfrak{m}\otimes_{V}\mathfrak{m}$ is flat.  

\subsection{Acknowledgements}

I would like to express my heartfelt gratitude to my advisor, Kiran Kedlaya, for his encouragement and guidance, for his feedback on multiple drafts of this paper and for numerous helpful discussions. I would like to thank Johannes Anschütz and Lucas Mann for some valuable conversations on the notions of quasi-coherent sheaves developed in  \cite{Mann22} and \cite{ALBM24}. I am grateful to my colleagues Jack J Garzella, Shubhankar Sahai and Nathan Wenger for helpful conversations. 

This work was supported by NSF grant DMS-2401536 "Global cohomological approaches to L-functions" (PI: Kiran Kedlaya). 
    
\section{Norm-multiplicative elements, boundedness and continuity}\label{sec:boundedness and continuity}

Let $(A, \lVert\cdot\rVert)$ be a seminormed ring.  
\begin{mydef}[Submetric seminormed module, submetric Banach module]A seminormed $(A, \lVert\cdot\rVert)$-module $(M, \lVert\cdot\rVert_{M})$ is called submetric if the map $A\times M\to M$ defining the $A$-module stucture on $M$ is submetric, i.e., if for all $x\in M$ and $f\in A$ we have \begin{equation*}\lVert fx\rVert_{M}\leq \lVert f\rVert\lVert x\rVert_{M}.\end{equation*}If $(M, \lVert\cdot\rVert_{M})$ is a Banach module, we call it a submetric Banach module.\end{mydef}Recall from the introduction that we denote the category of all seminormed $A$-modules over a seminormed ring $A$ (respectively, the category of all Banach $A$-modules over a Banach ring $A$), with bounded $A$-module homomorphisms as morphisms, by $\Snm_{A}$ (respectively, by $\Ban_{A}$) and we denote the category of submetric seminormed $A$-modules (respectively, submetric Banach $A$-modules), with submetric $A$-module homomorphisms as morphisms, by $\Snm_{A}^{\leq1}$ (respectively, $\Ban_{A}^{\leq1}$). Note that we have fully faithful inclusion functors $\Snm_{A}^{\leq1}\hookrightarrow \Snm_{A}$ and $\Ban_{A}^{\leq1}\hookrightarrow \Ban_{A}$. 

As in complex or real functional analysis, one of the basic facts of nonarchimedean functional analysis is the fact that a $K$-linear map between Banach spaces over a nonarchimedean field $K$ is bounded if and only if it is continuous. To generalize this assertion from nonarchimedean fields to other seminormed rings we need the notion of a seminorm-multiplicative element.

\begin{mydef}[Seminorm-multiplicative elements]\label{Norm-multiplicative elements}An element $\varpi$ in a seminormed ring $(A, \lVert\cdot\rVert_{A})$ is called seminorm-multiplicative (or multiplicative with respect to the seminorm $\lVert\cdot\rVert_{A}$) if $\lVert\varpi f\rVert_{A}=\lVert \varpi\rVert_{A}\lVert f\rVert_{A}$ for all $f\in A$. If $M$ is an $A$-module endowed with an abelian group seminorm $\lVert\cdot\rVert_{M}$, then $\varpi$ is called $\lVert\cdot\rVert_{A}$-multiplicative with respect to the seminorm $\lVert\cdot\rVert_{M}$ on $M$ if $\lVert \varpi x\rVert_{M}=\lVert\varpi\rVert_{A}\lVert x\rVert_{M}$ for all $x\in M$.\end{mydef}
\begin{rmk}\label{Caution 1}Note that if $\lVert\cdot\rVert$ and $\phi$ are two seminorms on the same ring $A$ such that $\phi$ is continuous with respect to the topology defined by $\lVert\cdot\rVert$, then an element $\varpi\in A$ can be multiplicative with respect to both $\lVert\cdot\rVert$ and $\phi$, but still fail to be $\lVert\cdot\rVert_{A}$-multiplicative with respect to $\phi$. In fact, since $\lVert1\rVert=1$, $\phi(1)=1$, an element $\varpi\in A$ is $\lVert\cdot\rVert_{A}$-multiplicative with respect to $\phi$ if and only if $\lVert\varpi\rVert=\phi(\varpi)$. For example, let $K$ be a field endowed with a non-trivial absolute value $\vert\cdot\vert$ and let $s\in(0, 1)$. Then $\phi=\vert\cdot\vert^{s}$ is again an absolute value which defines the same topology on $K$ as $\vert\cdot\vert$, every element of $K$ is multiplicative both with respect to $\vert\cdot\vert$ and $\phi$, but the only non-zero elements $\lambda\in K$ which are $\vert\cdot\vert$-multiplicative with respect to $\phi$ are the elements with absolute value $\vert\lambda\vert=1$.\end{rmk} 
Recall that an invertible element $\varpi$ in a seminormed ring $(A, \lVert\cdot\rVert)$ is seminorm-multiplicative if and only if $\lVert\varpi\rVert\lVert\varpi^{-1}\rVert=1$ (see, for example, \cite{Dine22}, Remark 2.13). The seminorm-multiplicative units in a seminormed ring $(A, \lVert\cdot\rVert)$ form a subgroup of $A^{\times}$ which we denote by $(A, \lVert\cdot\rVert)^{\times,m}$. We simply write $A^{\times,m}$ instead of $(A, \lVert\cdot\rVert)^{\times,m}$ whenever the seminorm $\lVert\cdot\rVert$ is understood from the context.
\begin{lemma}\label{Multiplicative topologically nilpotent units}Let $\varpi$ be a seminorm-multiplicative unit in a seminormed ring $(A, \lVert\cdot\rVert_{A})$ and let $(M, \lVert\cdot\rVert)\in \Snm_{A}^{\leq1}$. Then $\varpi$ is $\lVert\cdot\rVert_{A}$-multiplicative for the seminorm $\lVert\cdot\rVert$ on $M$.\end{lemma}
\begin{proof}Since $\varpi$ is invertible in $A$, we have\begin{equation*}\lVert\varpi\rVert_{A}\lVert x\rVert=\lVert\varpi\rVert_{A}\lVert\varpi^{-1}\varpi_{A} x\rVert\leq \lVert\varpi\rVert_{A}\lVert\varpi^{-1}\rVert\lVert\varpi x\rVert=\lVert\varpi x\rVert.\end{equation*}The opposite inequality holds by the assumption that $(M, \lVert\cdot\rVert)\in\Snm_{A}^{\leq1}$.\end{proof}
\begin{lemma}\label{Seminormed algebras and seminorm-multiplicative units}For a submetric homomorphism of seminormed algebras \begin{equation*}\varphi: (A, \lVert\cdot\rVert_{A})\to (B, \lVert\cdot\rVert_{B}),\end{equation*}we have $\lVert\varphi(\varpi)\rVert_{B}=\lVert\varpi\rVert_{A}$ for all $\varpi\in A^{\times,m}$.\end{lemma}
\begin{proof}Let $\varpi\in A^{\times,m}$. By Lemma \ref{Multiplicative topologically nilpotent units}, we have $\varphi(A^{\times,m})\subseteq B^{\times,m}$. In particular, $\lVert\varphi(\varpi)^{-1}\rVert_{B}=\lVert\varphi(\varpi)\rVert_{B}^{-1}$. But then \begin{equation*}\lVert\varphi(\varpi)\rVert_{B}^{-1}=\lVert\varphi(\varpi^{-1})\rVert_{B}\leq\lVert\varpi^{-1}\rVert_{A}=\lVert\varpi\rVert_{A}^{-1}\end{equation*}and the assertion follows.\end{proof}
\begin{lemma}\label{Seminormed algebras and seminorm-multiplicative units 2}Suppose that $\varphi: (A, \lVert\cdot\rVert_{A})\to (B, \lVert\cdot\rVert_{B})$ is a submetric homomorphism of seminormed rings. If $\lVert A^{\times,m}\rVert_{A}$ is dense in $\mathbb{R}_{\geq0}$, then so is $\lVert B^{\times,m}\rVert_{B}$.\end{lemma}
\begin{proof}Combining Lemma \ref{Multiplicative topologically nilpotent units} and Lemma \ref{Seminormed algebras and seminorm-multiplicative units} we see that \begin{equation*}\lVert A^{\times,m}\rVert_{A}=\lVert\varphi(A^{\times,m})\rVert_{B}\subseteq \lVert B^{\times,m}\rVert_{B}.\end{equation*}\end{proof}
The key to establishing the promised relationship between boundedness and continuity is the following lemma. The idea of its proof goes back to the work of Johansson and Newton (see \cite{Johansson-Newton2}, Lemma 2.1.7).
\begin{lemma}\label{Boundedness vs. continuity}Let $(A, \lVert\cdot\rVert_{A})$ be a seminormed ring with a seminorm-multiplicative topologically nilpotent unit $\varpi$. Let $M$ be a topological $A$-module whose topology is defined by an abelian group seminorm $\lVert\cdot\rVert$ (not necessarily an $A$-module seminorm) and let $\lVert\cdot\rVert'$ be another abelian group seminorm on $M$ which is continuous with respect to $\lVert\cdot\rVert$. Suppose that $\varpi$ is $\lVert\cdot\rVert_{A}$-multiplicative with respect to each of the seminorms $\lVert\cdot\rVert$, $\lVert\cdot\rVert'$ on $M$. Then $\lVert\cdot\rVert'$ is bounded with respect to $\lVert\cdot\rVert$. More precisely, if $m>0$ is an integer such that $\lVert x\rVert\leq\lVert\varpi^{m}\rVert_{A}$ implies $\lVert x\rVert'\leq1$, then\begin{equation*}\lVert x\rVert'\leq \lVert\varpi\rVert_{A}^{-2m}\lVert x\rVert\end{equation*}for all $x\in M$.\end{lemma}
\begin{proof}We argue as in the proof of \cite{Dine22}, Lemma 2.11 (which is a slight modification of the proof of \cite{Johansson-Newton2}, Lemma 2.1.7); we include the argument here for the sake of completeness. By the continuity of $\lVert\cdot\rVert'$ with respect to $\lVert\cdot\rVert$ we can find $\epsilon\in(0,1)$ such that $\lVert x\rVert\leq\epsilon$ implies $\lVert x\rVert'\leq1$ for any $x\in M$. Choose an integer $m>0$ such that $\lVert\varpi^{m}\rVert_{A}\leq\epsilon$. Fix some $x\in M$. To prove that $\lVert x\rVert'\leq \lVert\varpi\rVert_{A}^{-2m}\lVert x\rVert$, first note that the kernel of any seminorm is equal to the closure of $\{0\}$ with respect to the topology defined by that seminorm, so $\lVert x\rVert=0$ implies $\lVert x\rVert'=0$, by continuity of $\lVert\cdot\rVert'$ with respect to $\lVert\cdot\rVert$. 

Now assume that $\lVert x\rVert\neq 0$. Recall (for example, from \cite{Dine22}, Remark 2.13) that an invertible element $a$ in a seminormed ring $A$ is seminorm-multiplicative if and only if $a$ satisfies $\lVert a\rVert^{-1}=\lVert a^{-1}\rVert$; in particular, $a$ is seminorm-multiplicative if and only if $a^{-1}$ is. Hence $\lVert\varpi\rVert_{A}^{-1}=\lVert\varpi^{-1}\rVert_{A}$ and the same also holds for every integer power of $\varpi$. If for every $n\in\mathbb{Z}$ we had $\lVert \varpi^{mn}\rVert_{A}\lVert x\rVert\leq1$, then we would obtain $\lVert x\rVert\leq\lVert\varpi^{mn}\rVert_{A}^{-1}=\lVert\varpi^{-mn}\rVert_{A}$ for all $n\in\mathbb{Z}$ and thus $\lVert x\rVert=0$. This shows that there exists a lowest integer $n\in\mathbb{Z}$ such that \begin{equation*}\lVert\varpi^{mn}\rVert_{A}\lVert x\rVert\leq1.\end{equation*}Then \begin{equation*}\lVert\varpi^{m(n+1)}x\rVert=\lVert\varpi^{m(n+1)}\rVert_{A}\lVert x\rVert=\lVert\varpi^{m}\rVert_{A}\lVert\varpi^{mn}\rVert_{A}\lVert x\rVert\leq \lVert\varpi^{m}\rVert_{A}\leq\epsilon\end{equation*}and this implies that $\lVert\varpi^{m(n+1)}x\rVert'\leq 1$ by our choice of $\epsilon$, where we used the assumption that $\varpi$ is $\lVert\cdot\rVert_{A}$-multiplicative with respect to $\lVert\cdot\rVert$. Since $\varpi$ is also $\lVert\cdot\rVert_{A}$-multiplicative with respect to $\lVert\cdot\rVert'$, this last inequality is equivalent to $\lVert\varpi^{m(n+1)}\rVert_{A}\lVert x\rVert'\leq1$. But then \begin{equation*}\lVert x\rVert'\leq\lVert\varpi\rVert_{A}^{-m(n+1)}=\lVert\varpi\rVert_{A}^{-2m}\lVert\varpi^{-m(n-1)}\rVert_{A}<\lVert\varpi\rVert_{A}^{-2m}\lVert x\rVert,\end{equation*}by our choice of $n\in\mathbb{Z}$.\end{proof}
The announced relationship between boundedness and continuity can now be formulated as follows.
\begin{cor}\label{Boundedness vs. continuity 1.5}Let $(A, \lVert\cdot\rVert_{A})$ be a seminormed ring with a seminorm-multiplicative topologically nilpotent unit $\varpi$ and let $(M, \lVert\cdot\rVert)$ be a seminormed $A$-module. Let $\lVert\cdot\rVert'$ be another $A$-module seminorm on $M$ which is continuous with respect to $\lVert\cdot\rVert$. Then $\lVert\cdot\rVert'$ is bounded with respect to $\lVert\cdot\rVert$.\end{cor}
\begin{proof}Up to replacing $\lVert\cdot\rVert$ and $\lVert\cdot\rVert'$ with bounded-equivalent $A$-module seminorms we may assume that \begin{equation*}(M, \lVert\cdot\rVert), (M, \lVert\cdot\rVert')\in \Snm_{A}^{\leq1}\end{equation*}(see Remark 3.7 and Proposition 3.28 in \cite{Bambozzi-Ben-Bassat16}). By Lemma \ref{Multiplicative topologically nilpotent units}, $\varpi$ is then a $\lVert\cdot\rVert_{A}$-multiplicative element with respect to both $\lVert\cdot\rVert$ and $\lVert\cdot\rVert'$. We conclude by Lemma \ref{Boundedness vs. continuity}.\end{proof}
\begin{cor}\label{Boundedness vs. continuity 2}Let $(A,\lVert\cdot\rVert_{A})$ be a seminormed ring with a seminorm-multiplicative topologically nilpotent unit $\varpi$, let $(M, \lVert\cdot\rVert_{M})$ be a seminormed $(A, \lVert\cdot\rVert_{A})$-module and let\begin{equation*}\varphi: M\to N\end{equation*}be a continuous homomorphism of topological $A$-modules, where the topology on $M$ is defined by $\lVert\cdot\rVert_{M}$ and the topology on $N$ is defined by an abelian group seminorm $\lVert\cdot\rVert_{N}$ satisfying\begin{equation*}\lVert\varpi x\rVert_{N}=\lVert\varpi\rVert_{A}\lVert x\rVert_{N}\end{equation*}for all $x\in N$. Then $\varphi$ is bounded with respect to the seminorms $\lVert\cdot\rVert_{M}$ on $M$ and $\lVert\cdot\rVert_{N}$ on $N$.\end{cor}
\begin{proof}By \cite{Bambozzi-Ben-Bassat16}, Remark 3.7 and Proposition 3.28, we may again assume that $(M, \lVert\cdot\rVert_{M})$ is a submetric seminormed $(A, \lVert\cdot\rVert_{A})$-module. It suffices to prove that the subspace seminorm $\lVert\cdot\rVert'$ on $\varphi(M)$ induced from the seminorm $\lVert\cdot\rVert_{N}$ on $N$ is bounded with respect to the quotient seminorm $\lVert\cdot\rVert$ on $\varphi(M)$ induced from $\lVert\cdot\rVert_{M}$ via $\varphi$. By the continuity of $\varphi$, we know that $\lVert\cdot\rVert'$ is continuous with respect to $\lVert\cdot\rVert$. By definition of the quotient seminorm, the assumption that $(M, \lVert\cdot\rVert_{M})\in\Snm_{A}^{\leq1}$ implies $(\varphi(M), \lVert\cdot\rVert)\in\Snm_{A}^{\leq1}$. By Lemma \ref{Multiplicative topologically nilpotent units}, this entails that $\varpi$ is $\lVert\cdot\rVert_{A}$-multiplicative with respect to the seminorm $\lVert\cdot\rVert$ on $\varphi(M)$. We conclude by applying Lemma \ref{Boundedness vs. continuity}.\end{proof}
\begin{cor}\label{Boundedness vs. continuity 3}Let $(A,\lVert\cdot\rVert_{A})$ be a seminormed ring with a seminorm-multiplicative topologically nilpotent unit $\varpi$. Let $\varphi: M\to N$ be a homomorphism between seminormed $(A, \lVert\cdot\rVert_{A})$-modules which is continuous and open onto its image. Then $\varphi$ is a strict homomorphism of seminormed $A$-modules, i.e., the quotient seminorm on $\varphi(M)\subseteq N$ is bounded-equivalent to the subspace seminorm.\end{cor}
\begin{proof}The assumption that $\varphi$ is continuous and open onto its image means that the subspace semiorm $\lVert\cdot\rVert'$ and the quotient seminorm $\lVert\cdot\rVert$ on $\varphi(M)$ define the same topology. By \cite{Bambozzi-Ben-Bassat16}, Remark 3.7 and Proposition 3.28, we may assume (up to replacing both seminorms with bounded-equivalent ones) that $(\varphi(M), \lVert\cdot\rVert')$ and $(\varphi(M), \lVert\cdot\rVert)$ are submetric seminormed $A$-modules, in which case $\varpi$ is $\lVert\cdot\rVert_{A}$-multiplicative with respect to both seminorms on $\varphi(M)$, by Lemma \ref{Multiplicative topologically nilpotent units}. Then the assertion follows from Lemma \ref{Boundedness vs. continuity}. \end{proof}
For comparing two seminorms on a ring $A$, we can also use the following variant of Lemma \ref{Boundedness vs. continuity}.
\begin{lemma}[\cite{Dine25}, Lemma 2.20]\label{Boundedness vs. continuity, variant for rings}Let $(A, \lVert\cdot\rVert)$ be a seminormed ring and let $\phi$ be a seminorm on $A$ which is continuous with respect to the topology on $A$ defined by $\lVert\cdot\rVert$. Suppose that there exists a topologically nilpotent unit $\varpi\in A$ which is multiplicative with respect to both $\lVert\cdot\rVert$ and $\phi$. If \begin{equation*}\phi(\varpi)=\lVert\varpi\rVert,\end{equation*}then the seminorm $\phi$ is bounded with respect to $\lVert\cdot\rVert$. In particular, every continuous multiplicative seminorm $\phi$ on $A$ which satisfies $\phi(\varpi)=\lVert\varpi\rVert$ is bounded.\end{lemma}
\begin{proof}As we observed in Remark \ref{Caution 1}, the assumptions that $\varpi$ is multiplicative with respect to $\lVert\cdot\rVert_{A}$, $\phi$ and that $\phi(\varpi)=\lVert\cdot\rVert$ imply that $\varpi$ is $\lVert\cdot\rVert_{A}$-multiplicative with respect to $\phi$. Therefore, the assertion follows from Lemma \ref{Boundedness vs. continuity}. \end{proof}
Recall that a Tate ring $A$ is called uniform if its subring of power-bounded elements $A^{\circ}$ is bounded. Recall also (for example, from \cite{Berkovich}, \S1.3) that for every seminormed ring $(A, \lVert\cdot\rVert)$ the limit \begin{equation*}\vert f\vert_{\spc}=\lim_{n\to\infty}\lVert f^{n}\rVert^{1/n}=\inf_{n}\lVert f^{n}\rVert^{1/n}, f\in A, \end{equation*}defines a power-multiplicative seminorm which is the largest power-multiplicative seminorm on $A$ bounded above by $\lVert\cdot\rVert$ (i.e., every power-multiplicative seminorm on $A$ bounded above by $\lVert\cdot\rVert$ is also bounded above by $\vert\cdot\vert_{\spc}$) and which is called the spectral seminorm on $A$ derived from the seminorm $\lVert\cdot\rVert$. Then Lemma \ref{Boundedness vs. continuity, variant for rings} gives us the following result. 
\begin{lemma}[\cite{Dine25}, Lemma 2.25]\label{Uniform Tate rings}Let $A$ be a uniform Tate ring and let $\varpi\in A$ be a topologically nilpotent unit. There exists a unique power-multiplicative seminorm $\vert\cdot\vert_{\spc, \varpi}$ defining the topology on $A$ for which $\varpi$ is seminorm-multiplicative and $\vert\varpi\vert_{\spc, \varpi}=\frac{1}{2}$. In particular, for any ring of definition $A_{0}$ of $A$ with $\varpi\in A_{0}$ the spectral seminorm derived from the canonical extension $\lVert\cdot\rVert_{A_{0}, \varpi}$ of the $\varpi$-adic seminorm on $A$ is equal to $\vert\cdot\vert_{\spc, \varpi}$.\end{lemma}
\begin{rmk}\label{Caution 2}As a word of caution, let us point out that neither the assumption that $\varpi$ be $\lVert\cdot\rVert_{A}$-multiplicative with respect to both $\lVert\cdot\rVert$ and $\lVert\cdot\rVert'$ in Lemma \ref{Boundedness vs. continuity} nor the analogous assumption $\phi(\varpi)=\lVert\varpi\rVert$ in Lemma \ref{Boundedness vs. continuity, variant for rings} (nor the corresponding assumptions in the corollaries) can be omitted. Indeed, consider, as in Remark \ref{Caution 1}, a field $K$ equipped with a non-trivial absolute value $\vert\cdot\vert$ and define another absolute value by $\phi=\vert\cdot\vert^{s}$ for some $s\in (0, 1)$. The two absolute values define the same topology; however, if $\phi$ was bounded with respect to $\vert\cdot\vert$, then, by \cite{BGR}, Proposition 1.3.1/2, we would actually have $\phi\leq\vert\cdot\vert$. But this is clearly not true: For every $\varpi\in K$ with $\vert\varpi\vert<1$, we have $\phi(\varpi)=\vert\varpi\vert^{s}>\vert\varpi\vert$, as $s\in(0,1)$.\end{rmk}
We also record the following consequence of Corollary \ref{Boundedness vs. continuity 1.5}.
\begin{lemma}\label{Bounded open submodules}Let $(A, \lVert\cdot\rVert_{A})$ be a seminormed ring with a seminorm-multiplicative topologically nilpotent unit $\varpi$ and let $A_{0}$ be a bounded open subring of $A$ with $\varpi\in A_{0}$. Then for every seminormed $A$-module $(M, \lVert\cdot\rVert)$ and every bounded open $A_{0}$-submodule $M_{0}$ of $M$ with $M=M_{0}[\varpi^{-1}]$ the given seminorm $\lVert\cdot\rVert$ on $M$ is bounded-equivalent to the canonical extension $\lVert\cdot\rVert_{M,M_{0}}$ of the $\varpi$-adic seminorm on $M_{0}$.\end{lemma}
\begin{proof}We first prove the lemma in the case $M_{0}=M_{\leq1}$. By Corollary \ref{Boundedness vs. continuity 1.5} it suffices to show that the two seminorms $\lVert\cdot\rVert$ and $\lVert\cdot\rVert_{M,M_{\leq1}}$ define the same topology on $M$. This will follow if we show that \begin{equation}\varpi^{n}M_{\leq1}=\{\, x\in M\mid \lVert x\rVert\leq\lVert\varpi\rVert_{A}^{n}\,\},\end{equation}up to replacing $\lVert\cdot\rVert$ by some bounded-equivalent $A$-module seminorm. Using \cite{Bambozzi-Ben-Bassat16}, Proposition 3.28, we may assume that $\lVert\cdot\rVert$ satisfies $\lVert fx\rVert\leq \lVert f\rVert_{A}\lVert x\rVert$ for all $f\in A$ and $x\in M$. In particular, we see that the left hand side of equation (1) is contained in the right hand side. To prove the opposite inclusion, recall from Lemma \ref{Multiplicative topologically nilpotent units} that under our assumptions on $\lVert\cdot\rVert$ every invertible element of $A$ which is multiplicative with respect to the norm $\lVert\cdot\rVert_{A}$ on $A$ is also a multiplicative element for the seminorm $\lVert\cdot\rVert$ on $M$; in particular, this applies to the inverse $\varpi^{-1}$ of $\varpi$. Hence, if $x\in M$ satisfies $\lVert x\rVert\leq\lVert\varpi\rVert_{A}^{n}$, then \begin{equation*}\lVert\varpi^{-n}x\rVert=\lVert\varpi\rVert_{A}^{-n}\lVert x\rVert\leq1.\end{equation*}In other words, $\varpi^{-n}x\in M_{\leq1}$, proving the equality (1).

It remains to deduce from the above the assertion for general $M_{0}$. Since $M_{0}$ is open and bounded, there exist integers $m, n$ such that \begin{equation*}\varpi^{m}M_{\leq1}\subseteq M_{0}\subseteq \varpi^{n}M_{\leq1}\end{equation*}(we use here that the topology defined by $\lVert\cdot\rVert$ coincides with that defined by $\lVert\cdot\rVert_{M,M_{\leq1}}$, as we have shown in the last paragraph). It follows that the seminorms $\lVert\cdot\rVert_{M,M_{0}}$ and $\lVert\cdot\rVert_{M,M_{\leq1}}$ are bounded-equivalent.\end{proof} 

\section{Gauge seminorms}\label{sec:gauge seminorms}

We now introduce the concept of the gauge seminorm on a module over a seminormed ring $(A, \lVert\cdot\rVert)$ which generalizes the analogous concept from functionional analysis over nonarchimedean fields (see Schneider \cite{Schneider}, \S2) to arbitrary seminormed rings. Under some rather mild assumptions on $(A, \lVert\cdot\rVert)$, this will allow us to describe the seminorm on any seminormed $A$-module $M$ in very explicit terms. We first prove the following elementary lemma. Recall that we denote by $A^{\times,m}$ the group of seminorm-multiplicative units in a seminormed ring $(A, \lVert\cdot\rVert)$. For a seminormed abelian group $A$ and $r>0$ we denote by $A_{\leq r}$ the closed ball of radius $r$ in $A$.
\begin{lemma}\label{Lattices}Let $(A, \lVert\cdot\rVert)$ be a seminormed ring with an open subring $A_{0}$ of $A_{\leq1}$, let $M$ be an $A$-module and let $M_{0}$ be an $A_{0}$-submodule of $M$. For any finite family of elements $a_1,\dots, a_n\in A^{\times,m}$, there exists some $i\in\{1,\dots, n\}$ such that \begin{equation*}a_{1}M_{0}+\dots+a_{n}M_{0}\subseteq a_{i}(A_{\leq1}\cdot M_{0}).\end{equation*}In particular, for every $r>0$, the subgroup $(A^{\times,m})_{\leq r}\cdot M_{0}$ of $M$ is given by\begin{equation*}(A^{\times,m})_{\leq r}\cdot M_{0}=\bigcup_{a\in (A^{\times,m})_{\leq r}}a(A_{\leq1}\cdot M_{0}).\end{equation*}\end{lemma}
\begin{proof}We proceed by induction on $n\geq1$, the case $n=1$ being trivial. Suppose that the assertion holds for some $n\geq1$ and let $a_{1},\dots, a_{n+1}\in A^{\times,m}$. We may assume that $\lVert a_{1}\rVert=\max_{1\leq i\leq n+1}\lVert a_{i}\rVert$. By the induction hypothesis we may also assume that \begin{equation*}a_{2}M_{0}+\dots+a_{n+1}M_{0}\subseteq a_{2}(A_{\leq1}\cdot M_{0}).\end{equation*}Since $\lVert a_{2}\rVert\leq \lVert a_{1}\rVert$ and $a_{1}\in A^{\times,m}$, we have \begin{equation*}\lVert a_{1}^{-1}a_{2}\rVert\leq \lVert a_{1}^{-1}\rVert\lVert a_{2}\rVert=\lVert a_{1}\rVert^{-1}\lVert a_{2}\rVert\leq1,\end{equation*}so $a_{1}^{-1}a_{2}\in A_{\leq1}$. Then \begin{equation*}a_{1}^{-1}(a_{1}M_{0}+a_{2}(A_{\leq1}M_{0}))=M_{0}+a_{1}^{-1}a_{2}(A_{\leq1}\cdot M_{0})\subseteq M_{0}+A_{\leq1}\cdot M_{0}.\end{equation*}Multiplying by $a_{1}$, we conclude that $a_{1}M_{0}+a_{2}(A_{\leq1}\cdot M_{0})\subseteq a_{1}(A_{\leq1}\cdot M_{0})$, as desired.\end{proof}
The following definition generalizes Schneider's definition of the gauge seminorm of a $K^{\circ}$-lattice in a vector space $V$ over a nonarchimedean field $K$ (see \cite{Schneider}, \S2), see also Proposition \ref{Another definition of the gauge seminorm} for the comparison of the two definitions. 
\begin{mydef}[Gauge seminorm]For a seminormed ring $(A, \lVert\cdot\rVert)$, an $A$-module $M$ and a subset $M_{0}$ of $M$ such that \begin{equation*}M=\bigcup_{r>0}A_{\leq r}\cdot M_{0},\end{equation*}we define a function $\lVert\cdot\rVert_{M,M_{0},\ast}: M\to\mathbb{R}_{\geq0}$ on $M$ by \begin{equation*}\lVert x\rVert_{M,M_{0},\ast}=\inf\{\, r>0\mid x\in A_{\leq r}\cdot M_{0}\,\}\end{equation*}and call it the gauge, or gauge seminorm, of $M_{0}$ on $M$.\end{mydef}
\begin{lemma}\label{The gauge seminorm is a seminorm}Let $(A, \lVert\cdot\rVert)$ be a seminormed ring, let $M$ be an $A$-module and let $M_{0}$ be a subset of $M$ such that \begin{equation*}M=\bigcup_{r>0}A_{\leq r}\cdot M_{0}.\end{equation*}The gauge seminorm $\lVert\cdot\rVert_{M,M_{0},\ast}$ is indeed a seminorm and $(M, \lVert\cdot\rVert_{M,M_{0},\ast})$ is a submetric seminormed $(A, \lVert\cdot\rVert)$-module.\end{lemma}
\begin{proof}That $\lVert\cdot\rVert_{M,M_{0},\ast}$ is an abelian group seminorm is obvious since $A_{\leq r}\cdot M_{0}$, for any $r>0$, is by definition a subgroup of $M$. Thus we only have to prove that $\lVert fx\rVert_{M,M_{0},\ast}\leq \lVert f\rVert\lVert x\rVert_{M,M_{0},\ast}$ for all $x\in M$ and $f\in A$. Suppose first that $\lVert f\rVert=0$. Choose $r>0$ such that $x\in A_{\leq r}\cdot M_{0}$. Then \begin{equation*}fx\in A_{\leq\epsilon}A_{\leq r}\cdot M_{0}\subseteq A_{\leq\epsilon r}\cdot M_{0}\end{equation*}for all $\epsilon>0$. Consequently, $\lVert fx\rVert_{M,M_{0},\ast}=0$, which proves the claim in the case when $\lVert f\rVert=0$. Now suppose that $\lVert f\rVert\neq 0$. Suppose that $r>0$, $x\in M$, $f\in A$ are such that $\lVert x\rVert_{M,M_{0},\ast}<r\lVert f\rVert^{-1}$. Then there exists $r'\in (0,r)$ such that $x\in A_{\leq r'\lVert f\rVert^{-1}}\cdot M_{0}$. But then \begin{equation*}fx\in fA_{\leq r'\lVert f\rVert^{-1}}\cdot M_{0}\subseteq A_{\leq r'}\cdot M_{0},\end{equation*}so $\lVert fx\rVert_{M,M_{0},\ast}<r$, as desired.\end{proof}
\begin{example}\label{Discrete seminorms and gauges}If the underlying topological ring of $A$ is a Tate ring, $(A_{0}, \varpi)$ is a pair of definition of this Tate ring and the seminorm on $A$ is the canonical extension $\lVert\cdot\rVert_{A,A_{0}}$ of the $\varpi$-adic seminorm on $A_{0}$, then, for every $r>0$, we see that \begin{equation*}A_{\leq r}=A_{\leq \lVert\varpi^{n_{r}}\rVert}=\varpi^{n_{r}}A_{0},\end{equation*}where $n_{r}=\max\{ n\mid r\leq \lVert\varpi^{n}\rVert\,\}$. It follows that, for any $M$ as above and any open $A_{0}$-submodule $M_{0}$ of $M$, the gauge seminorm $\lVert\cdot\rVert_{M,M_{0},\ast}$ is equal to the canonical extension $\lVert\cdot\rVert_{M,M_{0}}$ of the $\varpi$-adic seminorm on $M_{0}$.\end{example}
By contrast, if the seminorm $\lVert\cdot\rVert$ on $A$ is such that the subgroup $\lVert A^{\times,m}\rVert$ of $\mathbb{R}_{\geq0}$ is not discrete (and hence dense in $\mathbb{R}_{\geq0}$), then it turns out that every seminorm on an $A$-module $M$ making $M$ into a submetric seminormed $A$-module is a gauge seminorm. We will deduce this fact from the following lemma. 
\begin{lemma}\label{Description of norms}Let $(A, \lVert\cdot\rVert)$ be a seminormed ring. Let $\mu: M\to \mathbb{R}_{\geq0}$ be a function on an $A$-module $M$ satisfying \begin{equation*}\mu(fx)\leq \lVert f\rVert\mu(x)\end{equation*}for all $f\in A^{\times,m}$, $x\in M$ (respectively, for all $f\in A$, $x\in M$) and set $M_{\mu\leq1}=\mu^{-1}([0,1])$. If for every $x\in M$ and every $s>\mu(x)$ the intersection \begin{equation*}(\mu(x), s)\cap\lVert A^{\times,m}\rVert\end{equation*}is non-empty, then \begin{equation*}\mu(x)=\inf\{\, \lVert f\rVert\mid f\in A^{\times,m}, x\in fM_{\mu\leq1}\,\}\end{equation*}for all $x\in M$ (respectively, $\mu(x)=\inf\{\, \lVert f\rVert\mid f\in A, x\in fM_{\mu\leq1}\,\}$ for all $x\in M$).\end{lemma}
\begin{proof}Let $x\in M$ be an arbitrary element. Let $s>0$ be such that there exists $f\in A^{\times,m}$ (respectively, $f\in A$) with $\lVert f\rVert<s$ and $x\in fM_{\mu\leq1}$. Then $\mu(x)\leq\lVert f\rVert<s$. As $s$ was arbitrary, this shows that \begin{equation*}\mu(x)\leq\inf\{\, \lVert f\rVert\mid f\in A^{\times,m}, x\in fM_{\mu\leq1}\,\}\end{equation*}(respectively, that $\mu(x)\leq \inf\{\, \lVert f\rVert\mid f\in A, x\in fM_{\mu\leq1}\,\}$). 

Conversely, suppose that $\mu(x)<s$. Since $(\mu(x), s)\cap \lVert A^{\times,m}\rVert$ is non-empty, there exists $f\in A^{\times,m}$ such that $\mu(x)<\lVert f\rVert<s$. Then $\mu(f^{-1}x)\leq\lVert f^{-1}\rVert\mu(x)=\lVert f\rVert^{-1}\mu(x)<1$ and, consequently, $x=f(f^{-1}x)\in fM_{\mu\leq1}$. It follows that\begin{equation*}\inf\{\, \lVert f\rVert\mid f\in A, x\in fM_{\mu\leq1}\,\}\leq\inf\{\, \lVert f\rVert \mid f\in A^{\times,m}, x\in fM_{\mu\leq1}\,\}\leq \mu(x),\end{equation*}as desired.\end{proof}
\begin{lemma}\label{Value groups and gauges}Let $(A, \lVert\cdot\rVert)$ be a seminormed ring, let $M$ be an $A$-module and let $M_{0}$ be a subset of $M$ with $M=\bigcup_{r>0}A_{\leq r}\cdot M_{0}$. Then for every $x\in M$ and every $s>\lVert x\rVert_{M,M_{0},\ast}$ the intersection \begin{equation*}(\lVert x\rVert_{M,M_{0},\ast}, s)\cap \lVert A^{\times,m}\rVert\end{equation*}is non-empty.\end{lemma}
\begin{proof}If $x\in A_{\leq r}\cdot M_{0}$ for some $r>0$, we can write $x$ as $x=f_{1}x_{1}+\dots+f_{n}x_{n}$ with $f_1,\dots, f_n\in A_{\leq r}$ and $x_1,\dots, x_n\in M_{0}$. Assuming that $\lVert f_{1}\rVert=\max_{1\leq i\leq n}\lVert f_{i}\rVert$, we see that $x\in A_{\leq\lVert f_{1}\rVert}\cdot M_{0}$. Thus the set $\{\, \lVert f\rVert\mid f\in A, x\in A_{\leq\lVert f\rVert}\cdot M_{0}\,\}$ is cofinal in $\{\, r>0\mid x\in A_{\leq r}\cdot M_{0}\,\}$. The assertion follows. \end{proof}
\begin{lemma}\label{Closed unit balls of gauges}Let $(M, \lVert\cdot\rVert_{M})$ be a submetric seminormed module over a seminormed ring $(A, \lVert\cdot\rVert)$. The closed unit ball $M_{\leq1}=M_{\lVert\cdot\rVert_{M}\leq1}$ of $(M, \lVert\cdot\rVert_{M})$ coincides with the closed unit ball of $(M, \lVert\cdot\rVert_{M,M_{\leq1},\ast})$.\end{lemma}
\begin{proof}If $\lVert x\rVert_{M,M_{\leq1},\ast}\leq1$, then for every $r>1$, there exists $r'\in (0,r)$ such that $x\in A_{\leq r'}\cdot M_{\leq1}$. Since $(M, \lVert\cdot\rVert_{M})$ is a submetric seminormed $A$-module, this entails that $\lVert x\rVert_{M}\leq r'<r$. Consequently, $\lVert x\rVert_{M}\leq1$. The converse implication is obvious from the definition of $\lVert\cdot\rVert_{M,M_{\leq1},\ast}$. \end{proof}
\begin{prop}\label{Description of norms 1}Let $(A, \lVert\cdot\rVert)$ be a seminormed ring and let $(M, \lVert\cdot\rVert_{M})$ be a seminormed $(A, \lVert\cdot\rVert)$-module such that for every $x\in M$ and $s>\lVert x\rVert_{M}$, the intersection \begin{equation*}(\lVert x\rVert_{M}, s)\cap \lVert A^{\times,m}\rVert\end{equation*}is non-empty. Then \begin{align*}\lVert x\rVert_{M}=\lVert x\rVert_{M,M_{\leq1},\ast}=\inf\{\, \lVert f\rVert\mid f\in A, x\in fM_{\leq1}\,\} \\=\inf\{\, \lVert f\rVert\mid f\in A^{\times,m}, x\in fM_{\leq1}\,\}.\end{align*}Moreover, for every function $\mu: A\to\mathbb{R}_{\geq0}$ such that \begin{equation*}\mu(fx)\leq\lVert f\rVert\mu(x)\end{equation*}for all $f\in A^{\times,m}$, $x\in M$, such that $(\mu(x), s)\cap \lVert A^{\times,m}\rVert$ is non-empty for all $x\in M$, $s>\mu(x)$, and such that the closed unit ball with respect to $\mu$ coincides with $M_{\leq1}$, we have \begin{equation*}\lVert\cdot\rVert_{M}=\mu.\end{equation*}In particular, if $\lVert A^{\times,m}\rVert$ is dense in $\mathbb{R}_{\geq0}$, then for every submetric seminormed $(A, \lVert\cdot\rVert)$-module $(M, \lVert\cdot\rVert_{M})$ we have $\lVert\cdot\rVert_{M}=\lVert\cdot\rVert_{M,M_{\leq1},\ast}$.\end{prop}
\begin{proof}Follows from Lemma \ref{Description of norms}, Lemma \ref{Value groups and gauges} and Lemma \ref{Closed unit balls of gauges}.\end{proof}
\begin{prop}\label{Criterion for being a seminorm}Let $(A, \lVert\cdot\rVert)$ be a seminormed ring with the property that $\lVert A^{\times,m}\rVert$ is dense in $\mathbb{R}_{\geq0}$. If $\mu: M\to\mathbb{R}_{\geq0}$ is a function on an $A$-module $M$ such that \begin{equation*}\mu(fx)\leq\lVert f\rVert\mu(x)\end{equation*}for all $f\in A^{\times,m}$, $x\in M$ (respectively, for all $f\in A$, $x\in M$) and such that the closed unit ball $M_{\mu\leq1}$ with respect to $\mu$ is an $A_{\leq1}$-submodule of $M$, then $(M, \mu)$ is a seminormed abelian group (respectively, a submetric seminormed $(A, \lVert\cdot\rVert)$-module).\end{prop}
\begin{proof}By Lemma \ref{Description of norms} and the assumption that $\lVert A^{\times,m}\rVert$ be dense in $\mathbb{R}_{\geq0}$, we have \begin{equation*}\mu(x)=\inf\{\, \lVert f\rVert\mid f\in A^{\times,m}, x\in fM_{\mu\leq1}\,\}.\end{equation*}In particular, $\mu'(x)\leq\mu(x)$, where $\mu': M\to\mathbb{R}_{\geq0}$ is the function defined by\begin{equation*}\mu'(x)=\inf\{\, r>0\mid x\in (A^{\times,m})_{\leq r}\cdot M_{\mu\leq1}\,\}, x\in M.\end{equation*}It is clear that $\mu'$ is an abelian group seminorm since the sets $(A^{\times,m})_{\leq r}\cdot M$ are subgroups of $M$, by definition. Therefore, to prove that $\mu$ is an abelian group seminorm, it suffices to prove that $\mu(x)\leq\mu'(x)$ for all $x\in M$. Let $s>0$ and $x\in M$ satisfy $\mu'(x)<s$. By definition, there exists some $r\in (0, s)$ such that $x\in (A^{\times,m})_{\leq r}\cdot M_{\mu\leq1}$. But then, by Lemma \ref{Lattices} and the hypothesis that $M_{\mu\leq1}$ is an $A_{\leq1}$-submodule of $M$, there exists some $f\in (A^{\times,m})_{\leq r}$ with the property that $x\in fM_{\mu\leq1}$. In particular, $\mu(x)<s$ and, since $s\in (\mu'(x), \infty)$ was arbitrary, we have shown that $\mu(x)\leq\mu'(x)$, as desired. \end{proof}
\begin{rmk}\label{Dense norms and topologically nilpotent units}Note that if for some seminormed ring $(A, \lVert\cdot\rVert)$ the subset $\lVert A^{\times,m}\rVert$ of $\mathbb{R}_{\geq0}$ is dense in $\mathbb{R}_{\geq0}$, then $A$ automatically contains a seminorm-multiplicative topologically nilpotent unit $\varpi$ with $\varpi\in A_{\leq1}$. Indeed, since $\lVert A^{\times,m}\rVert$ is dense in $\mathbb{R}_{\geq0}$, there exists $\varpi\in A^{\times,m}$ such that $\lVert\varpi\rVert<1$.\end{rmk}
If a unit $\varpi\in A^{\times}$ in a ring $A$ admits a compatible system of $p$-power roots $(\varpi^{1/p^{n}})_{n}$ in $A$, then, for all $n\geq1$, we have $\varpi^{1/p^{n}}\in A^{\times}$ with inverse $\varpi^{-1}\varpi^{(p^{n}-1)/p^{n}}$. In this case we denote the inverse of $\varpi^{1/p^{n}}$ by $\varpi^{-1/p^{n}}$. In this way we obtain well-defined elements $\varpi^{s}$ of $A$ for every $s\in\mathbb{Z}[1/p]$.
\begin{lemma}\label{Roots of norm-multiplicative elements}Let $(A, \lVert\cdot\rVert)$ be a seminormed ring and let $\varpi\in A^{\times,m}$. If $\varpi$ admits a compatible system of $p$-power roots $(\varpi^{1/p^{n}})_{n}$ in $A$ such that $\lVert\varpi^{s}\rVert=\lVert\varpi\rVert^{s}$ for all $s\in\mathbb{Z}[1/p]$, then $\varpi^{s}$ is also seminorm-multiplicative for every $s\in \mathbb{Z}[1/p]$.\end{lemma}
\begin{proof}For any $s\in\mathbb{Z}[1/p]$, we check: \begin{equation*}\lVert\varpi^{-s}\rVert=\lVert\varpi\rVert^{-s}=(\lVert\varpi\rVert^{s})^{-1}=\lVert\varpi^{s}\rVert^{-1}.\end{equation*}The assertion follows from this and \cite{Dine22}, Remark 2.13.\end{proof}
\begin{rmk}Note that the assumption on the seminorm of $\varpi^{s}$, $s\in\mathbb{Z}[1/p]$, in the above lemma is not automatic for a seminorm-multiplicative topologically nilpotent unit $\varpi$ admitting a compatible system of $p$-power roots. To see this, let $\varpi$ be a non-unit non-zero-divisor in a ring $A_{0}$ and endow $A=A_{0}[\varpi^{-1}]$ with the canonical extension $\lVert\cdot\rVert_{A,A_{0}}=\lVert\cdot\rVert_{A,A_{0},\varpi}$ of the $\varpi$-adic seminorm on $A_{0}$ and suppose that $\varpi$ admits a compatible system of $p$-power roots $(\varpi^{1/p^{n}})_{n}$ in $A_{0}$. Then $\varpi$ is a seminorm-multiplicative element of $(A, \lVert\cdot\rVert_{A,A_{0},\varpi})$, but $\varpi^{1/p}$ is not: Indeed, $\lVert\varpi^{1/p}\rVert_{A,A_{0},\varpi}=1$, $\lVert\varpi^{(p-1)/p}\rVert_{A,A_{0},\varpi}=1$ but \begin{equation*}\lVert\varpi^{1/p}\varpi^{(p-1)/p}\rVert_{A,A_{0},\varpi}=\lVert\varpi\rVert_{A,A_{0},\varpi}<1.\end{equation*}\end{rmk} 
With Lemma \ref{Roots of norm-multiplicative elements} in hand, we deduce the following consequence of Lemma \ref{Description of norms 1}.   
\begin{cor}\label{Description of norms 2}Let $(A, \lVert\cdot\rVert)$ be a seminormed ring and suppose that there exists a topologically nilpotent unit $\varpi\in A$ which admits a compatible system of $p$-power roots $(\varpi^{1/p^{n}})_{n}$ in $A$ such that $\lVert\varpi^{s}\rVert=\lVert \varpi\rVert^{s}$ for all $s\in \mathbb{Z}[1/p]$. Then, for every $(M, \lVert\cdot\rVert_{M})\in \Snm_{A}^{\leq1}$, we have \begin{equation*}\lVert\cdot\rVert_{M}=\lVert\cdot\rVert_{M,M_{\leq1},\ast}.\end{equation*}\end{cor}
\begin{proof}Since $\varpi$ is topologically nilpotent, there exists $n\geq1$ such that $\lVert\varpi^{n}\rVert<1$. Since $\lVert\varpi^{n}\rVert=\lVert\varpi\rVert^{n}$, this implies that $\lVert\varpi\rVert<1$. Thus the set \begin{equation*}\{\,\lVert \varpi\rVert^{s}\mid s\in \mathbb{Z}[1/p]\,\}\end{equation*}is dense in $\mathbb{R}_{\geq0}$. By the assumption on $\varpi$, this entails that the set \begin{equation*}\{\,\lVert\varpi^{s}\rVert \mid s\in\mathbb{Z}[1/p]\,\}\end{equation*}is dense in $\mathbb{R}_{\geq0}$. But the assumption on $\varpi$ also implies that $\lVert\varpi^{-1}\rVert=\lVert\varpi\rVert^{-1}$, so $\varpi$ is a seminorm-multiplicative element and thus, by Lemma \ref{Roots of norm-multiplicative elements}, \begin{equation*}\{\,\lVert\varpi^{s}\rVert \mid s\in\mathbb{Z}[1/p]\,\}\subseteq \lVert A^{\times,m}\rVert.\end{equation*}We conclude by Proposition \ref{Description of norms 1}. \end{proof}
\begin{cor}\label{Submetric maps}Let $(A, \lVert\cdot\rVert)$ be a seminormed ring and let $(M, \lVert\cdot\rVert_{M})$, $(N, \lVert\cdot\rVert_{N})$ be submetric seminormed $(A, \lVert\cdot\rVert)$-modules such that for every $x\in M$, $y\in N$ and every $s_{1}>\lVert x\rVert_{M}$, $s_{2}>\lVert y\rVert_{N}$ the intersections $(\lVert x\rVert_{M}, s_1)\cap \lVert A^{\times,m}\rVert$ and $(\lVert y\rVert_{N}, s_{2})\cap A^{\times,m}$ are non-empty. Then an $A$-module homomorphism $\varphi: M\to N$ is submetric if and only if $\varphi(M_{\leq1})\subseteq N_{\leq1}$.

In particular, if $\lVert A^{\times,m}\rVert$ is dense in $\mathbb{R}_{\geq0}$, then an $A$-module homomorphism \begin{equation*}\varphi: M\to N\end{equation*}between any two submetric seminormed $(A, \lVert\cdot\rVert)$-modules $(M, \lVert\cdot\rVert_{M})$, $(N, \lVert\cdot\rVert_{N})$ is submetric if and only if $\varphi(M_{\leq1})\subseteq N_{\leq1}$.\end{cor}
\begin{proof}This follows from the description of the seminorms $\lVert\cdot\rVert_{M}$, $\lVert\cdot\rVert_{N}$ on $M$, $N$ given in Proposition \ref{Description of norms 1}.\end{proof}
\begin{cor}\label{Isometric isomorphisms}Let $(A, \lVert\cdot\rVert)$ be a seminormed ring such that $\lVert A^{\times,m}\rVert$ is dense in $\mathbb{R}_{\geq0}$. Then an $A$-module homomorphism $\varphi: M\to N$ between submetric seminormed $(A, \lVert\cdot\rVert)$-modules $(M, \lVert\cdot\rVert_{M})$, $(N, \lVert\cdot\rVert_{N})$ is an isometric isomorphism if and only if it restricts to a bijection $M_{\leq1}\tilde{\to}N_{\leq1}$.\end{cor}
\begin{proof}This is again an immediate consequence of Proposition \ref{Description of norms 1}. \end{proof}
We can also apply Lemma \ref{Description of norms} and Proposition \ref{Description of norms 1} to the seminormed ring $(A, \lVert\cdot\rVert)$ itself, which gives us the following results.
\begin{cor}\label{The seminorm on a seminormed ring}Let $(A, \lVert\cdot\rVert)$ be a seminormed ring and suppose that the subgroup $\lVert A^{\times,m}\rVert$ is dense in $\mathbb{R}_{\geq0}$. Then \begin{align*}\lVert f\rVert=\inf\{\, \lVert g\rVert\mid g\in A^{\times,m}, f\in gA_{\leq1}\,\} \\ =\inf\{\, r>0\mid f\in (A^{\times,m})_{\leq r}\cdot A_{\leq1}\,\}.\end{align*}\end{cor}
\begin{proof}Regard $(A, \lVert\cdot\rVert)$ as a submetric seminormed module over itself and apply Proposition \ref{Description of norms 1} to it.\end{proof}
\begin{cor}\label{Bounded submultiplicative seminorms}Let $(A, \lVert\cdot\rVert)$ be a seminormed ring such that $\lVert A^{\times,m}\rVert$ is dense in $\mathbb{R}_{\geq0}$. If $\mu: A\to \mathbb{R}_{\geq0}$ is a submultiplicative function such that $\mu\leq\lVert\cdot\rVert$ and such that the closed unit ball $A_{\mu\leq1}$ with respect to $\mu$ is a subring of $A$, then $\mu$ is a ring seminorm on $A$, and any two such functions are equal if and only if their closed unit balls are equal.\end{cor}
\begin{proof}If $\mu: A\to\mathbb{R}_{\geq0}$ is submultiplicative and bounded above by $\lVert\cdot\rVert$, then it satisfies $\mu(fg)\leq \mu(f)\mu(g)\leq \lVert f\rVert\mu(g)$ for all $f, g\in A$. To see that any function $\mu$ as in the statement of the corollary is a ring seminorm, apply Proposition \ref{Criterion for being a seminorm} with $M=A$. To deduce from this that any two functions with these properties are equal if their closed unit balls are equal, apply Lemma \ref{Description of norms}, again with $M=A$.\end{proof}
\begin{example}[Berkovich spectrum]\label{Berkovich spectrum of a densely normed ring}Let $(A, \lVert\cdot\rVert)$ be a seminormed ring such that $\lVert A^{\times,m}\rVert$ is dense in $\mathbb{R}_{\geq0}$ (note that by Lemma \ref{Seminormed algebras and seminorm-multiplicative units 2} this assumption is satisfied, for example, for every seminormed $K$-algebra over a densely valued nonarchimedean field $K$). Then Corollary \ref{Bounded submultiplicative seminorms} implies that the Berkovich spectrum $\mathcal{M}((A, \lVert\cdot\rVert))$, defined as the space of multiplicative seminorms on $A$ bounded above by $\lVert\cdot\rVert$, with the subspace topology induced from the product topology on $\mathbb{R}_{\geq0}^{A}$, actually coincides with the space of all multiplicative functions $\mu: A\to\mathbb{R}_{\geq0}$ such that $\mu\leq\lVert\cdot\rVert$ and such that the closed unit ball $A_{\mu\leq1}$ is a subring of $A$. The corollary also shows that the Berkovich spectrum $\mathcal{M}((A, \lVert\cdot\rVert))$ is in canonical bijection with a subset of the set of overrings of $A_{\leq1}$ in $A$.\end{example}
In a similar vein as the above example, we obtain the following curious consequence of Corollary \ref{Bounded submultiplicative seminorms} which concerns nonarchimedean fields.
\begin{prop}\label{Submultiplicative functions on nonarchimedean fields}Let $(K, \vert\cdot\vert)$ be a nonarchimedean field such that $\vert K^{\times}\vert$ is dense in $\mathbb{R}_{\geq0}$. If $\mu: K\to\mathbb{R}_{\geq0}$ is a submultiplicative function with $\mu\leq\vert\cdot\vert$ such that the closed unit ball $K_{\mu\leq1}$ is a subring of $K$, then $\mu=\vert\cdot\vert$.\end{prop}
We first need the following special case, which is also true for discretely valued nonarchimedean fields and is well-known.
\begin{lemma}\label{Berkovich spectrum of a nonarchimedean field}The Berkovich spectrum $\mathcal{M}(K)$ of a nonarchimedean field $K$ consists of a single point.\end{lemma}
\begin{proof}Let $\vert\cdot\vert'\in \mathcal{M}(K)$, i.e. $\vert\cdot\vert'$ is an absolute value on $K$ bounded above by the given absolute value $\vert\cdot\vert$ of $K$. But then, for any $f\in K^{\times}$, we have $\vert f\rvert = \vert f^{-1}\rvert^{-1}\leq \vert f^{-1}\vert'^{-1}=\vert f\vert'$, proving that $\vert\cdot\vert'=\vert\cdot\vert$.\end{proof}
\begin{lemma}\label{Norms on a nonarchimedean field}Let $K$ be a nonarchimedean field, with absolute value $\vert\cdot\vert$. If $\lVert\cdot\rVert$ is a norm on $K$ which is bounded with respect to $\vert\cdot\vert$, then $\lVert\cdot\rVert=\vert\cdot\vert$.\end{lemma}
\begin{proof}By \cite{Berkovich}, Theorem 1.2.1, \begin{equation*}\mathcal{M}((K, \lVert\cdot\rVert))\neq \emptyset,\end{equation*}so there exists an absolute value $\vert\cdot\vert'$ on $K$ bounded by $\lVert\cdot\rVert$. But then $\vert\cdot\vert'$ is bounded with respect to $\vert\cdot\vert$. Applying Lemma \ref{Berkovich spectrum of a nonarchimedean field}, we see that $\vert\cdot\vert'=\vert\cdot\vert$ and thus that $\lVert\cdot\rVert=\vert\cdot\vert$.\end{proof}
\begin{proof}[Proof of Proposition \ref{Submultiplicative functions on nonarchimedean fields}]By Corollary \ref{Bounded submultiplicative seminorms}, we know that $\mu$ must be a ring seminorm on $K$. In this case the assertion is a special case of Lemma \ref{Norms on a nonarchimedean field}.\end{proof}   
For a seminormed ring $(A, \lVert\cdot\rVert)$, the seminorm $\lVert\cdot\rVert$ on $A$ can be described as a gauge seminorm in multiple ways, with no additional assumption on the values of $\lVert\cdot\rVert$.
\begin{lemma}\label{Seminormed rings and gauges}Let $(A, \lVert\cdot\rVert)$ be a seminormed ring. For every open subring $A_{0}$ of $A_{\leq1}$, the seminorm $\lVert\cdot\rVert$ coincides with $\lVert\cdot\rVert_{A,A_{0},\ast}$.\end{lemma}
\begin{proof}If $f\in A$ satisfies $\lVert f\rVert\leq r$ for some $r>0$, then $f\in A_{\leq r}\subseteq A_{\leq r}\cdot A_{0}$ and thus $\lVert f\rVert_{A,A_{0},\ast}\leq r$. It follows that $\lVert f\rVert_{A,A_{0},\ast}\leq\lVert f\rVert$ for all $f\in A$. Conversely, suppose that $\lVert f\rVert_{A,A_{0},\ast}<r$ for some $r>0$. Then there exists $r'\in (0,r)$ such that $f\in A_{\leq r'}\cdot A_{0}$. Since $A_{0}\subseteq A_{\leq1}$, it follows that $\lVert f\rVert\leq r'<r$.\end{proof}

\section{Main results on almost mathematics}\label{sec:almost mathematics}

The notion of the gauge seminorm of a general open $A_{\leq1}$-submodule $M_{0}$ of a submetric seminormed module $M$ is also intimately related to the following definition, which is inspired by almost mathematics.
\begin{mydef}[Module of almost elements]\label{Module of almost elements}Let $(A, \lVert\cdot\rVert)$ be a seminormed ring and let $A_{0}$ be an open and bounded subring of $A$. For an $A_{0}$-module $M_{0}$ such that the canonical map $M_{0}\hookrightarrow M=M_{0}\otimes_{A_{0}}A$ is injective, the $A_{0}$-module of almost elements of $M_{0}$ is the $A_{0}$-submodule of $M$ given by \begin{equation*}(M_{0})_{\ast}:=\bigcap_{r>1}A_{\leq r}\cdot M_{0}.\end{equation*}\end{mydef}
\begin{example}\label{Module of almost elements and roots}The above notation and the name 'module of almost elements' can be justified as follows. Suppose that $\varpi$ is a topologically nilpotent unit in $A$ which admits a compatible system of $p$-power roots $(\varpi^{1/p^{n}})_{n}$ in $A$ and that the norm on $A$ satisfies $\lVert\varpi^{s}\rVert=\lVert\varpi\rVert^{s}$ for all $s\in\mathbb{Z}[1/p]$ (in particular, $\lVert\varpi^{-1}\rVert=\lVert\varpi\rVert^{-1}$, so $\varpi$ is seminorm-multiplicative). Then, using Lemma \ref{Roots of norm-multiplicative elements}, we have \begin{equation*}(M_{0})_{\ast}=\bigcap_{r>1}A_{\leq r}\cdot M_{0}=\bigcap_{n\in\mathbb{Z}_{>0}}A_{\leq\lVert\varpi^{-1/p^{n}}\rVert}\cdot M_{0}=\bigcap_{n\in\mathbb{Z}_{>0}}\varpi^{-1/p^{n}}A_{\leq1}\cdot M_{0}\end{equation*}for every open subring $A_{0}$ of $A_{\leq1}$ and every $\varpi$-torsion-free $A_{0}$-module $M_{0}$. Choosing $A_{0}=A_{\leq1}$, we obtain \begin{equation*}(M_{0})_{\ast}=\bigcap_{n\in\mathbb{Z}_{>0}}\varpi^{-1/p^{n}}M_{0},\end{equation*}which is precisely the module of almost elements of the almost $A_{\leq1}$-module $M_{0}^{a}$ in the sense of almost mathematics (see Lemma \ref{Module of almost elements and functor of almost elements} below). However, we caution the reader that for general open subrings $A_{0}\subsetneq A_{\leq1}$ with $\varpi^{1/p^{n}}\in A_{0}$ for all $n\geq0$ the $A_{0}$-module $(M_{0})_{\ast}$ could be bigger than $(M_{0}^{a})_{\ast}$ (see Example \ref{Example where the module of almost elements is too big}).\end{example}
\begin{lemma}\label{Closed unit ball and module of almost elements}Let $(A, \lVert\cdot\rVert)$ be a seminormed ring and let $A_{0}$ be an open subring $A$ with $A_{0}\subseteq A_{\leq1}$. Let $M$ be an $A$-module and let $M_{0}$ be an $A_{0}$-submodule of $M$ such that \begin{equation*}M=\bigcup_{r>0}A_{\leq r}\cdot M_{0}.\end{equation*}Then $(M_{0})_{\ast}$ is the closed unit ball $M_{\leq1}$ of $(M, \lVert\cdot\rVert_{M,M_{0},\ast})$.\end{lemma}
\begin{proof}Immediate from the definitions.\end{proof}
\begin{lemma}\label{Gauge seminorms and lattices}Let $A_{0}$ be an open subring of a seminormed ring $(A, \lVert\cdot\rVert)$ contained in $A_{\leq1}$. Let $M$ be an $A$-module and let $M_{0}$ be an $A_{0}$-submodule of $M$ such that \begin{equation*}M=\bigcup_{r>0}A_{\leq r}\cdot M_{0}.\end{equation*}Then \begin{equation*}\lVert\cdot\rVert_{M,M_{0},\ast}=\lVert\cdot\rVert_{M,(M_{0})_{\ast},\ast}.\end{equation*}\end{lemma}
\begin{proof}Follows from Lemma \ref{Value groups and gauges}, Proposition \ref{Description of norms 1} and Lemma \ref{Closed unit ball and module of almost elements}.\end{proof}
\begin{cor}\label{Seminormed rings and rings of almost elements}For any seminormed ring $(A, \lVert\cdot\rVert)$ and any open subring $A_{0}$ of $A_{\leq1}$, we have $(A_{0})_{\ast}=A_{\leq1}$.\end{cor}
\begin{proof}Combine Lemma \ref{Closed unit ball and module of almost elements} and Lemma \ref{Gauge seminorms and lattices} with Lemma \ref{Seminormed rings and gauges}.\end{proof}
\begin{example}\label{Example where the module of almost elements is too big}Let $(A, \lVert\cdot\rVert)$ be a Banach ring with a topologically nilpotent unit $\varpi$ and suppose that $\varpi$ admits a compatible system of $p$-power roots in $A$ satisfying $\lVert\varpi^{s}\rVert=\lVert\varpi\rVert^{s}$ for all $s\in\mathbb{Z}[1/p]$. Then, by Corollary \ref{Seminormed rings and rings of almost elements},\begin{equation*}(\mathbb{Z}+(\varpi)_{A_{\leq1}})_{\ast}=A_{\leq1},\end{equation*}which can also be seen directly from Example \ref{Module of almost elements and roots}. Note that in this example \begin{equation*}A_{\leq1}\supsetneq\bigcap_{n>0}\varpi^{-1/p^{n}}(\mathbb{Z}+(\varpi)_{A_{\leq1}}).\end{equation*}\end{example}
We can now describe the category of submetric seminormed modules over any seminormed ring $A$ with the property that $\lVert A^{\times,m}\rVert$ is dense in $\mathbb{R}_{\geq0}$.
\begin{thm}\label{Seminormed modules and lattices}Let $(A, \lVert\cdot\rVert)$ be a seminormed ring and suppose that the subset $\lVert A^{\times,m}\rVert$ of $\mathbb{R}_{\geq0}$ is dense in $\mathbb{R}_{\geq0}$. Let $A_{0}$ be an open subring of $A$ with $A_{0}\subseteq A_{\leq1}$. The functor \begin{equation*}M\mapsto M_{\leq1}\end{equation*}is an equivalence of categories between the category $\Snm_{A}^{\leq1}$ of submetric seminormed $A$-modules and the full subcategory of the category of $A_{0}$-modules spanned by $A_{0}$-modules $M_{0}$ such that the canonical map \begin{equation*}M_{0}\to M_{0}\otimes_{A_{0}}A\end{equation*}is injective and $M_{0}=(M_{0})_{\ast}$, with a quasi-inverse given by \begin{equation*}M_{0}\mapsto (M_{0}\otimes_{A_{\leq1}}A, \lVert\cdot\rVert_{M_{0}\otimes_{A_{\leq1}}A,M_{0},\ast}).\end{equation*}\end{thm}
\begin{proof}For any $A_{0}$-module $M_{0}$ such that $M_{0}\to M_{0}\otimes_{A_{0}}A$ is injective, the closed unit ball for the gauge seminorm $\lVert\cdot\rVert_{M,M_{0},\ast}$ on $M=M_{0}\otimes_{A_{0}}A$ is equal to $(M_{0})_{\ast}$, by Lemma \ref{Closed unit ball and module of almost elements}. Conversely, if $(M, \lVert\cdot\rVert)$ is a submetric seminormed $A$-module, then, since the seminormed ring $A$ admits a topologically nilpotent unit $\varpi$ with $\varpi\in A_{\leq1}$ (Remark \ref{Dense norms and topologically nilpotent units}), we have $M=M_{\leq1}[\varpi^{-1}]=M_{\leq1}\otimes_{A_{0}}A$. By the assumption that $\lVert A^{\times,m}\rVert$ is dense in $\mathbb{R}_{\geq0}$ and Lemma \ref{Description of norms}, the seminorm $\lVert\cdot\rVert_{M}$ on $M$ coincides with the gauge seminorm $\lVert\cdot\rVert_{M,M_{\leq1},\ast}$. This shows that the two functors are indeed quasi-inverse to each other.\end{proof}
We also have the following analog of Theorem \ref{Seminormed modules and lattices} for Banach modules.
\begin{thm}\label{Banach modules and lattices}Let $(A, \lVert\cdot\rVert)$ be a Banach ring with an open subring $A_{0}$ of $A_{\leq1}$ and a topologically nilpotent unit $\varpi$ contained in $A_{0}$, and suppose that $\lVert A^{\times,m}\rVert$ is dense in $\mathbb{R}_{\geq0}$. The functor \begin{equation*}M\mapsto M_{\leq1}\end{equation*}is an equivalence of categories between the category of submetric Banach $A$-modules and the category of $\varpi$-torsion-free $\varpi$-adically complete $A_{0}$-modules $M_{0}$ with $M_{0}=(M_{0})_{\ast}$, with a quasi-inverse functor given by \begin{equation*}M_{0}\mapsto (M_{0}[\varpi^{-1}], \lVert\cdot\rVert_{M_{0}[\varpi^{-1}],M_{0},\ast}).\end{equation*}\end{thm}
\begin{proof}In view of Theorem \ref{Seminormed modules and lattices}, we only have to verify that a submetric seminormed $A$-module $M$ is complete if and only if its closed unit ball $M_{\leq1}$ is $\varpi$-adically complete. However, this follows from Lemma \ref{Bounded open submodules}.\end{proof}
If some topologically nilpotent unit $\varpi$ of $A$ admits a compatible system of $p$-power roots $(\varpi^{1/p^{n}})_{n}$, we can reformulate the above two theorems using the language of almost mathematics for the almost setup $(A_{0}, (\varpi^{1/p^{\infty}})_{A_{0}})$, where $(\varpi^{1/p^{\infty}})_{A_{0}}$ is the (flat) ideal of $A_{0}$ generated by the elements $\varpi^{1/p^{n}}$ for $n\geq1$. For this purpose, we need a notion of $\varpi$-torsion-free $A_{0}^{a}$-module which generalizes the notion of a flat $K^{\circ a}$-module in the classical case when $A=K$ is a nonarchimedean field (\cite{Gabber-Ramero}, Definition 2.4.4(i); cf.~also \cite{Scholze}, Lemma 5.3(i)). 
\begin{mydef}[$\varpi$-torsion-free almost modules]\label{Torsion-free almost modules}Let $(V, \mathfrak{m})$ be a basic setup for almost mathematics and let $\varpi\in\mathfrak{m}$ be a non-zero-divisor. Then a $V^{a}$-module $N$ is called $\varpi$-torsion-free if there exists a $\varpi$-torsion-free $V$-module $M_{0}$ such that $N\cong M_{0}^{a}$.\end{mydef}
Recall from \cite{Gabber-Ramero}, Proposition 2.2.14, that, in the setting of the above definition, the functor of almost elements \begin{equation*}N\mapsto \Hom_{V^{a}}(V^{a}, N)\end{equation*}is a right adjoint to the localization functor $M\mapsto M^{a}$ from the category  of $V$-modules to the category of $V^{a}$-modules and that the counit of this adjunction is an isomorphism. We are mostly interested in the case when \begin{equation*}(V, \mathfrak{m})=(A_{0}, (\varpi^{1/p^{\infty}})_{A_{0}}),\end{equation*}with $A_{0}\subseteq A_{\leq1}$ an open subring in a seminormed ring $A$ and with $\varpi\in A_{0}$ a topologically nilpotent unit of $A$ having a compatible system of $p$-power roots $\varpi^{1/p^{n}}$ in $A_{0}$. In this case, the same argument as in the proof of \cite{Scholze}, Lemma 5.3(ii), relates the functor $N\mapsto N_{\ast}$ to Definition \ref{Module of almost elements}.
\begin{lemma}\label{Module of almost elements and functor of almost elements}Let $(A, \lVert\cdot\rVert)$ be a seminormed ring, let $A_{0}$ be an open subring contained in $A_{\leq1}$ and let $\varpi\in A_{0}$ be a topologically nilpotent unit of $A$ which admits a compatible system of $p$-power roots $(\varpi^{1/p^{n}})_{n}$ in $A_{0}$. We perform almost mathematics with respect to the basic setup $(A_{0}, (\varpi^{1/p^{\infty}})_{A_{0}})$, where $(\varpi^{1/p^{\infty}})_{A_{0}}$ is the flat ideal \begin{equation*}(\varpi^{1/p^{\infty}})_{A_{0}}=\bigcup_{n>0}(\varpi^{1/p^{n}})_{A_{0}}.\end{equation*}Let $M_{0}$ be a $\varpi$-torsion-free $A_{0}$-module. Then \begin{equation*}\Hom_{A_{0}^{a}}(A_{0}^{a}, M_{0}^{a})=\bigcap_{n>0}\varpi^{-1/p^{n}}M_{0}.\end{equation*}\end{lemma}
\begin{proof}Since the ideal $(\varpi^{1/p^{\infty}})_{A_{0}}$ of $A_{0}$ is flat, the left hand side of the equation is equal to $\Hom_{A_{0}}((\varpi^{1/p^{\infty}})_{A_{0}}, M_{0})$. We prove that the $A_{0}$-linear map \begin{equation*}\Hom_{A_{0}}((\varpi^{1/p^{\infty}})_{A_{0}}, M_{0})\to M_{0}[\varpi^{-1}], \varphi\mapsto \varpi^{-1}\varphi(\varpi),\end{equation*}defines an isomorphism \begin{equation*}\Hom_{A_{0}}((\varpi^{1/p^{\infty}})_{A_{0}}, M_{0})\tilde{\to}\bigcap_{n>0}\varpi^{-1/p^{n}}M_{0}.\end{equation*}First of all, note that, for every $A_{0}$-linear map $\varphi: (\varpi^{1/p^{\infty}})_{A_{0}}\to M_{0}$ we have \begin{equation*}\varpi^{1/p^{n}}(\varpi^{-1}\varphi(\varpi))=\varpi^{-1}\varphi(\varpi^{1/p^{n}}\varpi)=\varpi^{-1}\varpi\varphi(\varpi^{1/p^{n}})=\varphi(\varpi^{1/p^{n}})\in M_{0},\end{equation*}so the map $\varphi\mapsto \varpi^{-1}\varphi(\varpi)$ indeed takes values in $\bigcap_{n>0}\varpi^{-1/p^{n}}M_{0}$. To show that this map is injective, assume that $\varphi$ is such that $\varpi^{-1}\varphi(\varpi)=0$ in $M_{0}[\varpi^{-1}]$. Then, since $M_{0}$ is $\varpi$-torsion-free, $\varphi(\varpi)=0$. By the $A_{0}$-linearity of $\varphi$, this entails that \begin{equation*}\varpi^{\frac{p^{n}-1}{p^{n}}}\varphi(\varpi^{1/p^{n}})=0\end{equation*}for all $n\geq 1$, so $\varphi=0$, using again that $M_{0}$ is $\varpi$-torsion-free. This shows that the map $\varphi\mapsto \varpi^{-1}\varphi(\varpi)$ is injective. On the other hand, for every $x\in\bigcap_{n>0}\varpi^{-1/p^{n}}M_{0}$ we can define an $A_{0}$-linear map $\varphi_{x}: (\varpi^{1/p^{\infty}})_{A_{0}}\to M_{0}$ by $\varpi^{1/p^{n}}\mapsto \varpi^{1/p^{n}}x$ and we see that $\varpi^{-1}\varphi_{x}(\varpi)=x$.\end{proof}
\begin{lemma}\label{Torsion-free almost modules and almost elements}Let $(V, \mathfrak{m})$ be a basic setup for almost mathematics and let $\varpi\in\mathfrak{m}$ be a non-zero-divisor. A $V^{a}$-module $N$ is $\varpi$-torsion-free if and only if $N_{\ast}$ is $\varpi$-torsion-free.\end{lemma}
\begin{proof}If $N_{\ast}$ is $\varpi$-torsion-free, then $N$ is $\varpi$-torsion-free since the adjunction morphism ${N_{\ast}}^{a}\to N$ is an isomorphism (\cite{Gabber-Ramero}, Proposition 2.2.14(iii)). Conversely, assume that $M_{0}$ is a $\varpi$-torsion-free $A_{0}$-module such that $M_{0}^{a}=N$. Let \begin{equation*}\varphi\in N_{\ast}=\Hom_{V}(\widetilde{\mathfrak{m}}, M_{0})\end{equation*}be any non-zero element. Then there exists an element $x\in\widetilde{m}=\mathfrak{m}\otimes_{V}\mathfrak{m}$ whose image $\varphi(x)$ in $M_{0}$ is a non-zero element of $M_{0}$. Since $M_{0}$ is $\varpi$-torsion-free, $\varpi^{k}\varphi(x)\neq0$ and thus $\varpi^{k}\varphi\neq0$, for all $k$.\end{proof}
\begin{lemma}\label{Two modules of almost elements}Let $(A, \lVert\cdot\rVert)$, $A_{0}$ and $\varpi$ be as in Lemma \ref{Module of almost elements and functor of almost elements}. Let $M_{0}$ be a $\varpi$-torsion-free $A_{0}$-module and let $N=M_{0}^{a}$. If $A_{0}=A_{\leq1}$, then $N_{\ast}=(M_{0})_{\ast}$, where the right hand side is the module of almost elements as defined in Definition \ref{Module of almost elements}. In general, $(N_{\ast})_{\ast}=(M_{0})_{\ast}$.\end{lemma}
\begin{proof}The assertion in the case $A_{0}=A_{\leq1}$ follows from Example \ref{Module of almost elements and roots} and Lemma \ref{Module of almost elements and functor of almost elements}. In general, Example \ref{Module of almost elements and roots} and Lemma \ref{Module of almost elements and functor of almost elements} imply that \begin{equation*}M_{0}\subseteq N_{\ast}\subseteq (M_{0})_{\ast}.\end{equation*}Since, by Lemma \ref{Closed unit ball and module of almost elements}, $(M_{0})_{\ast}$ is the closed unit ball of $M=M_{0}[\varpi^{-1}]$ with respect to the gauge seminorm $\lVert\cdot\rVert_{M,M_{0},\ast}$, we have \begin{equation*}((M_{0})_{\ast})_{\ast}=(M_{0})_{\ast}.\end{equation*}Hence the inclusions $M_{0}\subset N_{\ast}\subseteq (M_{0})_{\ast}$ imply that $(N_{\ast})_{\ast}=(M_{0})_{\ast}$.  \end{proof}
\begin{mydef}[$\varpi$-adically complete almost modules]\label{Complete almost modules}Let $(V, \mathfrak{m})$ be a basic setup for almost mathematics and let $\varpi\in\mathfrak{m}\setminus\{0\}$. A $V^{a}$-module $M$ is called $\varpi$-adically complete if the map \begin{equation*}M\to \widehat{M}=\varprojlim_{n}M/\varpi^{n}M\end{equation*}is an isomorphism of almost modules.\end{mydef}
\begin{lemma}\label{Completeness does not depend on choices}Let $(V, \mathfrak{m})$ be a basic setup for almost mathematics, let $\varpi\in\mathfrak{m}\setminus\{0\}$ and let $M$ be a $V^{a}$-module. The following are equivalent: \begin{enumerate}[(1)]\item $M$ is $\varpi$-adically complete. \item There exists a $V$-module $M_{0}$ with $M_{0}^{a}=M$ such that the completion map $M_{0}\to \widehat{M_{0}}$ is an almost isomorphism. \item For every $A_{0}$-module $M_{0}$ with $M_{0}^{a}=M$, the completion map $M_{0}\to\widehat{M_{0}}$ is an almost isomorphism.\end{enumerate}\end{lemma}
\begin{proof}Suppose that $M_{0}$ is an $A_{0}$-module and let $M=M_{0}^{a}$. By \cite{Gabber-Ramero}, Proposition 2.2.14 and Proposition 2.2.23, the localization functor $N_{0}\mapsto N_{0}^{a}$ has both a left adjoint and a right adjoint, so it preserves all limits and colimits. Hence $M_{0}\to \varprojlim_{n}M_{0}/\varpi^{n}M_{0}$ being an almost isomorphism is equivalent to $M\to \varprojlim_{n}M/\varpi^{n}M$ being an isomorphism of almost modules. \end{proof}
\begin{lemma}\label{Torsion-free complete almost modules}Let $(V, \mathfrak{m})$ be a basic setup and let $\varpi\in\mathfrak{m}$ be a non-zero-divisor. If a $V^{a}$-module $M$ is $\varpi$-torsion-free, then $M$ is $\varpi$-adically complete if and only if some (equivalently, every) $\varpi$-torsion-free $V$-module $M_{0}$ with $M_{0}^{a}=M$ is $\varpi$-adically complete.\end{lemma}
\begin{proof}One implication is trivial. For the other, assume that $M$ is $\varpi$-torsion-free and $\varpi$-adically complete and let $M_{0}$ be a $\varpi$-torsion-free $V$-module with $M_{0}^{a}=M$. We know from Lemma \ref{Completeness does not depend on choices} that the completion map $M_{0}\to \widehat{M_{0}}$ is an almost isomorphism, i.e., its kernel and cokernel are annihilated by every $p$-power root of $\varpi$. In particular, the kernel and cokernel are $\varpi$-torsion. Since $M_{0}$ is $\varpi$-torsion-free, the kernel must be zero. Since $M_{0}$ is $\varpi$-torsion-free and $M_{0}/\varpi M_{0}=\widehat{M_{0}}/\varpi\widehat{M_{0}}$, the cokernel is zero as well, so $M_{0}$ is $\varpi$-adically complete, as claimed.  \end{proof} 
We can now prove our first main result, which is a variant of Theorem \ref{Seminormed modules and lattices} (respectively, of Theorem \ref{Banach modules and lattices}) in the world of almost mathematics.
\begin{thm}\label{Seminormed modules and almost lattices}Let $(A, \lVert\cdot\rVert)$ be a seminormed ring (respectively, a Banach ring) and suppose that there exists a topologically nilpotent unit $\varpi\in A_{\leq1}$ of $A$ which admits a compatible system of $p$-power roots $(\varpi^{1/p^{n}})_{n}$ in $A_{\leq1}$ such that \begin{equation*}\lVert\varpi^{s}\rVert=\lVert\varpi\rVert^{s}\end{equation*}for all $s\in\mathbb{Z}[1/p]$. We perform almost mathematics with respect to the basic setup \begin{equation*}(A_{\leq1}, (\varpi^{1/p^{\infty}})_{A_{\leq1}})\end{equation*}given by the flat ideal \begin{equation*}(\varpi^{1/p^{\infty}})_{A_{\leq1}}=\bigcup_{n>0}(\varpi^{1/p^{n}})_{A_{\leq1}}\end{equation*}of $A_{\leq1}$. Then the functors \begin{equation*}N\mapsto (N_{\ast}[\varpi^{-1}], \lVert\cdot\rVert_{N_{\ast}[\varpi^{-1}],N_{\ast},\ast})\end{equation*}and \begin{equation*}M\mapsto M_{\leq1}^{a}\end{equation*}are quasi-inverse equivalences between the category of submetric seminormed $(A, \lVert\cdot\rVert)$-modules (respectively, submetric Banach $(A, \lVert\cdot\rVert)$-modules) and the category of $\varpi$-torsion-free $A_{\leq1}^{a}$-modules (respectively, $\varpi$-adically complete $\varpi$-torsion-free $A_{\leq1}^{a}$-modules).\end{thm}
\begin{proof}Let $N$ be a $\varpi$-torsion-free (respectively, $\varpi$-torsion-free and $\varpi$-adically complete) $A_{\leq1}^{a}$-module. By Definition \ref{Torsion-free almost modules} (respectively, by Lemma \ref{Torsion-free complete almost modules}), there exists a $\varpi$-torsion-free (respectively, $\varpi$-torsion-free and $\varpi$-adically complete) $A_{\leq1}$-module $M_{0}$ with $M_{0}^{a}=N$. By Lemma \ref{Two modules of almost elements}, $N_{\ast}=(M_{0})_{\ast}$, where the right hand side is the module of almost elements as defined in Definition \ref{Module of almost elements}. By Lemma \ref{Closed unit ball and module of almost elements}, $(M_{0})_{\ast}$ is equal to the closed unit ball $M_{\leq1}$ of $M=N_{\ast}[\varpi^{-1}]=M_{0}[\varpi^{-1}]$ with respect to the gauge seminorm $\lVert\cdot\rVert_{M,N_{\ast},\ast}$. Since the adjunction morphism ${N_{\ast}}^{a}\to N$ is an isomorphism of $A_{\leq1}^{a}$-modules (\cite{Gabber-Ramero}, Proposition 2.2.14(iii)), we conclude that there is a canonical isomorphism $N\cong M_{\leq1}^{a}$. Conversely, let $(M, \lVert\cdot\rVert_{M})$ be a seminormed (respectively, Banach) $A$-module. Then \begin{equation*}(M_{\leq1}^{a})_{\ast}=(M_{\leq1})_{\ast}=M_{\leq1},\end{equation*}so the equality $\lVert\cdot\rVert_{M}=\lVert\cdot\rVert_{M,(M_{\leq1}^{a})_{\ast},\ast}$ follows from Corollary \ref{Description of norms 2}.\end{proof}
\begin{cor}\label{Almost isometries}Let $(A, \lVert\cdot\rVert)$ be a seminormed ring (respectively, a Banach ring) and suppose that there exists a topologically nilpotent unit $\varpi\in A_{\leq1}$ of $A$ which admits a compatible system of $p$-power roots $(\varpi^{1/p^{n}})_{n}$ in $A_{\leq1}$ satisfying \begin{equation*}\lVert\varpi^{s}\rVert=\lVert\varpi\rVert^{s}\end{equation*}for all $s\in\mathbb{Z}[1/p]$. Then a submetric $A$-module homomorphism $\varphi: M\to N$ between submetric seminormed $A$-modules $M$, $N$ is an isometry if and only if the $\varpi^{\infty}$-torsion submodule of the cokernel of $M_{\leq1}\to N_{\leq1}$ is almost zero.\end{cor}
\begin{proof}Isometries are precisely the strict monomorphisms (or, equivalently, the regular monomorphisms) in $\Snm_{A}^{\leq1}$. Thus, by Theorem \ref{Seminormed modules and almost lattices}, $\varphi: M\to N$ is an isometry if and only if the corresponding morphism of $A_{\leq^1}^{a}$-modules $M_{\leq1}^{a}\to N_{\leq1}^{a}$ is a regular monomorphism in the category of $\varpi$-torsion-free $A_{\leq1}^{a}$-modules. In other words, $\varphi$ is an isometry if and only if $M_{\leq1}^{a}\to N_{\leq1}^{a}$ is the kernel of its cokernel in the category of $\varpi$-torsion-free $A_{\leq1}^{a}$-modules. Let $\varphi: M\to N$ be an injective submetric $A$-module homomorphism and let $P_{0}$ be the cokernel of $\varphi\vert_{M_{\leq1}}: M_{\leq1}\to N_{\leq1}$ in the category of all $A_{\leq1}$-modules. Since the localization functor $\Mod_{A_{\leq1}}\to \Mod_{A_{\leq1}^{a}}$ preserves limits and colimits, \begin{center}\begin{tikzcd}0\arrow{r} & M_{\leq1}^{a}\arrow{r} & N_{\leq1}^{a}\arrow{r} & P_{0}^{a}\arrow{r} & 0 \end{tikzcd}\end{center}is a short exact sequence in the abelian category of $A_{\leq1}^{a}$-modules. Therefore, the monomorphism $M_{\leq1}^{a}\to N_{\leq1}^{a}$ is a regular monomorphism in the full subcategory of $\Mod_{A_{\leq1}^{a}}$ consisting of $\varpi$-torsion-free $A_{\leq1}^{a}$-modules if and only if $P_{0}^{a}$ is a $\varpi$-torsion-free $A_{\leq1}^{a}$-module. But this is the case if and only if the quotient map $P_{0}\to P_{0}/(\varpi^{\infty}-\tor)$ is an almost isomorphism.\end{proof}
We continue our study of gauge seminorms, turning our attention to more general open $A_{\leq1}$-submodules of a submetric seminormed module $M$, not neccessarily equal to the closed unit ball of $M$.
\begin{lemma}\label{Module of almost elements and value groups}Let $(A, \lVert\cdot\rVert)$ be a seminormed ring such that $\lVert A^{\times,m}\rVert$ is dense in $\mathbb{R}_{\geq0}$, let $A_{0}$ be an open subring of $A_{\leq1}$ and let $M_{0}$ be an $A_{0}$-submodule of an $A$-module $M$. Then \begin{equation*}(M_{0})_{\ast}=\bigcap_{r>1}(A^{\times,m})_{\leq r}\cdot A_{\leq1}\cdot M_{0}.\end{equation*}\end{lemma}
\begin{proof}Let $x\in (M_{0})_{\ast}$. Let $r>1$. We want to prove that $x\in (A^{\times,m})_{\leq r}\cdot A_{\leq1}M_{0}$. Let $r'\in (0,r)$ and write $x=f_{1}'x_{1}+\dots+f_{n}'x_{n}$ for some $x_{1},\dots, x_{n}\in M_{0}$, $f_{1}',\dots, f_{n}'\in A_{\leq r'}$. By Corollary \ref{The seminorm on a seminormed ring}, there exist $f_{1},\dots,f_{n}\in (A^{\times,m})_{\leq r}$ such that $f_{i}'\in f_{i}A_{\leq1}$ for all $i=1,\dots, n$. But then $x\in f_{1}A_{\leq1}M_{0}+\dots+f_{n}A_{\leq1}M_{0}\subseteq (A^{\times,m})_{\leq r}\cdot A_{\leq1}\cdot M_{0}$, as desired. \end{proof}
We deduce from the above lemma, Lemma \ref{Lattices} and Lemma \ref{Description of norms} that our definition of the gauge seminorm generalizes the definition of the gauge seminorm of a $K^{\circ}$-lattice in a vector space over a nonarchimedean field $K$ given by Schneider in \cite{Schneider}, \S2.
\begin{prop}\label{Another definition of the gauge seminorm}Let $(A, \lVert\cdot\rVert)$ be a seminormed ring such that $\lVert A^{\times,m}\rVert$ is dense in $\mathbb{R}_{\geq0}$, let $M$ be an $A$-module and let $M_{0}$ be an $A_{0}$-submodule of $M$ with $M=\bigcup_{r>0}A_{\leq r}\cdot M_{0}$, for some open subring $A_{0}\subseteq A_{\leq1}$. Then \begin{equation*}\lVert x\rVert_{M,M_{0},\ast}=\inf\{\, \lVert f\rVert \mid f\in A^{\times,m}, x\in fA_{\leq1}M_{0}\,\}\end{equation*}for all $x\in M$.\end{prop}
\begin{proof}By Lemma \ref{Gauge seminorms and lattices} and Lemma \ref{Closed unit ball and module of almost elements}, $\lVert\cdot\rVert_{M,M_{0},\ast}=\lVert\cdot\rVert_{M,(M_{0})_{\ast},\ast}$ and $(M_{0})_{\ast}$ is the closed unit ball of $(M, \lVert\cdot\rVert_{M,M_{0},\ast})$. By Lemma \ref{Value groups and gauges}, $\lVert M\rVert_{M,M_{0},\ast}$ is contained in the closure of $\lVert A^{\times,m}\rVert$. Hence, by Lemma \ref{Description of norms}, it suffices to prove that the functions $\lVert\cdot\rVert_{M,M_{0},\ast}$ and \begin{equation*}\mu: M\to\mathbb{R}_{\geq0}, x\mapsto \inf\{\, \lVert f\rVert\mid f\in A^{\times,m}, x\in fA_{\leq1}M_{0}\,\},\end{equation*}have the same closed unit ball. By inspection, the closed unit ball $\mu^{-1}([0,1])$ of $\mu$ is given by \begin{equation*}\bigcap_{r>1}\bigcup_{f\in (A^{\times,m})_{\leq1}}fA_{\leq1}M_{0}.\end{equation*}By Lemma \ref{Lattices}, this is equal to \begin{equation*}\bigcap_{r>1}(A^{\times,m})_{\leq r}\cdot A_{\leq1}\cdot M_{0}.\end{equation*}But by Lemma \ref{Module of almost elements and value groups} the latter subset of $M$ is equal to the closed unit ball $(M_{0})_{\ast}$ of $(M, \lVert\cdot\rVert_{M,M_{0},\ast})$.\end{proof}
\begin{cor}\label{Gauge seminorm via roots}Let $(A, \lVert\cdot\rVert)$ be a seminormed ring, let $A_{0}\subseteq A_{\leq1}$ be an open subring and let $\varpi\in A_{0}$ be a seminorm-multiplicative topologically nilpotent in $A$ admitting a compatible system of $p$-power roots $(\varpi^{1/p^{n}})_{n}$ in $A_{0}$ such that $\lVert\varpi^{s}\rVert=\lVert\varpi\rVert^{s}$ for all $s\in\mathbb{Z}[1/p]$. Let $(M, \lVert\cdot\rVert_{M})$ be a submetric seminormed $(A, \lVert\cdot\rVert)$-module and let $M_{0}$ be an open $A_{0}$-submodule such that \begin{equation*}\bigcap_{n>0}\varpi^{-1/p^{n}}M_{0}=M_{\leq1}.\end{equation*}Then \begin{equation*}\lVert x\rVert_{M}=\inf\{\, \lVert\varpi\rVert^{s}\mid s\in\mathbb{Z}[1/p], x\in\varpi^{s}M_{0}\,\}\end{equation*}for all $x\in M$.\end{cor}
\begin{proof}Consider $(A, \lVert\cdot\rVert)$ as a submetric seminormed $\mathbb{Z}(T^{1/p^{\infty}})$-algebra, where \begin{equation*}\mathbb{Z}(T^{1/p^{\infty}})=\mathbb{Z}[T^{1/p^{\infty}}][T^{-1}]\end{equation*}is endowed with the seminorm \begin{equation*}\lVert \sum_{i}n_{i}T^{s_{i}}\rVert=\max\{\, \lVert\varpi\rVert^{s_{i}}\mid n_{i}\neq 0\,\}\end{equation*}and where the map $\mathbb{Z}(T^{1/p^{\infty}})\to A$ sends $T$ to $\varpi$. Then $(M, \lVert\cdot\rVert_{M})$ is also a submetric seminormed $\mathbb{Z}(T^{1/p^{\infty}})$-module. Note that any $h=\sum_{i}n_{i}\varpi^{s_{i}}\in \mathbb{Z}(\varpi^{1/p^{\infty}})$ satisfies \begin{equation*}h\in \varpi^{s_{j}}\mathbb{Z}[\varpi^{1/p^{\infty}}],\end{equation*}where $j$ is such that $s_{j}$ is minimal among all $s_{i}$ (equivalently, such that $\lVert\varpi\rVert^{s_{j}}$ is maximal among $\lVert\varpi\rVert^{s_{i}}$). Thus, using Proposition \ref{Description of norms 1}, Lemma \ref{Gauge seminorms and lattices} and Proposition \ref{Another definition of the gauge seminorm} we obtain: \begin{align*}\lVert x\rVert=\inf\{\, \lVert h\rVert\mid h\in \mathbb{Z}(\varpi^{1/p^{\infty}}), x\in h\mathbb{Z}[\varpi^{1/p^{\infty}}]M_{0}\,\} \\ =\inf\{\, \lVert\varpi\rVert^{s}\mid s\in\mathbb{Z}[1/p], x\in \varpi^{s}M_{0}\,\},\end{align*}as desired. \end{proof}
\begin{cor}\label{Ring seminorm via roots}Let $(A, \lVert\cdot\rVert)$ be a seminormed ring, let $A_{0}$ be an open subring of $A_{\leq1}$ and let $\varpi\in A_{0}$ be a seminorm-multiplicative topologically nilpotent unit in $A$ admitting a compatible system of $p$-power roots $(\varpi^{1/p^{n}})_{n}$ in $A$ such that $\lVert\varpi^{s}\rVert=\lVert\varpi\rVert^{s}$ for all $s\in\mathbb{Z}[1/p]$. If $\bigcap_{n>0}\varpi^{-1/p^{n}}A_{0}=A_{\leq1}$, then \begin{equation*}\lVert f\rVert=\inf\{\, \lVert\varpi\rVert^{s}\mid s\in\mathbb{Z}[1/p], f\in \varpi^{s}A_{0}\,\}\end{equation*}for all $f\in A$.\end{cor}
\begin{proof}Apply the previous corollary to $(M, \lVert\cdot\rVert_{M})=(A, \lVert\cdot\rVert)$. \end{proof}
\begin{cor}\label{Infimum of discrete norms}Let $(A, \lVert\cdot\rVert)$ be a seminormed ring, let $A_{0}\subseteq A_{\leq1}$ be an open subring and let $\varpi\in A_{0}$ be a seminorm-multiplicative topologically nilpotent in $A$ admitting a compatible system of $p$-power roots $(\varpi^{1/p^{n}})_{n}$ in $A_{0}$ such that $\lVert\varpi^{s}\rVert=\lVert\varpi\rVert^{s}$ for all $s\in\mathbb{Z}[1/p]$. Let $(M, \lVert\cdot\rVert_{M})$ be a submetric seminormed $(A, \lVert\cdot\rVert)$-module and let $M_{0}$ be an open $A_{0}$-submodule such that $\bigcap_{n>0}\varpi^{-1/p^{n}}M_{0}=M_{\leq1}$. Then \begin{equation*}\lVert\cdot\rVert_{M}=\inf\lVert\cdot\rVert_{M,M_{0},\varpi^{1/p^{n}}}.\end{equation*}\end{cor}
\begin{proof}Follows immediately from Corollary \ref{Gauge seminorm via roots}.\end{proof}
We now want to establish analogs of Theorem \ref{Seminormed modules and lattices}, Theorem \ref{Banach modules and lattices} and Theorem \ref{Seminormed modules and almost lattices} for categories of seminormed algebras and uniform seminormed algebras. The key observation in this respect is captured in the following proposition.
\begin{prop}\label{Gauge seminorms of subrings}Let $(A, \lVert\cdot\rVert)$ be a seminormed ring such that $\lVert A^{\times,m}\rVert$ is dense in $\lVert A\rVert$. Let $B$ be an $A$-algebra and let $B_{0}$ be a subring of $B$ such that $B=\bigcup_{r>0}A_{\leq r}\cdot B_{0}$. Then the gauge $\lVert\cdot\rVert_{B,B_{0},\ast}$ is a ring seminorm on $B$.\end{prop}
\begin{proof}We have to prove that \begin{equation*}\lVert fg\rVert_{B,B_{0},\ast}\leq \lVert f\rVert_{B,B_{0},\ast}\lVert g\rVert_{B,B_{0},\ast}\end{equation*}for all $f, g\in B\setminus\{0\}$. Using Proposition \ref{Another definition of the gauge seminorm}, we compute: \begin{align*}\lVert f\rVert_{B,B_{0},\ast}\lVert g\rVert_{B,B_{0},\ast}=\inf\{\, \lVert h\rVert\mid h\in A^{\times,m}, f\in hB_{0}\,\}\cdot \inf\{\, \lVert h'\rVert\mid h'\in A^{\times,m}, g\in h'B_{0}\,\} \\ =\inf\{\, \lVert h\rVert\lVert h'\rVert \mid h, h'\in A^{\times,m}, f\in hB_{0}, g\in h'B_{0}\,\} \\ \geq\inf\{\, \lVert hh'\rVert\mid h, h'\in A^{\times,m}, f\in hB_{0}, g\in h'B_{0}\,\} \\ \geq\inf\{\, \lVert hh'\rVert\mid h, h'\in A^{\times,m}, fg\in hh'B_{0}\,\} \\ \geq\inf\{\, \lVert h\rVert\mid h\in A^{\times,m}, fg\in hB_{0}\,\}=\lVert fg\rVert_{B,B_{0},\ast}.\end{align*}\end{proof}
\begin{cor}\label{Module of almost elements is a ring}Let $(A, \lVert\cdot\rVert)$ be a seminormed ring such that $\lVert A^{\times,m}\rVert$ is dense in $\lVert A\rVert$. Let $B$ be an $A$-algebra and let $B_{0}$ be a subring of $B$ such that $B=\bigcup_{r>0}A_{\leq r}\cdot B_{0}$. Then the module of almost elements $(B_{0})_{\ast}$ of $B_{0}$ is a subring of $B$.\end{cor}
\begin{proof}Follows from Lemma \ref{Closed unit ball and module of almost elements} and Proposition \ref{Gauge seminorms of subrings}.\end{proof}
\begin{prop}\label{Gauges of valuation rings}Let $(A, \lVert\cdot\rVert)$ be a seminormed ring such that $\lVert A^{\times,m}\rVert$ is dense in $\mathbb{R}_{\geq0}$ and let $\varpi$ be a seminorm-multiplicative topologically nilpotent unit of $A$. Let $B$ be an $A$-algebra and let $B_{0}$ be a $\varpi$-adically separated subring of $B$ with $B=B_{0}[\varpi^{-1}]$. If $B_{0}$ is a valuation ring, then $B$ is a field and $\lVert\cdot\rVert_{B,B_{0},\ast}$ is an absolute value on $B$. In particular, if $B_{0}$ is a $\varpi$-adically complete valuation ring, then $(B, \lVert\cdot\rVert_{B,B_{0},\ast})$ is a nonarchimedean field.\end{prop}
\begin{proof}By \cite{FK}, Ch.~0, Proposition 6.7.2, $B=B_{0}[\varpi^{-1}]$ is a field. To prove that the gauge of $B_{0}$ is an absolute value (i.e., is multiplicative), we have to show that for every $x\in B$ we have $\lVert x^{-1}\rVert_{B,B_{0},\ast}\leq \lVert x\rVert_{B,B_{0},\ast}^{-1}$. By Lemma \ref{Gauge seminorms and lattices}, $\lVert\cdot\rVert_{B,B_{0},\ast}=\lVert\cdot\rVert_{B,(B_{0})_{\ast},\ast}$. Suppose that $\lVert x\rVert_{B,B_{0},\ast}^{-1}<r$ for some $r\in \lVert A^{\times,m}\rVert$. This means that $x\not\in fB_{0}$ for every $f\in A$ with $\lVert f\rVert\leq r^{-1}$. In particular, $f^{-1}x\not\in B_{0}$ for every $f\in A^{\times,m}$ with $\lVert f\rVert\leq r^{-1}$. Since $B_{0}$ is a valuation ring with fraction field $B$, this implies that $fx^{-1}\in B_{0}$ for all $f\in A^{\times,m}$ with $\lVert f\rVert\leq r^{-1}$. But for $f\in A^{\times,m}$ the condition $\lVert f\rVert\leq r^{-1}$ is equivalent to $\lVert f^{-1}\rVert\geq r$. Thus we can interpret the above property of $x$ as saying that $x^{-1}\in fB_{0}$ for all $f\in A^{\times,m}$ with $\lVert f\rVert\geq r$. Since we assumed that $r\in\lVert A^{\times,m}\rVert$, this means that, in particular, there exists $f\in A^{\times,m}$ with $\lVert f\rVert=r$ and $x^{-1}\in fB_{0}$. Therefore, $\lVert x^{-1}\rVert_{B,B_{0},\ast}\leq r$. Since $\lVert A^{\times,m}\rVert$ is dense in $\mathbb{R}_{\geq0}$ and $r\in\lVert A^{\times,m}\rVert$ was arbitrary, we conclude that $\lVert x^{-1}\rVert_{B,B_{0},\ast}\leq \lVert x\rVert_{B,B_{0},\ast}^{-1}$, as desired.\end{proof}
\begin{lemma}[\cite{Dine22}, Lemma 2.24]\label{Power-bounded}Let $(A, \lVert\cdot\rVert)$ be a seminormed ring with power-multiplicative seminorm. Then $A^{\circ}=A_{\leq1}$ and $A^{\circ\circ}=A_{<1}$.\end{lemma}
\begin{rmk}As a word of caution, let us note that for a general seminormed ring $(A, \lVert\cdot\rVert)$ the subring of power-bounded elements $A^{\circ}$ need not coincide with the closed unit ball $A_{\vert\cdot\vert_{\spc}\leq1}$ of $(A, \vert\cdot\vert_{\spc})$. See \cite{Andre18}, Remarque 2.3.3(2), for an example.\end{rmk}
\begin{lemma} \label{Completely integrally closed} Let $A$ be a Tate ring (for example, a seminormed ring which contains a topologically nilpotent unit). If $A_{0}$ is a ring of definition of $A$ which is completely integrally closed in $A$, then $A_{0}$ coincides with the ring of power-bounded elements $A^{\circ}$ of $A$. Conversely, if the Tate ring $A$ is uniform (that is, $A^{\circ}$ is bounded), then $A^{\circ}$ is completely integrally closed in $A$.\end{lemma}
\begin{proof}Let $\varpi\in A$ be a topologically nilpotent unit and let $f \in A^{\circ}$ be a power-bounded element. Then there exists some integer $m \geq 0$ such that $f^{n} \in \varpi^{-m}A_{0}$ for all $n$. But this means that $f$ is contained in the complete integral closure of $A_{0}$ inside $A$. Conversely, suppose that $A$ is uniform and that $f\in A$ is almost integral over $A^{\circ}$, i.e., $f^{n}$, $n\geq1$, lies in a finitely generated $A^{\circ}$-submodule $M$ of $A$. Let $g_1,\dots, g_{m}$ be generators of $M$. Choose $k\geq0$ such that $\varpi^{k}g_{1},\dots, \varpi^{k}g_{m}\in A^{\circ}$. Then $f^{n}\in M\subseteq \varpi^{-k}A^{\circ}$ for all $n$. Since $A^{\circ}$ is a ring of definition of $A$ (being open and bounded), this means that $f$ is a power-bounded element.\end{proof}
\begin{lemma}\label{Power-multiplicative}Let $(A, \lVert\cdot\rVert)$ be a seminormed ring. The seminorm $\lVert\cdot\rVert$ is power-multiplicative if and only if there exists an integer $k\geq2$ such that $\lVert f^{k}\rVert=\lVert f\rVert^{k}$ for all $f\in A$.\end{lemma}
\begin{proof}If there exists an integer $k\geq2$ as above, then $\lVert f^{k^{n}}\rVert=\lVert f\rVert^{k^{n}}$ for all $n\geq1$. But $(k^{n})_{n}$ is a cofinal subset of $\mathbb{Z}_{>1}$, so $\vert f\vert_{\spc}=\inf_{n}\lVert f^{n}\rVert^{1/n}=\inf_{n}\lVert f^{k^{n}}\rVert^{1/k^{n}}=\lVert f\rVert$.\end{proof}
\begin{lemma}\label{Power-multiplicative gauges}Let $(A, \lVert\cdot\rVert)$ be a seminormed ring such that $\lVert A^{\times,m}\rVert$ is dense in $\mathbb{R}_{\geq0}$, let $B$ be an $A$-algebra and let $B_{0}$ be a subring of $B$ with $B=\bigcup_{r>0}A_{\leq r}\cdot B_{0}$. Suppose that there exists an integer $k\geq2$ such that $f^{k}\in B_{0}$ implies $f\in B_{0}$ for every element $f\in B$. Then the seminorm $\lVert\cdot\rVert_{B,B_{0},\ast}$ is power-multiplicative. In particular, if $B_{0}$ is integrally closed in $B$, then $\lVert\cdot\rVert_{B,B_{0},\ast}$ is power-multiplicative. \end{lemma}
\begin{proof}We proceed as in the proof of \cite{Mihara}, Lemma 1.16. It suffices to prove the inequality $\lVert f\rVert_{B,B_{0},\ast}^{k}\leq\lVert f^{k}\rVert_{B,B_{0},\ast}$ for all $f\in B$, the opposite inequality being a consequence of Proposition \ref{Gauge seminorms of subrings}. The case $f\neq0$ is trivial, so we have to prove the above inequality for $f\neq0$. Suppose that $\lVert f^{k}\rVert_{B,B_{0},\ast}<r$ for some $r>0$. Since $\lVert A^{\times,m}\rVert$ is dense in $\mathbb{R}_{\geq0}$, we can choose $g\in A^{\times,m}$ with \begin{equation*}\lVert g\rVert\in [\lVert f^{k}\rVert_{B,B_{0},\ast}^{1/k}, r^{1/k}).\end{equation*}Since $(B, \lVert\cdot\rVert_{B,B_{0},\ast})$ is a submetric seminormed $(A, \lVert\cdot\rVert)$-module, $g^{-1}\in A^{\times,m}$ is multiplicative with respect to the seminorm $\lVert\cdot\rVert_{B,B_{0},\ast}$ on $B$, by Lemma \ref{Multiplicative topologically nilpotent units}. Thus \begin{equation*}\lVert (g^{-1}f)^{k}\rVert_{B,B_{0},\ast}=\lVert g\rVert^{-k}\lVert f^{k}\rVert_{B,B_{0},\ast}\leq1,\end{equation*}so $(g^{-1}f)^{k}\in B_{0}$. By the assumption on $B_{0}$, this means that $g^{-1}f\in B_{0}$ and, a fortiori, $\lVert g^{-1}f\rVert_{B,B_{0},\ast}\leq1$. Since $g$ is multiplicative with respect to $\lVert\cdot\rVert_{B,B_{0},\ast}$, this means that \begin{equation*}\lVert f\rVert_{B,B_{0},\ast}\leq\lVert g\rVert_{B,B_{0},\ast}<r^{1/k}.\end{equation*}Since $r>\lVert f^{k}\rVert_{B,B_{0},\ast}$ was arbitrary, this proves that $\lVert f\rVert_{B,B_{0},\ast}^{k}\leq\lVert f^{k}\rVert_{B,B_{0},\ast}$. \end{proof}
\begin{cor}\label{Module of almost elements of a root closed subring}Let $(A, \lVert\cdot\rVert)$ be a seminormed ring such that $\lVert A^{\times,m}\rVert$ is dense in $\mathbb{R}_{\geq0}$, let $B$ be an $A$-algebra and let $B_{0}$ be a subring of $B$. If there exists some integer $k\geq2$ such that for $f\in B$ the condition $f^{k}\in B_{0}$ implies $f\in B_{0}$, then $(B_{0})_{\ast}$ is a completely integrally closed subring of $B$. In particular, if $B_{0}$ is an integrally closed subring of $B$, then $(B_{0})_{\ast}$ is completely integrally closed in $B$.\end{cor}
\begin{proof}By Lemma \ref{Closed unit ball and module of almost elements} and Lemma \ref{Power-multiplicative gauges}, $(B_{0})_{\ast}$ is the closed unit ball of a power-multiplicative seminorm on $B$. By Lemma \ref{Power-bounded}, $(B_{0})_{\ast}$ is the subring of power-bounded elements of $B$ when $B$ is endowed with this seminorm. By Lemma \ref{Completely integrally closed} this means that $(B_{0})_{\ast}$ is completely integrally closed in $B$. \end{proof}
\begin{prop}\label{Seminormed algebras and lattices}Let $(A, \lVert\cdot\rVert)$ be a seminormed ring (respectively, a Banach ring) such that $\lVert A^{\times,m}\rVert$ is dense in $\mathbb{R}_{\geq0}$ and let $\varpi\in A^{\times,m}$ with $\lVert\varpi\rVert<1$. Let $A_{0}$ be an open subring of $A_{\leq1}$ containing $\varpi$. The functor \begin{equation*}B\mapsto B_{\leq1}\end{equation*}is an equivalence between the category of submetric seminormed $(A, \lVert\cdot\rVert)$-algebras $B$ (respectively, of submetric Banach $(A, \lVert\cdot\rVert)$-algebras $B$) and the category of $\varpi$-torsion-free (respectively, $\varpi$-torsion-free and $\varpi$-adically complete) $A_{0}$-algebras $B_{0}$ satisfying $(B_{0})_{\ast}=B_{0}$, and a quasi-inverse to $B\mapsto B_{\leq1}$ is given by the functor \begin{equation*}B_{0}\mapsto (B_{0}[\varpi^{-1}], \lVert\cdot\rVert_{B,B_{0},\ast}).\end{equation*}Moreover, $B\mapsto B_{\leq1}$ restricts to an equivalence between the category of uniform submetric seminormed $(A, \lVert\cdot\rVert)$-algebras $B$ (respectively, uniform submetric seminormed $(A, \lVert\cdot\rVert)$-algebras $B$) and the category of $\varpi$-torsion-free (respectively, $\varpi$-torsion-free and $\varpi$-adically complete) $A_{0}$-algebras $B_{0}$ with the property that $B_{0}$ is completely integrally closed in $B_{0}[\varpi^{-1}]$.\end{prop}
\begin{proof}The first assertion follows from Theorem \ref{Seminormed modules and lattices} (respectively, Theorem \ref{Banach modules and lattices}) and Proposition \ref{Gauge seminorms of subrings}. The last assertion then follows from Lemma \ref{Power-bounded}, Lemma \ref{Completely integrally closed} and Lemma \ref{Power-multiplicative gauges}. \end{proof}
\begin{rmk}In the case when $A$ is a Banach algebra over some densely valued nonarchimedean field $K$, a variant of Proposition \ref{Seminormed algebras and lattices} was observed by André in \cite{Andre18}, Cor.~2.3.2.\end{rmk}
\begin{mydef}Let $(A, \lVert\cdot\rVert)$ be a seminormed ring, let $A_{0}$ be an open subring of $A_{\leq1}$ and let $\varpi\in A_{0}$ be a topologically nilpotent unit of $A$ which admits a compatible system of $p$-power roots $(\varpi^{1/p^{n}})_{n}$ in $A_{0}$. We perform almost mathematics relative to the basic setting $(A_{0}, (\varpi^{1/p^{\infty}})_{A_{0}})$. An $A_{0}^{a}$-algebra $C$ is called almost root closed if there exists a $\varpi$-torsion-free $A_{0}$-algebra $B_{0}$ with $B_{0}^{a}\cong C$ and an integer $k\geq2$ such that for every $f\in B_{0}[\varpi^{-1}]$ with $f^{k}\in B_{0}$ we have $f\in B_{0}$.\end{mydef}
\begin{lemma}\label{Almost root closed and almost integrally closed}Let $(A, \lVert\cdot\rVert)$ be a seminormed ring, let $\varpi\in A^{\times,m}$ with $\lVert\varpi\rVert<1$ admitting a compatible system of $p$-power roots $(\varpi^{1/p^{n}})_{n}$ in $A^{\times,m}$. We perform almost mathematics relative to the basic setting $(A_{\leq1}, (\varpi^{1/p^{\infty}})_{A_{\leq1}})$. An $A_{\leq1}^{a}$-algebra $C$ is almost root closed if and only if there exists a $\varpi$-torsion-free $A_{\leq1}$-algebra $B_{0}$ with $B_{0}^{a}\cong C$ such that $B_{0}$ is completely integrally closed in $B_{0}[\varpi^{-1}]$.\end{lemma}
\begin{proof}Note that the assumptions on $\varpi$ imply that $\lVert A^{\times,m}\rVert$ is dense in $\mathbb{R}_{\geq0}$. If $B_{0}$ is a $\varpi$-torsion-free $A_{\leq1}$-algebra with $B_{0}^{a}\cong C$ as in the definition of an almost root closed $A_{\leq1}^{a}$-algebra, then, by Lemma \ref{Power-multiplicative gauges}, the gauge seminorm $\lVert\cdot\rVert_{B,B_{0},\ast}$ on $B=B_{0}[\varpi^{-1}]$ is power-multiplicative. By Lemma \ref{Power-bounded} and Lemma \ref{Completely integrally closed}, the closed unit ball $B_{\leq1}$ of $(B, \lVert\cdot\rVert_{B,B_{0},\ast})$ is completely integrally closed in $B=B_{\leq1}[\varpi^{-1}]$. By Lemma \ref{Two modules of almost elements}, $(B_{0})_{\ast}=C_{\ast}$, so $B_{\leq1}^{a}\cong B_{0}^{a}$. \end{proof}
\begin{lemma}\label{Almost root closed and almost integrally closed 2}Let $(A, \lVert\cdot\rVert)$ be a seminormed ring and let $\varpi\in A^{\times,m}\cap A_{<1}$ admitting a compatible system of $p$-power roots $(\varpi^{1/p^{n}})_{n}$ in $A^{\times,m}$. An $A_{\leq1}^{a}$-algebra $C$ is almost root closed if and only if $C_{\ast}$ is completely integrally closed in $C_{\ast}[\varpi^{-1}]$.\end{lemma}
\begin{proof}Follows from the proof of Lemma \ref{Almost root closed and almost integrally closed}.\end{proof}
We also have the following analog of Theorem \ref{Seminormed modules and almost lattices}.
\begin{thm}\label{Seminormed algebras and almost lattices}Let $(A, \lVert\cdot\rVert)$ be a seminormed ring and suppose that there exists a seminorm-multiplicative topologically nilpotent unit $\varpi$ of $A$ which admits a compatible system of $p$-power roots in $A^{\times,m}$. We perform almost mathematics relative to the basic setting $(A_{\leq1}, (\varpi^{1/p^{\infty}})_{A_{\leq1}})$. The functor \begin{equation*}B\mapsto B_{\leq1}^{a}\end{equation*}from the category of submetric seminormed $(A, \lVert\cdot\rVert)$-algebras (respectively, submetric Banach $(A, \lVert\cdot\rVert)$-algebras) to the category of $\varpi$-torsion-free (respectively, $\varpi$-torsion-free and $\varpi$-adically complete) $A_{\leq1}^{a}$-algebras is an equivalence of categories, with quasi-inverse given by \begin{equation*}C\mapsto (C_{\ast}[\varpi^{-1}], \lVert\cdot\rVert_{C_{\ast}[\varpi^{-1}],C_{\ast},\ast}).\end{equation*}The equivalence $B\mapsto B_{\leq1}^{a}$ restricts to an equivalence between the category of uniform submetric seminormed $(A, \lVert\cdot\rVert)$-algebras (respectively, of uniform submetric Banach $(A, \lVert\cdot\rVert)$-algebras) and the category of almost root closed (respectively, almost root closed and $\varpi$-adically complete) $A_{\leq1}^{a}$-algebras.\end{thm}
\begin{proof}Follows from Proposition \ref{Seminormed algebras and lattices} and Lemma \ref{Almost root closed and almost integrally closed 2}. \end{proof}

\section{Strict homomorphisms, isometries and complete tensor products}\label{sec:complete tensor products}

For many seminormed rings $A$ (such as a perfectoid Tate ring endowed with a power-multiplicative norm defining its topology) we can use the notion of a seminorm-multiplicative topologically nilpotent unit to give an algebraic description of strict morphisms in the categories of seminormed modules and of submetric seminormed modules. In particular, we want to describe isometries of submetric seminormed modules (that is, strict monomorphisms in the category of submetric seminormed modules), obtaining a stronger version of Corollary \ref{Almost isometries}. We begin with an elementary lemma which holds true for any seminormed ring $A$ with a (not necessarily seminorm-multiplicative) topologically nilpotent unit.
\begin{lemma}\label{Isometries}Let $A$ be a seminormed ring with a topologically nilpotent unit $\varpi$ and let $M\hookrightarrow N$ be an isometry between seminormed $A$-modules. Then the restriction \begin{equation*}M_{\leq1}\hookrightarrow N_{\leq1}\end{equation*}of $M\hookrightarrow N$ to the closed unit ball of $M$ is a $\varpi$-adic isometry, i.e., the induced map $M_{\leq1}/\varpi^{k}M_{\leq1}\to N_{\leq1}/\varpi^{k}N_{\leq1}$ is injective for every $k\geq1$.\end{lemma}
\begin{proof}Since $M\hookrightarrow N$ is an isometry, $N_{\leq1}\cap M=M_{\leq1}$. In particular, \begin{equation*}N_{\leq1}\cap\varpi^{-k}M_{\leq1}\subseteq M_{\leq1}\end{equation*}for every $k\geq1$. Fix some integer $k\geq1$ and let $x\in M_{\leq1}\cap\varpi^{k}N_{\leq1}$. Write $x=\varpi^{k}y$ for some $y\in N_{\leq1}$. Then $y\in N_{\leq1}\cap\varpi^{-k}M_{\leq1}\subseteq M_{\leq1}$, so $x\in\varpi^{k}M_{\leq1}$.\end{proof}
The following simple lemma is presumably well-known but we include it here for the reader's convenience.
\begin{lemma}\label{Isometries 1.5}Let $A_{0}$ be a ring with a non-zero-divisor $\varpi$ and let $\varphi_{0}: M_{0}\to N_{0}$ is an $A_{0}$-module homomorphism, where $N_{0}$ is $\varpi$-torsion-free. The following are equivalent: \begin{enumerate}[(1)]\item The induced $A_{0}/\varpi A_{0}$-module homomorphism $M_{0}/\varpi M_{0}\to N_{0}/\varpi N_{0}$ is injective. \item $\varphi_{0}$ is a $\varpi$-adic isometry, i.e., for every $k\geq1$, the induced $A_{0}/\varpi^{k}A_{0}$-module homomorphism $M_{0}/\varpi^{k}M_{0}\to N_{0}/\varpi^{k}N_{0}$ is injective.\end{enumerate}\end{lemma}
\begin{proof}Only one implication is non-trivial. Assume (1). We prove by induction on $k\geq1$ that $M_{0}/\varpi^{k}M_{0}\to N_{0}/\varpi^{k}N_{0}$ is injective for all $k\geq1$. So, assume that $M_{0}/\varpi^{k}M_{0}\to N_{0}/\varpi^{k}N_{0}$ is injective for some $k\geq1$ and let $x\in M_{0}\cap \varphi_{0}^{-1}(\varpi^{k+1}N_{0})$. We want to prove that $x\in\varpi^{k+1}M_{0}$. Write $\varphi_{0}(x)=\varpi^{k+1}y$ for some $y\in N_{0}$. Since $M_{0}/\varpi^{k}M_{0}\to N_{0}/\varpi^{k}N_{0}$ is injective, $x=\varpi^{k}z$ for some $z\in M_{0}$. Then $\varpi^{k}\varphi_{0}(z)=\varphi_{0}(x)=\varpi^{k+1}y$. Since $N_{0}$ is $\varpi$-torsion-free, this entails that $\varphi_{0}(z)=\varpi y$. By the injectivity of $M_{0}/\varpi M_{0}\to N_{0}/\varpi N_{0}$, this means that $z\in\varpi M_{0}$ and thus $x\in \varpi^{k+1}M_{0}$.\end{proof}
\begin{lemma}\label{Strictness can be checked on open submodules}Let $(A, \lVert\cdot\rVert)$ be a seminormed ring with a seminorm-multiplicative topologically nilpotent unit $\varpi$. Let $A_{0}$ be an open subring of $A$, let $M$ be an $A$-module and let $M_{0}$ be an $A_{0}$-submodule of $M$ with $M_{0}[\varpi^{-1}]=M$. If $\lVert\cdot\rVert$, $\lVert\cdot\rVert'$ are two seminorms on $M$ such that $(M, \lVert\cdot\rVert)$, $(M, \lVert\cdot\rVert')$ are submetric seminormed $(A, \lVert\cdot\rVert)$-modules and if for some $C>0$ we have\begin{equation*}\lVert x\rVert'C\leq \lVert x\rVert\end{equation*}for all $x\in M_{0}$, then the same inequality also holds for all $x\in M$.\end{lemma}
\begin{proof}By Lemma \ref{Multiplicative topologically nilpotent units}, $\varpi$ and $\varpi^{-1}$ are multiplicative elements for the seminorms $\lVert\cdot\rVert$ and $\lVert\cdot\rVert'$ on $M$. The assertion follows from this and the assumption that $M_{0}[\varpi^{-1}]=M$.\end{proof}
\begin{prop} \label{Bounded torsion and norm estimates} Let $A_{0}$ be a ring, let $\varpi \in A_{0}$ be a non-zero-divisor and let $\varphi: M_{0} \to N_{0}$ be a map of $A_{0}$-modules. Endow $M_{0}$ and $N_{0}$ with their respective $\varpi$-adic seminorms. Let $m\geq 1$ be an integer. Consider the following assertions:
\begin{enumerate}[(1)]
\item For all $y \in M_{0}$ we have \begin{equation*} \inf_{z \in M_{0}, \\\ \varphi(z) = \varphi(y)} \lVert z \rVert \lVert \varpi^{m}\rVert \leq \lVert \varphi(y) \rVert. \end{equation*}
\item There exists a constant $C \in (\lVert {\varpi}^{m+1}\rVert, \lVert {\varpi}^{m}\rVert]$ such that 
\begin{equation*} \inf_{z \in M_{0}, \\\ \varphi(z)=\varphi(y)} \lVert z \rVert C \leq \lVert \varphi(y) \rVert \end{equation*} for all $y \in M_{0}$.
\item The $\varpi^{\infty}$-torsion submodule \begin{equation*} P_{\varpi-tor} = \{\, x \in P \mid  \text{there exists $n$ such that ${\varpi}^{n}x = 0$} \,\} \end{equation*} of the cokernel $P := N_{0}/\im(\varphi)$ is annihilated by ${\varpi}^{m}$. 
\end{enumerate} Then (3) implies (1) and if the module $N_{0}$ is assumed to be $\varpi$-torsion-free, then the three assertions are equivalent.\end{prop}
\begin{proof} (3) $\Rightarrow$ (1):  Let $m \geq 0$ be such that ${\varpi}^{m}$ annihilates the $\varpi$-torsion-submodule of $N_{0}/\im(\varphi)$. The quotient seminorm on $\im(\varphi)$ coincides with the $\varpi$-adic seminorm. Hence (by replacing $\varphi$ with the inclusion map $\im(\varphi) \hookrightarrow N_{0}$) we may assume that $\varphi$ is injective and identify $\im(\varphi)$ with $M_{0}$. Denote by $\lVert \cdot \rVert_{M_{0}}, \lVert \cdot \rVert_{N_{0}}$ the respective $\varpi$-adic seminorms on $M_{0}$ and $N_{0}$ and denote by $\lVert \cdot \rVert$ the $\varpi$-adic seminorm on $A_{0}$. We want to prove that $\lVert \cdot \rVert_{M_{0}}$ is bounded above by the restriction of $\lVert \cdot \rVert_{N_{0}}$ multiplied by the constant $\lVert \varpi^{m} \rVert$. By the proof of \cite{FGK}, Theorem 4.2.2 (2), we know that the equality \begin{equation*} M_{0} \cap {\varpi}^{k+m}N_{0}=\varpi^{k}(M_{0} \cap \varpi^{m}N_{0}) \end{equation*}holds for every $k\geq0$. Suppose that $x\in M_{0}\subseteq N_{0}$ satisfies $\lVert \varpi \rVert^{-m}\lVert x \rVert_{N_{0}}=\lVert {\varpi}^{-m}x \rVert_{N_{0}}< r$ for some $r>0$. Then there exists $k\geq0$ such that $\varpi^{-m}x \in \varpi^{k}N_{0}$ and $\lVert\varpi^{k}\rVert<r$. Consequently, $x\in M_{0}\cap\varpi^{k+m}N_{0}=\varpi^{k}(M_{0}\cap {\varpi}^{m}N_{0}) \subseteq {\varpi}^{k}M_{0}$. It follows that $\lVert x\rVert_{M_{0}}\leq \lVert\varpi^{k}\rVert<r$ whence we see that $\lVert x\rVert_{M_{0}}\leq\lVert\varpi\rVert^{-m}\lVert x \rVert_{N_{0}}$, as desired.

The implication (1) $\Rightarrow$ (2) is trivial, so it remains to prove the implication (2) $\Rightarrow$ (3) when $N_{0}$ is $\varpi$-torsion-free. Thus let $N_{0}$ be a $\varpi$-torsion-free $A_{0}$-module and assume that assertion (2) holds for the map $\varphi: M_{0} \to N_{0}$. By assumption we have $\lVert\varpi^{m+1}\rVert<C$. This means that \begin{equation*}\varpi^{m+1}N_{0} \subseteq \{\, x \in N_{0} \mid \lVert x \rVert_{N_{0}}<C \,\} \end{equation*}. Let $x \in N_{0}$ and suppose there exists $n \geq 0$ such that $\varpi^{n}x \in \im(\varphi)$. Then there also exists $n\geq m+1$ such that $\varpi^{n}x \in \im(\varphi)$. Write $\varpi^{n}x=\varphi(y)$ for some $y \in M_{0}$. Then $\lVert \varphi(y) \rVert_{N_{0}}<C$ and hence \begin{equation*} \inf_{z \in M_{0}, \varphi(z) = \varphi(y)}\lVert z\rVert_{M_{0}}\leq C^{-1}\lVert \varphi(y)\rVert_{N_{0}}<1.\end{equation*}This shows that for some $z \in M_{0}$ with $\varphi(z) =\varphi(y)$ we have $\lVert z\rVert_{M_{0}}<1$. But by the definition of the $\varpi$-adic seminorm this implies $z=\varpi y_{1}$ for some $y_{1} \in M_{0}$ so that $\varpi^{n}x=\varphi(y)=\varphi(\varpi y_{1})=\varpi\varphi(y_{1})$. Since $N_{0}$ is $\varpi$-torsion-free we deduce: ${\varpi}^{n-1}x=\varphi(y_{1})$. If $n-1>m$, then applying the same procedure as above to $y_{1}$ we obtain an element $y_{2}$ with ${\varpi}^{n-2}x=\varphi(y_{2})$ and so forth. After $n-m$ steps we see that $\varpi^{m}x \in \im(\varphi)$. Since $m$ did not depend on the particular element $x$ we conclude that \begin{equation*}\varpi^{m}(N_{0}/\im(\varphi))_{\varpi-\tor}=0,\end{equation*}as we claimed.\end{proof}
\begin{prop}\label{Bounded torsion and norm estimates 2}Let $(A, \lVert\cdot\rVert)$ be a seminormed ring, let $A_{0}$ be an open subring of $A_{\leq1}$. Suppose that $\varpi\in A_{0}$ is a topologically nilpotent unit of $A$ which admits a compatible system of $p$-power roots $(\varpi^{1/p^{n}})_{n}$ in $A_{0}$ with the property that $\lVert\varpi^{s}\rVert=\lVert\varpi\rVert^{s}$ for all $s\in\mathbb{Z}[1/p]$. Let $\varphi: (M, \lVert\cdot\rVert_{M})\to (N, \lVert\cdot\rVert_{N})$ be a morphism of submetric seminormed $(A, \lVert\cdot\rVert)$-modules and let $M_{0}\subseteq M$, $N_{0}\subseteq N$ be open $A_{0}$-submodules with $(M_{0}^{a})_{\ast}=M_{\leq1}$ and $(N_{0}^{a})_{\ast}=N_{\leq1}$ such that $\varphi(M_{0})\subseteq N_{0}$. \begin{enumerate}[(1)]\item If the $\varpi^{\infty}$-torsion submodule of $N_{0}/\varphi(M_{0})$ is annihilated by $\varpi^{m}$ for some integer $m\geq0$, then \begin{equation*}\inf_{z\in M, \\\ \varphi(z)=\varphi(y)}\lVert z\rVert_{M}\lVert\varpi\rVert^{m}\leq\lVert\varphi(y)\rVert_{N}\end{equation*}for all $y\in M$. \item Conversely, if for some $m\in\mathbb{Z}_{>0}$, we have \begin{equation*}\inf_{z\in M, \\\ \varphi(z)=\varphi(y)}\lVert z\rVert_{M}\lVert\varpi\rVert^{m}\leq\lVert\varphi(y)\rVert_{N}\end{equation*}for all $y\in M$, then the $\varpi^{\infty}$-torsion submodule of $N_{0}/\varphi(M_{0})$ is annihilated by $\varpi^{m+\frac{2}{p^{n}}}$ for every $n\in\mathbb{Z}_{>0}$.\end{enumerate}\end{prop}
\begin{proof}(1) By Proposition \ref{Bounded torsion and norm estimates} and Lemma \ref{Strictness can be checked on open submodules}, the hypothesis in (1) entails the inequalities \begin{equation*}\inf_{z\in M, \varphi(z)=\varphi(y)}\lVert z\rVert_{M,M_{0},\varpi^{1/p^{n}}}\lVert\varpi\rVert^{m}\leq \lVert\varphi(y)\rVert_{N,N_{0},\varpi^{1/p^{n}}}\end{equation*}for all $y\in M$ and all $n\geq1$. The assertion (1) follows from this and Corollary \ref{Infimum of discrete norms}.

(2) Suppose that \begin{equation*}\inf_{z\in M, \\\ \varphi(z)=\varphi(y)}\lVert z\rVert_{M}\lVert\varpi\rVert^{m}\leq\lVert\varphi(y)\rVert_{N}\end{equation*}for some $m\in\mathbb{Z}_{\geq0}$. By Corollary \ref{Infimum of discrete norms}, we have $\lVert\cdot\rVert_{N}\leq\lVert\cdot\rVert_{N,N_{0},\varpi^{1/p^{n}}}$ for all $n\geq0$, and, by Lemma \ref{Boundedness vs. continuity}, \begin{equation*}\lVert\cdot\rVert_{M,M_{0},\varpi^{1/p^{n}}}\lVert\varpi\rVert^{2/p^{n}}\leq\lVert\cdot\rVert_{M}.\end{equation*}It follows that \begin{equation*}\inf_{z\in M, \\\ \varphi(z)=\varphi(y)}\lVert z\rVert_{M,M_{0},\varpi^{1/p^{n}}}\lVert\varpi\rVert^{m+\frac{2}{p^{n}}}\leq\lVert\varphi(y)\rVert_{N,N_{0},\varpi^{1/p^{n}}}\end{equation*}for all $y\in M$ and all $n\geq1$. By Proposition \ref{Bounded torsion and norm estimates}, this means that the $\varpi^{\infty}$-torsion submodule of $N_{0}/\varphi(M_{0})$ is annihilated by $(\varpi^{1/p^{n}})^{p^{n}m+2}=\varpi^{m+\frac{2}{p^{n}}}$ for all $n\geq1$.\end{proof} 
\begin{cor}\label{Isometries 2}Let $(A, \lVert\cdot\rVert)$ be a seminormed ring, let $A_{0}$ be an open subring of $A_{\leq1}$. Suppose that $\varpi\in A_{0}$ is a topologically nilpotent unit of $A$ which admits a compatible system of $p$-power roots $(\varpi^{1/p^{n}})_{n}$ in $A_{0}$ with the property that $\lVert\varpi^{s}\rVert=\lVert\varpi\rVert^{s}$ for all $s\in\mathbb{Z}[1/p]$. Then an injective submetric homomorphism of submetric seminormed $(A, \lVert\cdot\rVert)$-modules $\varphi: (M, \lVert\cdot\rVert_{M})\to (N, \lVert\cdot\rVert_{N})$ is an isometry if and only if there exist open $A_{0}$-submodules $M_{0}$ of $M$ and $N_{0}$ of $N$ with $(M_{0}^{a})_{\ast}=M_{\leq1}$, $(N_{0}^{a})_{\ast}=N_{\leq1}$ such that $\varphi(M_{0})\subseteq N_{0}$ and the $A_{0}/\varpi A_{0}$-module homomorphism $M_{0}/\varpi M_{0}\to N_{0}/\varpi N_{0}$ induced by $\varphi$ is injective.\end{cor}
\begin{proof}If $\varphi$ is an isometry, $M_{\leq1}/\varpi M_{\leq1}\to N_{\leq1}/\varpi N_{\leq1}$ is injective, by Lemma \ref{Isometries}. Conversely, suppose that $M_{0}\subseteq M$, $N_{0}\subseteq N$ are open $A_{0}$-submodules as in the statement of the corollary. By Proposition \ref{Bounded torsion and norm estimates}, the cokernel $N_{0}/\varphi(M_{0})$ is $\varpi$-torsion-free. We conclude by Proposition \ref{Bounded torsion and norm estimates 2} that $\varphi$ is an isometry.\end{proof} 
For any seminormed ring $A$ there is a canonical closed symmetric monoidal structure on each of the categories $\Snm_{A}$ and $\Snm_{A}^{\leq1}$, where the tensor product of two seminormed $A$-modules (respectively, submetric seminormed $A$-modules) $(M, \lVert\cdot\rVert_{M})$ and $(N, \lVert\cdot\rVert_{N})$ is the usual algebraic tensor product $M\otimes_{A}N$ equipped with the tensor product seminorm \begin{equation*}\lVert\cdot\rVert_{M}\otimes\lVert\cdot\rVert_{N}: M\otimes_{A}N\to\mathbb{R}_{\geq0}\end{equation*}given by \begin{equation*}(\lVert\cdot\rVert_{M}\otimes\lVert\cdot\rVert_{N})(x)=\inf\{\, \lVert x_{i}\rVert_{M}\lVert y_{i}\rVert_{N}\mid x=\sum_{i}x_{i}\otimes y_{i}, x_{i}\in M, y_{i}\in N\,\}.\end{equation*}If $A$ is a Banach ring, the categories $\Ban_{A}$ and $\Ban_{A}^{\leq1}$ also carry a canonical closed symmetric monoidal structure given by the complete tensor product $M\widehat{\otimes}_{A}N$ of two Banach modules (respectively, submetric Banach modules) $(M, \lVert\cdot\rVert_{M})$ and $(N, \lVert\cdot\rVert_{N})$, that is, the completion of the algebraic tensor product $M\otimes_{A}N$ with respect to the tensor product seminorm $\lVert\cdot\rVert_{M}\otimes\lVert\cdot\rVert_{N}$. In the case of two submetric seminormed $A$-modules $M$, $N$ whose seminorms take values in the closure of $\lVert A^{\times,m}\rVert$ we can describe the tensor product seminorm in very explicit terms.
\begin{prop}\label{Tensor products and gauges}Let $(A, \lVert\cdot\rVert)$ be a seminormed ring and let $(M, \lVert\cdot\rVert_{M})$, $(N, \lVert\cdot\rVert_{N})$ be submetric seminormed $A$-modules and suppose that $\lVert A\rVert$, $\lVert M\rVert_{M}$ and $\lVert N\rVert_{N}$ are contained in the closure of $\lVert A^{\times,m}\rVert$ in $\mathbb{R}_{\geq0}$. Let $A_{0}$ be an open subring of $A_{\leq1}$ and let $M_{0}\subseteq M$, $N_{0}\subseteq N$ be open $A_{0}$-submodules with $(M_{0})_{\ast}=M_{\leq1}$ and $(N_{0})_{\ast}=N_{\leq1}$. Then the tensor product seminorm $\lVert\cdot\rVert_{M}\otimes\lVert\cdot\rVert_{N}$ on $M\otimes_{A}N$ is equal to the gauge of \begin{equation*}((A_{\leq1}\cdot M_{0})\otimes_{A_{0}}(A_{\leq1}\cdot N_{0}))/(\varpi^{\infty}-\tor).\end{equation*}\end{prop}
\begin{proof}By Proposition \ref{Description of norms 1} and Lemma \ref{Gauge seminorms and lattices}, the seminorm $\lVert\cdot\rVert_{M}$ (respectively, $\lVert\cdot\rVert_{N}$) is equal to the gauge of $M_{0}$ (respectively, of $N_{0}$). By Proposition \ref{Another definition of the gauge seminorm}, this means that \begin{equation*}\lVert x\rVert_{M}=\inf\{\, \lVert f\rVert \mid f\in A^{\times,m}, x\in fA_{\leq1}M_{0}\,\}\end{equation*}for all $x\in M$ and similarly for $\lVert\cdot\rVert_{N}$. Then the same argument as in the proof of \cite{Schneider}, Lemma 17.2, shows that the tensor product seminorm $\lVert\cdot\rVert_{M}\otimes\lVert\cdot\rVert_{N}$ is given by \begin{equation*}x\mapsto \inf\{\, \lVert f\rVert\mid f\in A^{\times,m}, x\in f((A_{\leq1}\cdot M_{0})\otimes_{A_{0}}(A_{\leq1}\cdot N_{0}))/(\varpi^{\infty}-\tor)\,\}.\end{equation*}Since the images of $\lVert\cdot\rVert_{M}$, $\lVert\cdot\rVert_{N}$ are contained in the closure of $\lVert A^{\times,m}\rVert$ inside $\mathbb{R}_{\geq0}$, the image of $\lVert\cdot\rVert_{M}\otimes\lVert\cdot\rVert_{N}$ is also contained in the closure of $\lVert A^{\times,m}\rVert$, by definition of the tensor product seminorm. We conclude by Proposition \ref{Another definition of the gauge seminorm}.\end{proof}
We note the following two special cases of the above lemma.
\begin{cor}\label{Tensor products and gauges 2}Let $A_{0}$ be a ring with a non-zero-divisor $\varpi$ and let $M_{0}$, $N_{0}$ be $\varpi$-torsion-free $A_{0}$-modules. Equip $A=A_{0}[\varpi^{-1}]$, $M=M_{0}[\varpi^{-1}]$ and $N=N_{0}[\varpi^{-1}]$ with the canonical extensions of the respective $\varpi$-adic seminorms. Then the tensor product seminorm on $M\otimes_{A}N$ is equal to the canonical extension of the $\varpi$-adic seminorm on $M_{0}\otimes_{A_{0}}N_{0}$.\end{cor}
\begin{proof}Combine the previous proposition with Example \ref{Discrete seminorms and gauges}.\end{proof}
\begin{cor}\label{Tensor products and gauges 3}Let $(A, \lVert\cdot\rVert)$ be a seminormed ring and suppose that $\lVert A^{\times,m}\rVert$ is dense in $\mathbb{R}_{\geq0}$. Let $\varpi\in A^{\times,m}$ with $\lVert\varpi\rVert<1$ and let $(M, \lVert\cdot\rVert_{M}), (N, \lVert\cdot\rVert_{N})\in \Snm_{A}^{\leq1}$. Then, for any open subring $A_{0}$ of $A_{\leq1}$ with $\varpi\in A_{0}$ and any two open $A_{0}$-submodules $M_{0}\subseteq M$, $N_{0}\subseteq N$ with $(M_{0})_{\ast}=M_{\leq1}$, $(N_{0})_{\ast}=N_{\leq1}$, the tensor product seminorm $\lVert\cdot\rVert_{M}\otimes\lVert\cdot\rVert_{N}$ on $M\otimes_{A}N$ is equal to the gauge of $(M_{0}\otimes_{A_{0}}N_{0})/(\varpi^{\infty}-\tor)$.\end{cor}
\begin{proof}Using Lemma \ref{Seminormed rings and rings of almost elements}, we see that \begin{equation*}((A_{\leq1}\cdot M_{0})\otimes_{A_{0}}(A_{\leq1}\cdot N_{0}))/(\varpi^{\infty}-\tor)\subseteq ((M_{0}\otimes_{A_{0}}N_{0})/(\varpi^{\infty}-\tor))_{\ast}\end{equation*}and thus that \begin{equation*}(((A_{\leq1}\cdot M_{0})\otimes_{A_{0}}(A_{\leq1}\cdot N_{0}))/(\varpi^{\infty}-\tor))_{\ast}=((M_{0}\otimes_{A_{0}}N_{0})/(\varpi^{\infty}-\tor))_{\ast}.\end{equation*}By Lemma \ref{Gauge seminorms and lattices}, this means that the gauges of $((A_{\leq1}\cdot M_{0})\otimes_{A_{0}}(A_{\leq1}\cdot N_{0}))/(\varpi^{\infty}-\tor)$ and of $(M_{0}\otimes_{A_{0}}N_{0})/(\varpi^{\infty}-\tor)$ are equal. We conclude by applying Proposition \ref{Tensor products and gauges}.\end{proof}
\begin{cor}\label{Unit balls and tensor products}Let $(A, \lVert\cdot\rVert)$ be a seminormed ring and suppose that $\lVert A^{\times,m}\rVert$ is dense in $\mathbb{R}_{\geq0}$. Let $\varpi\in A^{\times,m}$ with $\lVert\varpi\rVert<1$ and let $(M, \lVert\cdot\rVert_{M}), (N, \lVert\cdot\rVert_{N})\in \Snm_{A}^{\leq1}$. Then, for any open subring $A_{0}$ of $A_{\leq1}$ with $\varpi\in A_{0}$ and any two open $A_{0}$-submodules $M_{0}\subseteq M$, $N_{0}\subseteq N$ with $(M_{0})_{\ast}=M_{\leq1}$, $(N_{0})_{\ast}=N_{\leq1}$, the closed unit ball of $M\otimes_{A}N$ with respect to the tensor product seminorm $\lVert\cdot\rVert_{M}\otimes\lVert\cdot\rVert_{N}$ is given by \begin{equation*}((M_{0}\otimes_{A_{0}}N_{0})/(\varpi^{\infty}-\tor))_{\ast}.\end{equation*}\end{cor}
\begin{proof}Follows from Corollary \ref{Tensor products and gauges 3} and Lemma \ref{Closed unit ball and module of almost elements}.\end{proof}
\begin{cor}\label{Unit balls and tensor products 2}Let $(A, \lVert\cdot\rVert)$ be a Banach ring and suppose that $\lVert A^{\times,m}\rVert$ is dense in $\mathbb{R}_{\geq0}$. Let $\varpi\in A^{\times,m}$ with $\lVert\varpi\rVert<1$ and let $(M, \lVert\cdot\rVert_{M}), (N, \lVert\cdot\rVert_{N})\in \Ban_{A}^{\leq1}$. Then, for any open subring $A_{0}$ of $A_{\leq1}$ with $\varpi\in A_{0}$ and any two open $A_{0}$-submodules $M_{0}\subseteq M$, $N_{0}\subseteq N$ with $(M_{0})_{\ast}=M_{\leq1}$, $(N_{0})_{\ast}=N_{\leq1}$, the closed unit ball of $M\widehat{\otimes}_{A}N$ for the tensor product norm is given by \begin{equation*}((M_{0}\widehat{\otimes}_{A_{0}}N_{0})/(\varpi^{\infty}-\tor))_{\ast},\end{equation*}where $M_{0}\widehat{\otimes}_{A_{0}}N_{0}$ is the $\varpi$-adic completion of $M_{0}\otimes_{A_{0}}N_{0}$. \end{cor}
\begin{proof}The closed unit ball of the completion $M\widehat{\otimes}_{A}N$ is the completion of the closed unit ball $(M\otimes_{A}N)_{\leq1}$ with respect to the subspace seminorm induced from the tensor product seminorm $\lVert\cdot\rVert_{M}\otimes\lVert\cdot\rVert_{N}$. By Lemma \ref{Bounded open submodules}, this subspace norm is bounded-equivalent to the $\varpi$-adic seminorm on $(M\otimes_{A}N)_{\leq1}$. It follows that $(M\widehat{\otimes}_{A}N)_{\leq1}$ is equal to the $\varpi$-adic completion of $(M\otimes_{A}N)_{\leq1}$. Hence the assertion follows from Corollary \ref{Tensor products and gauges 3} and the fact that the $\varpi$-adic completion of the module of almost elements of a $\varpi$-torsion-free $A_{0}$-module is equal to the module of almost elements of the $\varpi$-adic completion.\end{proof}

\section{Proof of Theorem \ref{Relation to condensed mathematics}}\label{sec:Proof of the theorem on condensed math}

We now relate our previous results to condensed almost mathematics in the sense of Mann \cite{Mann22}. Recall that, for any almost setup $(V, \mathfrak{m})$ and any strong limit cardinal $\kappa$, Mann defines the category of static ($\kappa$-condensed) almost $(V, \mathfrak{m})$-modules $\Cond(V, \mathfrak{m})_{\kappa}$ as the Serre quotient of the category $\Cond(V)_{\kappa}$ of $\kappa$-condensed $V$-modules by the thick subcategory consisting of almost zero $\kappa$-condensed $V$-modules (\cite{Mann22}, Definition 2.2.5(c)). In this definition, $V$ is viewed as a condensed ring and $\mathfrak{m}$ as a condensed $V$-module by endowing both with the discrete topology. In the following, we always take $\kappa=\aleph_{0}$ and set \begin{equation*}\Cond(V, \mathfrak{m}):=\Cond(V, \mathfrak{m})_{\aleph_{0}},\end{equation*}the category of static (light) condensed almost $(V, \mathfrak{m})$-modules. Similarly, all profinite sets, condensed sets and condensed modules mentioned in this section are going to be light profinite sets, light condensed sets and light condensed modules, so we omit the adjective "light" from our exposition. Also, no derived categories will appear in this paper, so we also omit the adjective "static" when speaking about condensed almost $(V, \mathfrak{m})$-modules. We refer to \cite{Condensed}, \cite{ScholzeAnalytic} and \cite{Kedlaya24} for basic terminology on condensed mathematics.     
\begin{lemma}\label{Discrete almost modules}Let $(V, \mathfrak{m})$ be an almost setup. If $N_{0}$, $N_1$ are two $V$-modules such that $N_{0}^{a}\cong N_{1}^{a}$, then the condensed $V$-modules obtained by equipping $N_{0}$ and $N_{1}$ with the discrete topology are almost isomorphic.\end{lemma}
\begin{proof}We have to show that for every profinite set $S$ the $V$-modules of continuous functions $C(S, N_{0})$ and $C(S, N_1)$, where $N_{0}$, $N_1$ are endowed with the discrete topology, are almost isomorphic. Since $N_{0}$, $N_{1}$ are endowed with the discrete topology, the two $V$-modules of continuous functions are the same as the $V$-modules of all functions $N_{0}^{S}$, $N_{1}^{S}$. Let $\varphi: N_{0}\to (N_{1}^{a})_{\ast}$ be an almost isomorphism, i.e., the kernel and cokernel of $\varphi$ are annihilated by $\mathfrak{m}$. Since arbitrary direct products in the category of $V$-modules are exact, the kernel and cokernel of the induced map $N_{0}^{S}\to (N_{1}^{a})_{\ast}^{S}$ are also annihilated by $\mathfrak{m}$. Since the localization functor from $V$-modules to almost $(V, \mathfrak{m})$-modules commutes with limits (\cite{Gabber-Ramero}, Proposition 2.2.23), this means that $\varphi$ induces an isomorphism of almost modules between $(N_{0}^{S})^{a}$ and $(N_{1}^{S})^{a}$.\end{proof}
\begin{lemma}\label{Adically complete almost modules}Let $(V, \mathfrak{m})$ be an almost setup and let $\varpi\in\mathfrak{m}\setminus\{0\}$. If $N_{0}$, $N_{1}$ are two $\varpi$-torsion-free, $\varpi$-adically complete $V$-modules with $N_{0}^{a}\cong N_{1}^{a}$, then the condensed $V$-modules $\underline{N_{0}}$ and $\underline{N_{1}}$ obtained by endowing $N_{0}$ and $N_{1}$ with the $\varpi$-adic topology are almost isomorphic.\end{lemma}
\begin{proof}Let $\varphi: N_{0}\to (N_{1}^{a})_{\ast}$ be a map of usual $V$-modules which is an almost isomorphism. We prove that, for every profinite set $S$, the induced map \begin{equation*}\varphi_{\ast}: C(S, N_{0})\to C(S, (N_{1}^{a})_{\ast}), f\mapsto \varphi\circ f,\end{equation*}between modules of continuous functions, where $N_{0}$, $(N_{1}^{a})_{\ast}$ are endowed with their $\varpi$-adic topologies, is an almost isomorphism. Consider the induced maps \begin{equation*}\varphi_{n}: N_{0}/\varpi^{n}N_{0}\to (N_{1}^{a})_{\ast}/\varpi^{n}(N_{1}^{a})_{\ast}\end{equation*}for all $n\geq1$. Since the localization functor is symmetric monoidal and exact, the corresponding almost maps $\varphi_{n}^{a}$ are obtained from $\varphi^{a}$ by tensoring with \begin{equation*}V^{a}/\varpi^{n}V^{a}=(V/\varpi^{n}V)^{a}.\end{equation*}By Lemma \ref{Discrete almost modules}, the corresponding almost maps \begin{equation*}\varphi_{n\ast}^{a}: C(S, N_{0}/\varpi^{n}N_{0})^{a}\to C(S, (N_{1}^{a})_{\ast}/\varpi^{n}(N_{1}^{a})_{\ast})^{a}\end{equation*}are isomorphisms for all $n\geq1$. Taking inverse limits in the category of almost $(V, \mathfrak{m})$-modules, we obtain an isomorphism of almost modules\begin{equation*}\varprojlim_{n}C(S, N_{0}/\varpi^{n}N_{0})^{a}\cong \varprojlim_{n}C(S, (N_{1}^{a})_{\ast}/\varpi^{n}(N_{1}^{a})_{\ast})^{a}.\end{equation*}But by \cite{Gabber-Ramero}, Proposition 2.2.23, the localization functor $N\mapsto N^{a}$ from the category of $V$-modules to the category of almost $(V, \mathfrak{m})$-modules is a right adjoint and thus preserves all limits. Using the $\varpi$-adic completeness of $N_{0}$, we conclude that the map \begin{equation*}C(S, N_{0})=\varprojlim_{n}C(S, N_{0}/\varpi^{n}N_{0})\to C(S, \widehat{(N_{1}^{a})_{\ast}})=\varprojlim_{n}C(S, (N_{1}^{a})_{\ast}/\varpi^{n}(N_{1}^{a})_{\ast})\end{equation*}is an almost isomorphism. But by Lemma \ref{Torsion-free complete almost modules}, $(N_{1}^{a})_{\ast}$ is $\varpi$-adically complete. It follows that $\varphi_{\ast}$ is an almost isomorphism, as claimed. On the other hand, the same argument shows that the map $C(S, N_1)\to C(S, (N_{1}^{a})_{\ast})$ induced by the canonical almost isomorphism $N_{1}\to (N_{1}^{a})_{\ast}$ is an almost isomorphism. In this way we obtain isomorphisms of almost modules \begin{equation*}C(S, N_{0})^{a}\to C(S, N_{1})^{a}\end{equation*}which are canonical and thus compatible with restriction along any continuous map of profinite sets $T\to S$.\end{proof}
Lemma \ref{Adically complete almost modules} allows us to define an embedding of the category of $\varpi$-torsion-free, $\varpi$-adically complete almost modules into the category of condensed almost modules. For an almost setup $(V, \mathfrak{m})$ and an element $\varpi\in \mathfrak{m}\setminus \{0\}$, we denote by $\Mod_{V}^{\wedge,\varpi-\tf}$ the category of $\varpi$-torsion-free, $\varpi$-adically complete $V$-modules and by $\Mod_{V^{a}}^{\wedge,\varpi-\tf}$ the category of $\varpi$-torsion-free, $\varpi$-adically complete almost $(V, \mathfrak{m})$-modules.   
\begin{prop}\label{Adic embedding}Let $(V, \mathfrak{m})$ be an almost setup. There exists a categorical embedding \begin{equation*}\Mod_{V^{a}}^{\wedge,\varpi-\tf}\hookrightarrow \Cond(V, \mathfrak{m}), N\mapsto \underline{N},\end{equation*}of the category of $\varpi$-torsion-free, $\varpi$-adically complete almost $(V, \mathfrak{m})$-modules into the category of condensed almost $(V, \mathfrak{m})$-modules which makes the following diagram of functors commute: \begin{center}\begin{tikzcd}\Cond(V) \arrow{r} & \Cond(V, \mathfrak{m}) \\ \Mod_{V}^{\wedge,\varpi-\tf}\arrow{r} \arrow[hook]{u} & \Mod_{V^{a}}^{\wedge,\varpi-\tf}, \arrow[hook]{u}\end{tikzcd}\end{center}where the horizontal arrows are the localization functors from (condensed) modules to (condensed) almost modules and where the left vertical arrow takes a $\varpi$-adically complete $V$-module $N_{0}$ to the condensed $V$-module $\underline{N_{0}}$ associated with $N_{0}$ when the latter is endowed with the $\varpi$-adic topology.\end{prop}
\begin{proof}The existence of a well-defined functor $\Mod_{V^{a}}^{\wedge,\varpi-\tf}\to \Cond(V, \mathfrak{m})$ making the diagram commute follows from Lemma \ref{Adically complete almost modules}. We prove that the functor is fully faithful. Let $M, N\in \Mod_{V^{a}}^{\wedge,\varpi-\tf}$ and let $\psi: \underline{M}\to \underline{N}$ be a morphism in $\Cond(V, \mathfrak{m})$. Let $M_{0}, N_{0}\in \Mod_{V}^{\wedge,\varpi-\tf}$ with $M_{0}^{a}=M$, $N_{0}^{a}=N$. Since the diagram commutes, we have $\underline{M_{0}}^{a}=\underline{M}$, $\underline{N_{0}}^{a}=\underline{N}$. Since the functors $\Cond(V)\to \Cond(V, \mathfrak{m})$ and $\Mod_{V}^{\wedge,\varpi-\tf}\hookrightarrow \Cond(V)$ are full, we see that $\psi$ comes from a morphism $\varphi: M_{0}\to N_{0}$. Then $\varphi^{a}$ is a morphism in $\Mod_{V^{a}}^{\wedge\varpi}$ with $\underline{\varphi}=\psi$. This shows that the functor $N\mapsto \underline{N}$ is full. To prove that it is faithful, let $\varphi_{1}, \varphi_{2}: M_{0}\to N_{0}$ be two morphisms in $\Mod_{V}^{\wedge,\varpi-\tf}$ such that $\underline{\varphi_{1}^{a}}=\underline{\varphi_{2}^{a}}$ in $\Cond(V, \mathfrak{m})$. Consider the canonical morphism $\iota: \eq(\varphi_{1}, \varphi_{2})\to M_{0}$ from the equalizer of $\varphi_{1}$, $\varphi_{2}$. Then the image of this morphism in $\Cond(V, \mathfrak{m})$ is an isomorphism. Since the diagram commutes, the image $\underline{\iota}$ of $\iota$ in $\Cond(V)$ is an almost isomorphism of condensed $V$-modules, i.e., $\iota_{S}^{\ast}: C(S, \eq(\varphi_{1}, \varphi_{2}))\to C(S, M_{0})$ is an almost isomorphism for every profinite set $S$. Taking $S$ to be the point, we see that $\iota$ is an almost isomorphism, so $\varphi_{1}^{a}=\varphi_{2}^{a}$ in $\Mod_{V^{a}}^{\wedge,\varpi-\tf}$.\end{proof}
\begin{mydef}[The $\varpi$-adic embedding]\label{Adic embedding, definition}We call the categorical embedding \begin{equation*}\Mod_{V^{a}}^{\wedge,\varpi-\tf}\hookrightarrow \Cond(V, \mathfrak{m}), N\mapsto \underline{N},\end{equation*}from Proposition \ref{Adic embedding} the $\varpi$-adic embedding of $\Mod_{V^{a}}^{\wedge\varpi}$ into $\Cond(V, \mathfrak{m})$.\end{mydef}
This should be contrasted with the analogous embedding obtained by endowing each object of $\Mod_{V}^{\wedge,\varpi-\tf}$ with the discrete topology (instead of the $\varpi$-adic topology), which can also be defined on the larger category of all almost $(V, \mathfrak{m})$-modules. Note that the assumption that the modules be $\varpi$-torsion-free $\varpi$-adically complete in Lemma \ref{Adically complete almost modules} was actually used in the proof of that lemma. 

The $\varpi$-adic completeness assumption also ensures that the condensed almost modules in the essential image of the $\varpi$-adic embedding are solid in the sense of the following definition.
\begin{mydef}[Solid condensed almost module]\label{Solid almost modules}Let $(V, \mathfrak{m})$ be an almost setup. We call a condensed almost $(V, \mathfrak{m})$-module \begin{equation*}N\in\Cond(V, \mathfrak{m})\end{equation*}solid if the condensed $V$-module $N_{\ast}$ (see \cite{Mann22}, Lemma 2.2.7(ii)) is solid. We denote by $\Cond_{\blacksquare}(V, \mathfrak{m})$ the full subcategory of $\Cond(V, \mathfrak{m})$ spanned by the solid condensed almost $(V, \mathfrak{m})$-modules and call the category of solid almost $(V, \mathfrak{m})$-modules.\end{mydef}
We sometimes omit the word "condensed" when speaking of solid condensed almost $(V, \mathfrak{m})$-modules. The lemma below is a sanity check. 
\begin{lemma}\label{Solid implies solid}Let $(V, \mathfrak{m})$ be an almost setup. If $M$ is a solid condensed $V$-module, then $M^{a}$ is a solid condensed almost $(V, \mathfrak{m})$-module.\end{lemma}
\begin{proof}Set $N=M^{a}$. Recall from \cite{Mann22}, Lemma 2.2.7(ii), that $N_{\ast}$ is defined by \begin{equation*}N_{\ast}=\underline{\Hom}_{\Cond(V)}(\widetilde{\mathfrak{m}}, M),\end{equation*}where $\widetilde{\mathfrak{m}}=\mathfrak{m}\otimes_{V}\mathfrak{m}$. Since $\widetilde{\mathfrak{m}}$ is endowed with the discrete topology, it is solid. But the internal Hom (in the category of condensed $V$-modules) of two solid condensed $V$-modules is solid, see, for example, \cite{Kedlaya24}, Proposition 6.2.4. \end{proof}
\begin{cor}\label{Another notion of almost solid}Let $(V, \mathfrak{m})$ be an almost setup. If $M$ is a condensed $V$-module such that the canonical map $M\to M_{\blacksquare}$ to the solidification of $M$ is an almost isomorphism, then $M^{a}$ is a solid almost $(V, \mathfrak{m})$-module.\end{cor}
\begin{example}For a condensed $V$-module $M$, the corresponding condensed almost $(V, \mathfrak{m})$-module $M^{a}$ can be almost solid without $M$ itself being solid. For example, consider any solid condensed $V$-module $N$ and let \begin{equation*}M=(\underline{\mathbb{R}}\otimes_{\underline{\mathbb{Z}}}V/\mathfrak{m})\oplus N,\end{equation*}where $-\otimes_{\underline{\mathbb{Z}}}-$ is the tensor product of condensed abelian groups. Since the solidification of $\underline{\mathbb{R}}$ (and of any condensed module over $\underline{\mathbb{R}}$) is zero (\cite{Kedlaya24}, Corollary 6.3.2), the solidification of $M_{\blacksquare}$ of $M$ is $N$. On the other hand, the map $M\to M_{\blacksquare}$ is an almost isomorphism since the first direct summand in the definition of $M$ is almost zero. In particular, $M$ is a condensed $V$-module which is not solid, but the associated condensed almost $(V, \mathfrak{m})$-module $M^{a}$ is solid in the sense of Definition \ref{Solid almost modules}.\end{example}    
\begin{mydef}[Solid tensor product of solid almost modules]\label{Solid tensor product}Let $(V, \mathfrak{m})$ be an almost setup. We endow the category $\Cond_{\blacksquare}(V, \mathfrak{m})$ of solid almost $(V, \mathfrak{m})$-modules with a symmetric monoidal structure given by \begin{equation*}M\otimes_{V^{a}}^{\blacksquare}N:=(M_{\ast}\otimes_{V}^{\blacksquare}N_{\ast})^{a},\end{equation*}where $\otimes_{V}^{\blacksquare}$ denotes the solid tensor product of solid condensed $V$-modules.\end{mydef}   
We can now finally prove Theorem \ref{Relation to condensed mathematics} which was announced in the introduction (see \cite{Diamonds}, \S7, for the definition and basic properties of totally disconnected perfectoid spaces).
\begin{thm}\label{Relation to condensed mathematics, precise version}Let $(A, \lVert\cdot\rVert)$ be a Banach ring with a norm-multiplicative topologically nilpotent unit $\varpi$ of norm $\leq1$ and suppose that $\varpi$ admits a compatible system $(\varpi^{1/p^{n}})_{n}$ of $p$-power roots in $A$ satisfying $\lVert\varpi^{1/p^{n}}\rVert=\lVert\varpi\rVert^{1/p^{n}}$ for all $n\geq1$. Then the closed unit ball functor \begin{equation*}M\mapsto M_{\leq1}^{a}\end{equation*}induces a categorical embedding\begin{equation*}\Ban_{A}^{\leq1}\hookrightarrow \Cond_{\blacksquare}(A_{\leq1}, (\varpi^{1/p^{\infty}})_{A_{\leq1}}).\end{equation*} 

Moreover, if $A$ is a perfectoid Tate ring endowed with a power-multiplicative norm defining its topology (so that $A_{\leq1}=A^{\circ}$, by \cite{Dine22}, Lemma 2.24) and if the affinoid perfectoid space $\Spa(A, A^{\circ})$ is totally disconnected, then the above embedding is symmetric monoidal when $\Ban_{A}^{\leq1}$ is endowed with the complete tensor product $-\widehat{\otimes}_{A}-$ and $\Cond_{\blacksquare}(A^{\circ}, (\varpi^{1/p^{\infty}})_{A^{\circ}})$ is endowed with solid tensor product $-\otimes_{A^{\circ a}}^{\blacksquare}-$ from Definition \ref{Solid tensor product}.\end{thm}
We first prove a series of lemmas.
\begin{lemma}\label{Torsion-free}Let $\varpi\in A_{0}$ be a non-zero-divisor in a ring $A_{0}$ and let $M_{0}$ be a $\varpi$-torsion-free $A_{0}$-module. For every $n\geq1$, the kernel of the multiplication-by-$\varpi$ map on $M_{0}/\varpi^{n+1}M_{0}$ is equal to $\varpi^{n}(M_{0}/\varpi^{n+1}M_{0})$.\end{lemma}
\begin{proof}Let $\overline{x}\in M_{0}/\varpi^{n+1}M_{0}$ with $\varpi\overline{x}=0$. Then $\overline{x}$ lifts to an element $x\in M_{0}$ with $\varpi x\in \varpi^{n+1}M_{0}$. Since $M_{0}$ is $\varpi$-torsion-free, we conclude that $x\in \varpi^{n}M_{0}$ and $\overline{x}\in \varpi^{n}(M_{0}/\varpi^{n+1}M_{0})$. \end{proof}
\begin{lemma}\label{Torsion-free 2}Let $\varpi\in A_{0}$ be a non-zero-divisor in a ring $A_{0}$ and let $N_{0}$ be an $A_{0}$-module such that for every $n\geq1$ the $A_{0}/\varpi^{n}A_{0}$-module $N_{0}/\varpi^{n}N_{0}$ is flat. For any $\varpi$-torsion-free $A_{0}$-module $M_{0}$, the $\varpi$-adically completed tensor product $M_{0}\widehat{\otimes}_{A_{0}}B_{0}$ is $\varpi$-torsion-free.\end{lemma}
\begin{proof}Let $x\in M_{0}\widehat{\otimes}_{A_{0}}N_{0}$ be an element annihilated by $\varpi$. If $x$ is not zero, then, by completeness, we can choose an integer $n>1$ such that \begin{equation*}x\not\in \varpi^{n}(M_{0}\widehat{\otimes}_{A_{0}}N_{0}).\end{equation*}Then, a fortiori, the image $\overline{x}$ of $x$ in \begin{equation*}(M_{0}\widehat{\otimes}_{A_{0}}N_{0})/\varpi^{n+1}(M_{0}\widehat{\otimes}_{A_{0}}N_{0})=M_{0}/\varpi^{n+1}M_{0}\otimes_{A_{0}/\varpi^{n+1}A_{0}}N_{0}/\varpi^{n+1}N_{0}\end{equation*}is a non-zero element annihilated by $\varpi$. By Lemma \ref{Torsion-free}, we have \begin{equation*}\ker(\times\varpi: M_{0}/\varpi^{n+1}M_{0}\to M_{0}/\varpi^{n+1}M_{0})=\varpi^{n}(M_{0}/\varpi^{n+1}M_{0}),\end{equation*}where $\times\varpi$ denotes the multiplication-by-$\varpi$ map. Since the $A_{0}/\varpi^{n}A_{0}$-module $N_{0}/\varpi^{n+1}N_{0}$ is flat, this entails that \begin{align*}\ker(\times\varpi: M_{0}/\varpi^{n+1}M_{0}\otimes_{A_{0}/\varpi^{n+1}A_{0}}N_{0}/\varpi^{n+1}N_{0}\to M_{0}/\varpi^{n+1}M_{0}\otimes_{A_{0}/\varpi^{n+1}A_{0}}N_{0}/\varpi^{n+1}N_{0})\\=\varpi^{n}(M_{0}/\varpi^{n+1}M_{0}\otimes_{A_{0}/\varpi^{n+1}A_{0}}N_{0}/\varpi^{n+1}N_{0}),\end{align*}so \begin{equation*}\overline{x}\in \varpi^{n}((M_{0}\widehat{\otimes}_{A_{0}}N_{0})/\varpi^{n+1}(M_{0}\widehat{\otimes}_{A_{0}}N_{0})).\end{equation*}But then $x\in \varpi^{n}(M_{0}\widehat{\otimes}_{A_{0}}N_{0})$, a contradiction.\end{proof}
The following lemma is a natural generalization of the first assertion of \cite{Diamonds}, Proposition 7.23.
\begin{lemma}\label{Flatness over totally disconnecteds}Let $(A, A^{+})$ be a perfectoid Tate Huber pair such that the affinoid perfectoid space $X=\Spa(A, A^{+})$ is totally disconnected. Let $\varpi\in A$ be a topologically nilpotent unit and equip $A$ with a power-multiplicative norm defining its topology and making the element $\varpi$ norm-multiplicative (this is always possible by Lemma \ref{Uniform Tate rings}). Then, for every submetric Banach $A$-module $M$, the $A^{+}/\varpi A^{+}$-module $M_{\leq1}/\varpi M_{\leq1}$ is flat.\end{lemma}
\begin{proof}Let $M$ be a submetric Banach $A$-module. For every $x\in X$, the connected component $C_{x}$ of $X$ containing $x$ is of the form $C_{x}=\Spa(K_{x}, K_{x}^{+})$ for some nonarchimedean field $K_{x}$ with open valuation subring $K_{x}^{+}\subset K_{x}$ is an open and bounded valuation subring. Write $C_{x}$ as an intersection of a family $(U_{x,i})_{i\in I_{x}}$ of closed-open subsets of $X$. By Shilov's Idempotent Theorem for Tate Huber pairs (\cite{Kedlaya-Liu}, Proposition 2.6.4), there exist idempotent elements $e_{x,i}\in A$ such that \begin{equation*}U_{x,i}=\{\, y\in X\mid \vert e_{x,i}(y)\vert=1\,\}\simeq\Spa(e_{x,i}A, e_{x,i}A^{+})\end{equation*}for all $x\in X$, $i\in I_{x}$, so $K_{x}^{+}$ can be written as a $\varpi$-adically completed direct limit \begin{equation*}K_{x}^{+}=(\varinjlim_{i\in I_{x}}e_{x,i}A^{+})^{\wedge}.\end{equation*}In particular, \begin{equation*}K^{+}/\varpi^{n}K_{x}^{+}=\varinjlim_{i\in I_{x}}e_{x,i}A^{+}/\varpi^{n}e_{x,i}A^{+}\end{equation*}for all $n\geq1$. Since each $e_{x,i}A^{+}/\varpi e_{x,i}A^{+}$ is flat over $A^{+}/\varpi^{n}A^{+}$ (in fact, it is a finite projective $A^{+}/\varpi^{n}A^{+}$-module), we see that $K_{x}^{+}/\varpi^{n}K_{x}^{+}$ is flat over $A^{+}/\varpi^{n}A^{+}$ for all $n\geq1$ and all $x\in X$. By Lemma \ref{Torsion-free 2}, this entails that $M_{\leq1}\widehat{\otimes}_{A^{+}}K_{x}^{+}$ is $\varpi$-torsion-free and hence flat over the valuation ring $K_{x}^{+}$. A fortiori, $M_{\leq1}/\varpi^{n}M_{\leq1}\otimes_{A^{+}/\varpi^{n}A^{+}}K_{x}^{+}/\varpi^{n}K_{x}^{+}$ is flat over $K_{x}^{+}/\varpi^{n}K_{x}^{+}$. 

Now, let $M_{1}\hookrightarrow M_{2}$ be an injection of $A^{+}/\varpi A^{+}$-modules. We have to prove that \begin{equation*}M_{1}\otimes_{A^{+}/\varpi A^{+}}M_{\leq1}/\varpi M_{\leq1}\to M_{2}\otimes_{A^{+}/\varpi A^{+}}M_{\leq1}/\varpi M_{\leq1}\end{equation*}is injective. Let $f$ be an element in the kernel of that map. By the flatness of \begin{equation*}M_{\leq1}/\varpi^{n}M_{\leq1}\otimes_{A^{+}/\varpi^{n}A^{+}}K_{x}^{+}/\varpi^{n}K_{x}^{+}\end{equation*}established in the previous paragraph, we know that the image of $f$ in $M_{1}\otimes_{A^{+}/\varpi A^{+}}M_{\leq1}/\varpi M_{\leq1}\otimes_{A^{+}/\varpi A^{+}}K_{x}^{+}/\varpi K_{x}^{+}$ is zero for every $x\in X$. Therefore, for every $x\in X$, there exists an index $i_{x}\in I_{x}$ such that the image of $f$ is already zero in $M_{1}\otimes_{A^{+}/\varpi A^{+}}M_{\leq1}/\varpi M_{\leq1}\otimes_{A^{+}/\varpi A^{+}}e_{x,i_{x}}A^{+}/\varpi e_{x,i_{x}}A^{+}$. Hence, to verify that $f$ is zero, it suffices to prove that the open subsets \begin{equation*}D(e_{x,i_{x}})=\Spec(e_{x,i_{x}}A^{+}/\varpi), x\in X,\end{equation*}cover $\Spec(A^{+}/\varpi)$. But, by \cite{BhattNotes}, Theorem 8.12, the specialization map \begin{equation*}\spc: X\to \Spec(A^{+}/\varpi)\end{equation*}defined by \begin{equation*}\spc(x)=\{\, f\in A^{+}\mid \vert f(x)\vert<1\,\}/\varpi\end{equation*}is surjective, so the assertion follows from the fact that $\vert e_{x,i_{x}}(x)\vert=1$ for every $x\in X$. \end{proof}
\begin{lemma}\label{Adic embedding 2}Let $(V, \mathfrak{m})$ be an almost setup. The $\varpi$-adic embedding of Definition \ref{Adic embedding, definition} fits into a commutative square \begin{center}\begin{tikzcd}\Cond(V, \mathfrak{m})\arrow{r}{M\mapsto M_{\ast}} & \Cond(V) \\ \Mod_{V^{a}}^{\wedge,\varpi-\tf} \arrow[hook]{u} \arrow{r}{M\mapsto M_{\ast}} & \Mod_{V}^{\wedge,\varpi-\tf}, \arrow[hook]{u}\end{tikzcd}\end{center}where the lower and upper horizontal arrows are the right adjoints of the two localization functors as defined by Gabber-Ramero in \cite{Gabber-Ramero}, Proposition 2.2.14, and by Mann in \cite{Mann22}, Lemma 2.2.7(ii).\end{lemma}
\begin{proof}By Proposition \ref{Adic embedding}, for any $M_{0}\in\Mod_{V}^{\wedge,\varpi-\tf}$, we have to prove that \begin{equation*}\underline{\Hom_{V}(\widetilde{\mathfrak{m}}, M_{0})}=\underline{\Hom}_{\Cond(V)}(\widetilde{\mathfrak{m}}, \underline{M_{0}})\end{equation*}in $\Cond(V)$. Recall that $\widetilde{\mathfrak{m}}$ carries the discrete topology; in particular, it is a compactly generated Hausdorff topological abelian group and $\Hom_{V}(\widetilde{\mathfrak{m}}, M_{0})$ is qual to the Hom in the category of topological abelian groups when $M_{0}$ is endowed with the $\varpi$-adic topology. Also, $M_{0}$ is $\varpi$-adically complete and, in particular, Hausdorff. Hence the assumptions of \cite{Condensed}, Proposition 4.2, are satisfied and the desired assertion follows from loc.~cit.\end{proof} 
\begin{proof}[Proof of Theorem \ref{Relation to condensed mathematics, precise version}]We obtain a categorical embedding \begin{equation*}\Ban_{A}^{\leq1}\hookrightarrow \Cond(A_{\leq1}, (\varpi^{1/p^{\infty}})_{A_{\leq1}})\end{equation*}by composing the equivalence \begin{equation*}\Ban_{A}^{\leq1}\tilde{\rightarrow}\Mod_{A_{\leq1}^{a}}^{\wedge,\varpi-\tf}, M\mapsto M_{\leq1}^{a},\end{equation*}from Theorem \ref{Seminormed modules and almost lattices} with the $\varpi$-adic embedding \begin{equation*}\Mod_{A_{\leq1}^{a}}^{\wedge,\varpi-\tf}\hookrightarrow \Cond(A_{\leq1}, (\varpi^{1/p^{\infty}})_{A_{\leq1}})\end{equation*}of Definition \ref{Adic embedding, definition}. By Proposition \ref{Adic embedding}, any object in the image of the embedding is of the form $N^{a}$, where $N$ is the condensed $A_{\leq1}$-module associated with some $\varpi$-adically complete $A_{\leq1}$-module (with its $\varpi$-adic topology). Since the condensed module associated with a $\varpi$-adically complete $V$-module, with the $\varpi$-adic topology, is solid, we conclude that any object in the image of our embedding is a solid almost $(A_{\leq1}, (\varpi^{1/p^{\infty}})_{A_{\leq1}})$-module in the sense of Definition \ref{Solid almost modules}, by Lemma \ref{Solid implies solid}. In other words, the above embedding $\Ban_{A}^{\leq1}\hookrightarrow\Cond(A_{\leq1}, (\varpi^{1/p^{\infty}})_{A_{\leq1}})$ factors through an embedding \begin{equation*}\Ban_{A}^{\leq1}\hookrightarrow \Cond_{\blacksquare}(A_{\leq1}, (\varpi^{1/p^{\infty}})_{A_{\leq1}}).\end{equation*}

It remains to prove that this embedding is compatible with the complete and solid tensor products when $A$ is perfectoid and $\Spa(A, A^{\circ})$ is totally disconnected. So, assume that $A$ (endowed with a power-multiplicative norm defining its topology), that $\varpi$ is a norm-multiplicative topologically nilpotent unit of $A$ which admits a compatible system of $p$-power roots in $A$ and that the adic spectrum $\Spa(A, A^{\circ})$ is totally disconnected. Let $M$, $N$ be two submetric Banach $A$-modules. By Lemma \ref{Flatness over totally disconnecteds}, $N_{\leq1}/\varpi^{n}N_{\leq1}$ is flat over $A^{\circ}/\varpi^{n}A^{\circ}$ for every $n\geq1$. By Lemma \ref{Torsion-free 2}, this implies that $M_{\leq1}\widehat{\otimes}_{A^{\circ}}N_{\leq1}$ is $\varpi$-torsion-free. It then follows from Corollary \ref{Unit balls and tensor products 2} that \begin{equation*}(M\widehat{\otimes}_{A}N)_{\leq1}^{a}=(M_{\leq1}\widehat{\otimes}_{A^{\circ}}N_{\leq1})^{a}.\end{equation*}By standard properties of solid condensed modules and the solid tensor product, the functor \begin{equation*}\Mod_{V}^{\wedge}\to \Cond(V), M_{0}\mapsto \underline{M_{0}},\end{equation*}which takes a $\varpi$-adically complete $A^{\circ}$-module, viewed as a topological module with the $\varpi$-adic topology, to the corresponding condensed $A^{\circ}$-module, factors through the category of solid condensed $A^{\circ}$-modules $\Cond_{\blacksquare}(A^{\circ})$ and takes the $\varpi$-adically completed tensor product to the solid tensor product. In particular, \begin{equation*}\underline{M_{\leq1}\widehat{\otimes}_{A^{\circ}}N_{\leq1}}=\underline{M_{\leq1}}\otimes_{A^{\circ}}^{\blacksquare}\underline{N_{\leq1}}.\end{equation*}Then the commutative square in Proposition \ref{Adic embedding} shows that \begin{equation}\underline{(M_{\leq1}\widehat{\otimes}_{A^{\circ}}N_{\leq1})^{a}}=(\underline{M_{\leq1}}\otimes_{A^{\circ}}^{\blacksquare}\underline{N_{\leq1}})^{a}.\end{equation}Note that $(M_{\leq1}^{a})_{\ast}=M_{\leq1}$, $(N_{\leq1}^{a})_{\ast}=N_{\leq1}$. By Lemma \ref{Adic embedding 2}, this entails \begin{equation*}(\underline{M_{\leq1}}^{a})_{\ast}=\underline{M_{\leq1}}\end{equation*}and \begin{equation*}(\underline{N_{\leq1}}^{a})_{\ast}=N_{\leq1}.\end{equation*}It follows that the right hand side of equation (2) coincides with the solid tensor product of the solid almost $(A^{\circ}, (\varpi^{1/p^{\infty}})_{A^{\circ}})$-modules $\underline{M_{\leq1}^{a}}$ and $\underline{N_{\leq1}^{a}}$. This concludes the proof of the theorem.\end{proof}

\bibliographystyle{plain} 
\bibliography{Bib}

\textsc{Department of Mathematics, University of California San Diego, La Jolla, CA 92093, United States} \newline 

E-mail address: \textsf{ddine@ucsd.edu}

\end{document}